\pdfoutput=1
\documentclass[12pt, twoside, leqno]{article}
\usepackage{libertine}
\usepackage{xspace}

\providecommand\NOINDEX[1]{}
\providecommand\ASLSTYLE[1]{}
\usepackage[anglais, surreels]{perso}
\usepackage{fancybox}
\newcommand{\omegaunCK}{\omega_1^{\textrm{CK}}}

\usepackage{thm-restate}

\usepackage{titling}
\usepackage{stackengine}
\setstackgap{L}{.8\baselineskip}

\title{Surreal fields stable under exponential, logarithmic, derivative and anti-derivative functions}
\usepackage{authblk}
\author{\stackunder{Olivier Bournez}{\small olivier.bournez@lix.polytechnique.fr}}
\author{\stackunder{Quentin Guilmant}{\small quentin.guilmant@lix.polytechnique.fr}}
\affil{École Polytechnique, LIX, 91128 Palaiseau Cedex, France\newline
\small{This work was partially supported by ANR Project $\partial$IFFERENCE.}}
\date{}
\newtheorem{Blabla}{Blabla}[section]
\newtheorem{theorem}[Blabla]{Theorem}

\newtheorem{proposition}[theorem]{Proposition}
\newtheorem{lemma}[theorem]{Lemma}
\newtheorem{corollary}[theorem]{Corollary}
\theoremstyle{definition}
\newtheorem{definition}[theorem]{Definition}

\newtheorem{remark}[theorem]{Remark}

\begin{document}
	\maketitle

\begin{abstract}
	The class of surreal numbers, denoted by $\Nobf$, initially proposed by Conway, is a universal ordered field in the sense that any ordered field can be embedded in it. They include in particular the real numbers and the ordinal numbers. They have strong relations with other fields such as field of transseries.  Following Gonshor, surreal numbers can be seen as signs sequences of ordinal length, with some exponential and logarithmic functions that extend the usual functions over the reals.  $\Nobf$   can actually be seen as an elegant (generalized) power series field with real coefficients, namely Hahn series with exponents in $\Nobf$ itself. 

	Some years ago, Berarducci and Mantova considered derivation over the surreal numbers, seeing them as germs of functions, in correspondence to transseries.
	In this article, following our previous work, we exhibit a sufficient condition on the structure of a surreal field to be stable under all operations among exponential, logarithm, derivation and anti-derivation. Motivated, in the long term, by computability considerations, we also provide a non-trivial application of this theorem:
	the existence of a pretty reasonable field that only requires ordinals up to $\epsilon_\omega$, which is far smaller than $\omegaunCK$ (resp. $\omega_1$), the first non-computable (resp. uncountable) ordinal. 

\end{abstract}

\section{Introduction}
Conway introduced in  \cite{conway2000numbers}  the class of surreal numbers. They were later on  popularized by Knuth \cite{knuth1974surreal}, and then formalized later on by Gonshor \cite{gonshor1986introduction}, and by many other authors.  The general initial idea is to define a class of numbers, based on a concept of ``simplicity''. This permits to obtain a real closed field that both contains the real numbers and the ordinals, as this provides an unification of Dedekind's construction of real numbers in terms of cuts of the rational numbers, and of von Neumann's construction of ordinal numbers by transfinite induction in terms of set membership.

Following the alternative presentation from Gonshor in \cite{gonshor1986introduction}, a surreal number can also be seen as an ordinal-length sequence over $\{+,-\}$, that we call a  \textbf{signs sequence}. Basically, the idea is that such sequences are ordered lexicographically, and have a tree-like structure. Namely, a $+$ (respectively $-$) added to a sequence $x$ denotes the simplest number greater (resp. smaller) than $x$ but smaller (resp. greater) than all the prefixes of $x$ which are greater (resp. smaller) than $x$.   With this definition of surreal numbers, it is possible to define operations such as addition, substraction, multiplication, division, obtaining a real closed field.  
%
Following Gonshor \cite{gonshor1986introduction}, based on ideas from Kruskal, it is also possible to define consistently classical functions such as the exponential function and the  logarithmic function over $\Nobf$, and to do analysis of this fields of numbers. 

It can be considered as ``the'' field that includes ``all numbers great and small'' \cite{ehrlich2012absolute}.  In particular, any divisible ordered Abelian group is isomorphic to an initial subgroup of $\Nobf$, and any real closed field is isomorphic to an initial subfield of $\Nobf$ \cite[Theorems 9 and 19]{ehrlich2001number}, \cite[Theorems 28 and 29]{conway2000numbers}. This leads to the fact that it can be considered as ``the'' field that includes ``all numbers great and small'' \cite{ehrlich2012absolute}.  
%
$\Nobf$ can also be equipped with a derivation, so that it can be considered as a fields of transseries \cite{berarducci2018surreal}. See example \cite{mantova2017surreal} for a survey of fascinating recent results in all these directions.  
%

More concretely, $\Nobf$ can also be seen as a field of (generalized) power series with real coefficients, namely as Hahn series where exponents are surreal numbers themselves. More precisely, write $\HahnField{\Kbb}{G}$ for the set of Hahn series with coefficients in $\Kbb$ and terms corresponding to elements of $G$, where $\Kbb$ is a field, and $G$ is some divisible ordered Abelian group: This means that $\HahnField{\Kbb}{G}$ corresponds to formal power series of the form $s=\sum_{g \in S} a_{g} t^{g}$, where $S$ is a well-ordered subset of $G$ and $a_{g} \in \Kbb .$ The support of $s$ is $\supp(s)=\enstq{g \in S}{a_{g} \neq 0}$ and the length of the serie of $s$ is the order type of $\supp(s)$. The field operations on $\HahnField{\Kbb}{G}$ are defined as expected, considering elements of $\HahnField{\Kbb}{G}$ as formal power series.
We have $\Nobf= \HahnField{\Rbb}{\Nobf}$.

Our previous work \cite{bournez2022surreal} has shown that some acceptable subfields of $\Nobf$ are stable by both exponential and logarithm. This fields are built with some restriction on ordinals allowed in the ordinal sum.  In the current article, we pursue the work by exhibiting some acceptable subfields of $\Nobf$ are stable by both exponential and logarithm, derivation and anti-derivation. We actually provide some sufficient condition on the structure of a surreal field to be stable under all these operations among exponential, logarithm, derivation and anti-derivation, and use these to derive such subfields.

\paragraph{More precise statements}
Given some ordinal $\gamma$ (or more generally a class of ordinals), we write $\HahnFieldOrd{\Kbb}{G}{\gamma}$ for the restriction of $\HahnField{K}{G}$ to formal power series whose support has an order type in $\gamma$ (that is to say, corresponds to some ordinal less than $\gamma$). We  have of course $\Nobf= \HahnFieldOrd{\Rbb}{\Nobf}{\Ord}$.
In this point of view, \textbf{$\epsilon$-numbers}, \ie{} ordinals $\lambda$, such that
$\omega^\lambda=\lambda$, play a major role as they are such limit on ordinal we can accept in our fields.

As we will often play with exponents of formal power series considered in the Hahn series, we propose to introduce the following notation:  We denote $$\SRF{\lambda}{\Gamma}=\HahnFieldOrd\Rbb{\Gamma}\lambda$$ when  $\lambda$ is an $\epsilon$-number and $\Gamma$ a divisible Abelian group.

As a consequence of MacLane's theorem (Theorem \ref{thm:macLane} below from \cite{maclane1939}, see also \cite[section 6.23]{alling1987foundations}), we know that $\SRF\lambda{\Nolt\mu}$ is a real-closed field when $\mu$ is a \textbf{multiplicative ordinal} (i.e. $\mu=\omega^{\omega^\alpha}$ for some ordinal $\alpha$) and $\lambda$ an $\epsilon$-number.

Furthermore:

\begin{restatable}[{\cite[Proposition 4.7]{DriesEhrlich01}}]{theorem}{thEhrlichquatresept}
	\label{th:Ehrlichquatresept}
	Let $\lambda$ be an $\epsilon$-number. Then
	\begin{enumerate}
		\item The field $\Nolambda$ can be expressed as \begin{equation} \label{eq:troissix} 
			\Nolambda=\bigcup_{\mu} \SRF{\lambda}{\Nolt\mu},
		\end{equation}
		where $\mu$ ranges over the additive ordinals less than $\lambda$ (equivalently, $\mu$ ranges over the multiplicative ordinals less than $\lambda$ ).
		\item  $\Nolambda$ is a real closed subfield of $\No$, and is closed under the restricted analytic functions of $\No$.
		\item $\Nolambda=\SRF{\lambda}{\Nolt\lambda}$ if and only if $\lambda$ is a regular cardinal.
	\end{enumerate}
\end{restatable}


Actually, even if we always can write $\Nolambda$ as an increasing union of fields by Equation \eqref{eq:troissix},  and even if $\Nolt\lambda$ is stable under exponential and logarithmic functions (Theorem \ref{thm:NoltStableExpLn}) none of the fields in this union has stability property beyond the fact that they are fields. Indeed:

\begin{proposition}[{\cite[Proposition 1.5]{bournez2022surreal}}]
	\label{prop:instable}
	$\SRF\lambda{\Nolt\mu}$ is never closed under exponential function for $\mu<\lambda$ a multiplicative ordinal. 
\end{proposition}

In our previous work \cite{bournez2022surreal}, we studied stability of subfields of $\Nobf$ by exponential and logarithm. 
For the sake of effectiveness and representations for ordinal Turing machines, we kept \textbf{ordinals} as small as possible to identify natural subfields stable by both these functions. This paper will keep the same spirit and we will give an example construction that only involves ordinals up to $\epsilon_\omega$, which is much smaller than $\omega_1$, the first \textbf{uncountable} ordinal, and even $\omegaunCK$, the first \textbf{non-computable} ordinal.

To achieve that purpose, we will have to handle carefully the $\epsilon$-numbers that are involved. Recall that there is some enumeration $(\epsilon_{\alpha})_{\alpha \in \Ord}$of $\epsilon$-numbers: 
Any $\epsilon$-number ordinal $\lambda$ is $\epsilon_\alpha$ for some ordinal $\alpha$.

\begin{definition}[Canonical sequence defining an $\epsilon$-number]
	\label{def:gammaLambda}
	Let $\lambda$ be an $\epsilon$-number.  Ordinal $\lambda$ can always  be written as $\lambda=\sup\suitelt{e_\beta}\beta{\gamma_\lambda}$ for some \textbf{canonical sequence}, where $\gamma_\lambda$ is the length of  this sequence, and this sequence is defined as follows:
	\begin{itemize}
		\item If $\lambda=\epsilon_{0}$ then we can write $\epsilon_0=\sup\{\omega,\omega^\omega,\omega^{\omega^\omega},\dots\}$ and  we take\linebreak $\omega,\omega^\omega,\omega^{\omega^\omega},\dots$ as canonical sequence for $\epsilon_0$. Its length is $\omega$, and for $\beta<\lambda$, $e_\beta$ is $\omega^{\iddots^\omega}$ where there are $\beta$ occurrences of $\omega$ in the exponent.
		
		\item If $\lambda=\epsilon_{\alpha}$, where $\alpha$ is a non-zero limit ordinal, then we can write $\lambda=\underset{\beta<\alpha}\sup\epsilon_\beta$ and we take $\suitelt{\epsilon_\beta}\beta\alpha$ as the canonical sequence of $\lambda$. Its length is $\alpha$ and for $\beta<\alpha$, $e_\beta=\epsilon_\beta$.
		
		\item If $\lambda=\epsilon_{\alpha}$, where  $\alpha$ is a successor ordinal, then we can write 
		$$\lambda=\sup\{\epsilon_{\alpha-1}, {\epsilon_{\alpha-1}}^{\epsilon_{\alpha-1}}, {\epsilon_{\alpha-1}}^{{\epsilon_{\alpha-1}}^{\epsilon_{\alpha-1}}},\dots\}$$
		and we take $\epsilon_{\alpha-1}, {\epsilon_{\alpha-1}}^{\epsilon_{\alpha-1}}, {\epsilon_{\alpha-1}}^{{\epsilon_{\alpha-1}}^{\epsilon_{\alpha-1}}},\dots$ as the canonical sequence of $\lambda$. Its length is $\omega$, and for $\beta<\omega$, $e_\beta={\epsilon_{\alpha-1}}^{\iddots^{\epsilon_{\alpha-1}}}$ where there are $\beta$ occurrences of $\epsilon_{\alpha-1}$ in the exponent.
	\end{itemize}
\end{definition}

For example, the canonical sequence defining $\epsilon_1$ is $\epsilon_0,{\epsilon_0}^{\epsilon_0},{\epsilon_0}^{{\epsilon_0}^{\epsilon_0}},\dots$, the canonical sequence defining $\epsilon_\omega$ is $\epsilon_0,\epsilon_1,\epsilon_2,\dots$, 
the canonical sequence of $\epsilon_{\omega2}$ is 
$\epsilon_0,\epsilon_1,\epsilon_2,\dots,\epsilon_\omega,\epsilon_{\omega+1},\dots$ and the canonical sequence of $\epsilon_{\omega2+1}$ is \linebreak ${\epsilon_{\omega2}},{\epsilon_{\omega2}}^{\epsilon_{\omega2}}, {\epsilon_{\omega2}}^{{\epsilon_{\omega2}}^{\epsilon_{\omega2}}},\dots$

\begin{definition}
	\label{def:uparrow}
	Let $\Gamma$ be an Abelian subgroup of \Nobf{} and $\lambda$ be an $\epsilon$-number whose canonical sequence is $\suitelt{e_\beta}\beta{\gamma_\lambda}$. 
	We denote $\Gamma^{\uparrow\lambda}$ for the family of group $\suitelt{\Gamma_\beta}\beta{\gamma_\lambda}$ defined as follows:
	\begin{itemize}
		\item $\Gamma_0=\Gamma$;
		
		\item $\Gamma_{\beta+1}$ is the group generated by the groups $\Gamma_\beta$, $\SRF{e_\beta}{g\pa{(\Gamma_\beta)^*_+}}$ and the set $\enstq{h(a_i)}{\aSurreal\in\Gamma_\beta}$
		where $g$ and $h$ are Gonshor's functions associated to exponential and logarithm (see Section \ref{sec:expoLn} below for some details);
		
		\item For a limit ordinal number $\beta$, $\Gamma_\beta=\Unionlt\gamma\beta \Gamma_\gamma$.
	\end{itemize}
\end{definition}

When considering a family of set $(S_i)_{i\in I}$, we denote
\centre{$
\SRF\lambda{(S_i)_{i\in I}} = \Unionin iI \SRF\lambda{S_i}
$}
\lc{In particular,}
{$
\SRF\lambda{\Gamma^{\uparrow\lambda}} = \Unionlt i{\gamma_\lambda}\SRF\lambda{\Gamma_i}
$}
\begin{remark}\label{rk:inclusionUparrow}
	By construction, if $\Gamma\subseteq\Gamma'$ then $\SRF\lambda{\Gamma^{\uparrow\lambda}} \subseteq \SRF\lambda{{\Gamma'}^{\uparrow\lambda}}$.
\end{remark}
The idea behind the definition of $\Gamma^{\uparrow}$ is that at step $i+1$ we add new elements to close $\SRF\lambda{\Gamma_i}$ under exponential and logarithm. The reason why we add $\SRF{e_\beta}{g\pa{(\Gamma_\beta)^*_+}}$ to $\Gamma_\beta$ rather than $\SRF\lambda{g\pa{(\Gamma_\beta)^*_+}}$ is that we want to keep control on what we add in the new group. In our previous work \cite{bournez2022surreal} we came up with the following three statements:

\begin{lemma}[{\cite[Lemma 5.7]{bournez2022surreal}}]
	\label{lem:LogAtomicDesGammai}
	Write $\Gamma^{\uparrow\lambda}=\suitelt{\Gamma_\beta}\beta{\gamma_\lambda}$, and let 
	$$L=\enstq{\exp_n x, \ln_nx}{\begin{array}{c}
			x\in\Lbb,\ n\in\Nbb,\\
			\exists y\in\RlG\ \exists P\in\Pcal(y)\ \exists k\in\Nbb\quad P(k)=x
		\end{array}}$$
	we have for all $i<\gamma_\lambda$,
	$$L=\enstq{\exp_n x, \ln_nx}{\begin{array}{c}
		x\in\Lbb, n\in\Nbb,\\ 
		\exists y\in\SRF\lambda{\Gamma_i}\ \exists P\in\Pcal(y)\ \exists k\in\Nbb\quad P(k)=x
		\end{array}}$$
\end{lemma}

\begin{corollary}[{\cite[Corollary 5.8]{bournez2022surreal}}]
	\label{cor:LogAtomicDeGammaUp}
	Let $\Gamma$ be an Abelian additive subgroup of $\Nobf$ and
	$$L=\enstq{\exp_n x, \ln_nx}{\begin{array}{c}
			x\in\Lbb, n\in\Nbb, \\
			\exists y\in\RlG\ \exists P\in\Pcal(y)\ \exists k\in\Nbb\quad P(k)=x
		\end{array}}$$
	Then,
	$$L = \enstq{\exp_n x, \ln_nx}{\begin{array}{c}
			x\in\Lbb,\quad n\in\Nbb,\\
			\exists y\in\RlGup\ \exists P\in\Pcal(y)\ \exists k\in\Nbb\quad P(k)=x
	\end{array}}$$
\end{corollary}

\begin{restatable}[{\cite[Theorem 1.10]{bournez2022surreal}}]{theorem}{thmSRFGammaUpStableExpLn}
	\label{thm:SRFGammaUpStableExpLn}
	Let $\Gamma$ be an Abelian subgroup of \Nobf{} and $\lambda$ be an $\epsilon$-number, then
	$\SRF\lambda{\Gamma^{\uparrow\lambda}}$ is stable under exponential and logarithmic functions.
\end{restatable}

With such a notion, we managed to make a link between the two types of field involved in Theorems \ref{thm:NoltStableExpLn} and \ref{thm:SRFGammaUpStableExpLn}. More precisely, the fields $\SRF\lambda{\Gamma^{\uparrow\lambda}}$ are part of the fields $\Nolt\lambda$.

\begin{restatable}[{\cite[Theorem 1.11]{bournez2022surreal}}]{theorem}{thmNolambdaDecompCorpsStables}
	\label{thm:NolambdaDecompCorpsStables}
	$\Nolambda=\bigcup_{\mu} \SRF{\lambda}{{\Nolt{\mu}}^{\uparrow \lambda}}$, where $\mu$ ranges over the additive ordinals less than $\lambda$ (equivalently, $\mu$ ranges over the multiplicative ordinals less $\lambda$),
\end{restatable}
Notice that now,  $\Nolambda$ is  expressed as a increasing union of fields, each of them closed by $\exp$ and $\ln$. Indeed, by definition, if $\mu<\mu'$ then $\Nolt\mu\subseteq\Nolt{\mu'}$ and Remark \ref{rk:inclusionUparrow} gives $\SRF{\lambda}{{\Nolt{\mu}}^{\uparrow \lambda}} \subseteq \SRF{\lambda}{{\Nolt{\mu'}}^{\uparrow \lambda}}$. 

Finally, we proved that each field $\SRF{\lambda}{{\Nolt{\mu}}^{\uparrow \lambda}}$ is interesting for itself since none of them is $\Nolambda$. More precisely: 

\begin{restatable}[{\cite[Theorem 1.12]{bournez2022surreal}}]{theorem}{thmhierarchieUparrow}
	\label{thm:hierarchieUparrow}
	For all $\epsilon$-number $\lambda$, the hierarchy in previous theorem is strict:
	$$\SRF{\lambda}{{\Nolt{\mu}}^{\uparrow \lambda}} \subsetneq \SRF{\lambda}{{\Nolt{\mu'}}^{\uparrow \lambda}}$$	
	for all multiplicative ordinals $\mu$ and $\mu'$ such that $\omega<\mu<\mu'<\lambda$.
\end{restatable}

In this article, we will go further and investigate the case of stability under derivative and anti-derivative.

\begin{restatable}{mainthm}{thmcorpsStable}
	\label{thm:corpsStable}
	Let $\alpha$ be a limit ordinal 
	and $\suitelt{\Gamma_\beta}\beta\alpha$ be a sequence of Abelian subgroups of $\Nobf$ such that
	\begin{itemize}
		\item $\forall\beta<\alpha\quad \forall\gamma<\beta\qquad \Gamma_\gamma\subseteq\Gamma_\beta$
		
		\item $\forall\beta<\alpha\qquad \omega^{(\Gamma_\beta)^*_+}\succ^K\kappa_{-\epsilon_\beta}$
		
		\item  $\forall\beta<\alpha\quad \forall\gamma<\epsilon_\beta\qquad \kappa_{-\gamma}\in\omega^{\Gamma_\beta}$
		
		\item $\forall\beta<\alpha\quad\exists\eta_\beta<\epsilon_\beta\quad \forall x\in\omega^{\Gamma_\beta}\qquad \NR(x)<\eta_\beta$
	\end{itemize}
	Then $\Unionlt\beta\alpha \SRF{\epsilon_\beta}{\Gamma_\beta^{\uparrow\epsilon_\beta}}$ is stable under $\exp$, $\ln$, $\partial$ and anti-derivation (see section \ref{sec:deriv}).
\end{restatable}

Actually, we mainly focus on properties of the derivation suggested by Beraducci and Mantova and its anti-derivation. In particular, we establish various bounds that are useful to find fields stable under derivation and anti-derivation. We also prove the following:

\begin{restatable}{proposition}{propmajorationNuPartial}
	\label{prop:majorationNuPartial}
	For any $x\in\Nobf$, the set $\Pcal_\Lbb(x)$ is well-ordered with order type
	$\beta<\omega^{\omega^{\omega(\NR(x)+1)}}$. In particular,
	\centre{$\nu(\partial x)<\omega^{\omega^{\omega(\NR(x)+1)}}$}
\end{restatable}

In the above proposition and the above theorem, $\NR$ is the \textbf{nested truncation rank} which is defined in Definition \ref{def:nestedTruncRk} and $\nu(x)$ is the length of the series of the normal form the surreal number $x$ (see Definition \ref{def:nu}). The previous proposition is essential to control derivatives of surreal numbers and then get field stable under derivation. To handle anti-derivation, we came up with the following proposition:

\begin{restatable}{proposition}{propsupportPhi}
	\label{prop:supportPhi}
	Let $x$ be a surreal number. Let $\gamma$ be the smallest ordinal such that $\kappa_{-\gamma}\prec^K P(k_P)$ for all path $P\in\Pcal_\Lbb(x)$. Let $\lambda$ be the least $\epsilon$-number greater than $\NR(x)$ and $\gamma$. Then $\Unionin i\Nbb \supp\Phi^i(x)$ (see Definition \ref{def:phi}) is reverse well-ordered with order type less than $\omega^{\omega^{\lambda+2}}$.
\end{restatable} 

Thanks to Propositions \ref{prop:majorationNuPartial} and \ref{prop:supportPhi}, we will be able to prove our main theorem:

\paragraph{Organization of the paper}

This article is organized as follows. Section \ref{sec:toolbox} is a quick reminder of some lemmas about order types that will be useful at the end of this article. Section \ref{sec:introsurreals} recalls basics of the concepts and definitions of the theory of surreal numbers, and fixes the notations used in the rest of the paper. Section \ref{sec:erhlichandco} recalls what is known about the stability properties of various subfields of $\Nobf$ according to their signs sequence representation or Hahn series representation.  In Section \ref{sec:expoLn} we recall the definitions and properties of exponential and logarithm. In Section \ref{sec:deriv}, we recall some existing literature about log-atomic numbers and derivation and established some result about the nested truncation rank, a notion of rank related to the structure of the surreal numbers and to log-atomic numbers. Finally, in Section \ref{sec:stable}, we build surreal fields that are stable under exponentiation, logarithm, derivation and anti-derivation. We also show how this construction can lead too an example which  only uses ``small'' ordinals, which is good from a Computability Theory point of view.

\section{Order type toolbox}
\label{sec:toolbox}
In this section, we quickly take a look at some useful lemma about order type of well ordered sets. In all the following, circled operators ($\oplus,\otimes$) stand for usual operations over ordinal numbers. The usual symbols ($+,\times$) stand for natural operations, which are commutative.

Our first proposition is about the union of well ordered sets. This result is already knows but we still provide a proof since it is hard to find it in the literature.

\begin{lemma}[Folklore]
	\label{lem:ajoutDUnElementEnsBienOrd}
	Let $\Gamma$ be a totally ordered set, $A\subseteq\Gamma$ be a well-ordered subset with order type $\alpha$. Let $g\in\Gamma$. Then the
	set $A\cup \{g\}$ is well ordered with order type at most $\alpha+1$.
\end{lemma}

\begin{proof}
	We prove it by induction on $\alpha$.
	\begin{itemize}
		\item If $\alpha=0$ then $A\cup\{g\}$ has only one element, and then has order type $1=\alpha+1$.
		
		\item If $\alpha=\gamma+1$ is a successor ordinal. Let $u$ the largest element in $A$. If $u\leq g$ then $A\cup\{g\}$ has indeed order type at most $\alpha+1$. If not, then, by induction hypothesis, $\pa{A\setminus\{u\}}\cup\{g\}$ has order type at most $\gamma+1=\alpha$. Then $A\cup\{g\}=\pa{\pa{A\setminus\{u\}}\cup\{g\}}\cup\{u\}$ has order type at most $\alpha+1$.
		
		\item If $\alpha$ is a limit ordinal. If $g$ is larger than any element of $A$, then $A\cup\{g\}$ has order type $\alpha+1$. If not, let $a_0\in A$ such that $a_0\geq g$. For $a\in A$ such that $a> a_0$ set 
		\centre{$B_a=\{g\}\cup\enstq{a'\in A}{a'< a}$}
		\lc{Since $\alpha$ is limit, we have}{$A\cup \{g\} = \Union{a>a_0}{}B_a$}
		and each of the element in the union is an initial segment of $A\cup\{g\}$. 
		We also denote $\alpha_a$ the order type of the set~$\enstq{a'\in A}{a'< a}$. In particular, $\alpha_a<\alpha$. Using induction hypothesis, $B_a$ has order type at most $\alpha_a+1$. Then, since we have an increasing union of initial segments, the order type of $A\cup\{g\}$ is at most 
		\centre{$\sup\enstq{\alpha_a+1}{a>a_0}=\sup\enstq{\alpha'+1}{\alpha'<\alpha}= \alpha$} 
		since $\alpha$ is a limit ordinal.
	\end{itemize}
	We conclude thanks to the induction principle.
\end{proof}

\begin{proposition}[Union of well-ordered sets, folklore]
	\label{prop:unionEnsBienOrd}
	Let $\Gamma$ be a totally ordered set $A,B\subseteq\Gamma$ be non-empty well-ordered subsets with respective order types $\alpha$ and $\beta$. Then the
	subset $A\cup B$ is well ordered with order type at most $\alpha+\beta$.
\end{proposition}

\begin{proof} $A\cup B$ is well-ordered. Indeed, if we have an infinite decreasing	 	
	sequence of $A\cup B$, then we can extract either an infinite one for either $A$ or $B$ which is not possible. It remains to show the bound on its order type.
	We do it by induction over $\alpha$ and $\beta$.
	\begin{itemize}
		\item If $\alpha=\beta=1$, then $A\cup B$ has at most two elements. Then, its order type is at most $2=\alpha+\beta$.
		
		\item If $\alpha$ or $\beta$ is a successor ordinal. Since both cases are symmetric, we assume without loss of generality that $\beta=\gamma+1$. Let $u$ be the largest element of $B$ and $C=B\setminus\{u\}$. Then, by induction hypothesis, $A\cup C$ has order type at most $\alpha+\gamma$. Using Lemma \ref{lem:ajoutDUnElementEnsBienOrd}, we get that the order type of $A\cup B$ is at most $\alpha+\gamma+1=\alpha+\beta$.
		
		\item If $\alpha$ and $\beta$ are limit ordinal. $A$ or $B$ must be cofinal with $A\cup B$. For instance say it is $A$. For $a\in A$, let
		\centre{$A_a=\enstq{a'\in A}{a'<a}\qqandqq B_a=\enstq{b\in B}{b<a}$}
		\lc{We have}{$A\cup B = \Union{a\in A}{}A_a\cup B_a$}
		Since $A$ is cofinal with $A\cup B$, it is an increasing union of initial segments. Let $\alpha_a$ be the order type of $A_a$ and $\beta_a$ the one of $B_a$. We have $\alpha_a<\alpha$ and $\beta_a\leq\beta$. By induction hypothesis, $A_a\cup B_a$ has order type at most $\alpha_a+\beta_a$. Then $A\cup B$ has order type at most
		\centre{$\sup\enstq{\alpha_a+\beta_a}{a\in A}\leq \alpha+\beta$}
	\end{itemize}
	We conclude the proof using the induction principle.
\end{proof}

We know move to addition of well ordered subset of a group. Again this result in know but its proof is not easily findable in the literature. 

\begin{proposition}[Folklore]
	\label{prop:sommeEnsBienOrd} 
	Let $\Gamma$ be an ordered Abelian additive monoid and $A,B\subseteq\Gamma$ be non-empty well-ordered subsets with respective order types $\alpha$ and $\beta$. Then the
	subset $A+B=\enstq{a+b}{a\in A\quad B\in B}$ is well ordered with order type at most $\alpha\beta$.
\end{proposition}

\begin{proof}
	We do it by induction over $\alpha$ and $\beta$.
	\begin{itemize}
		\item If $\alpha=\beta=1$, then $A+B$ has only one element, then has order type $1=\alpha\beta$.
		
		\item If $\alpha$ or $\beta$ is not an additive ordinal\index{Ordinal number!additive ordinal}. Let say $\beta=\gamma + \delta$ with $\gamma,\delta<\beta$. We choose $\gamma,\delta$ such that $\gamma+\delta=\gamma\oplus\delta$. Let $B_1$ the initial segment of length $\gamma$ of $B$. Let $B_2=B\setminus B_1$. $B_2$ has order type $\delta$. Then, by induction hypothesis, $A+B_1$ has order type at most $\alpha\gamma$ and $A+B_2$ has order type at most $\alpha\delta$. Then, using Proposition \ref{prop:unionEnsBienOrd}, $A+B$ has order type at most $\alpha\gamma+\alpha\delta=\alpha\beta$.
		
		\item If both $\alpha$ and $\beta$ are additive ordinals. Assume $A+B$ has order type more than $\alpha\beta$. Let $a+b\in A+B$ such that the set $C$ defined by
		$$C:=\enstq{c\in A+B}{c< a+b}$$ 
		has order type $\alpha\beta$. Let 
		\centre{$A_0=\enstq{a'\in A}{a'<a}$ and $B_0=\enstq{b'\in B}{b'<b}$}
		and $\alpha_0$ and $\beta_0$ their respective order types. We have
		\centre{$C\subseteq \pa{A_0+B}\cup\pa{A+B_0}$}
		Using induction hypothesis and Proposition \ref{prop:unionEnsBienOrd}, $C$ has order type at most $\alpha_0\beta+\alpha\beta_0$. Since $\alpha_0<\alpha$ and $\beta_0<\beta$, we have $\alpha_0\beta<\alpha\beta$ and $\alpha\beta_0<\alpha\beta$. $\alpha$ and $\beta$ being additive ordinal, $\alpha\beta$ is itself an additive ordinal and then $C$ has order type less than $\alpha\beta$, what is a contradiction. Then $A+B$ has order type at most $\alpha\beta$.
	\end{itemize}
	We conclude thanks to the induction principle.
\end{proof}

In the same idea, we can take a look at a well ordered non-negative subset of an ordered group. The proof is less easy so we refer to \cite{weiermannMaximalOrderType} for the details.

\begin{proposition}[{\cite[Corollary 1]{weiermannMaximalOrderType}}]
	\label{prop:orderTypeMonoid}
	Let $\Gamma$ be an ordered Abelian group and $S\subseteq\Gamma_+$ be a well-ordered subset with order type $\alpha$. Then, $\inner S$, the monoid generated by $S$ in $\Gamma$ is itself well-ordered with order type at most $\omega^{\hat{\alpha}}$
	where, if the Cantor normal form of $\alpha$ is
	\centre{$\alpha=\Sum{i=1}{n}\omega^{\alpha_i}n_i$}
	\lc{then}{$\hat\alpha = \Sum{i=1}{n}\omega^{\alpha_i'}n_i$}
	\lc{and}{$\beta'=\begin{accolade}
			\beta+1 & \text{if $\beta$ is an $\epsilon$-number}\\
			\beta & \text{otherwise}
		\end{accolade}$}
	In particular, $\inner S$ has order type at most $\omega^{\omega\alpha}$ (commutative multiplication).
\end{proposition}

Finally, we consider finite sequences over a well ordered set.

\begin{theorem}[{\cite[Theorem 3.11]{DEJONGH1977195}} and  {\cite[Theorem 2.9]{SchmidtOrderTypes}}]
	\label{thm:borneTypeOrdreSuitesFinies}
	Let $(X,\leq)$ be a well ordered set with order type $\alpha$. Let $X^*$ be the set of finite sequences over $X$. Let $\beta$ the order type of $X^*$. We have
	\centre{$\beta \leq \begin{accolade}
			\omega^{\omega^{\alpha-1}} & \text{if }\alpha\text{ is finite}\\
			\omega^{\omega^{\alpha+1}} & \textit{if }\epsilon\leq\alpha<\epsilon+\omega\text{ for some $\epsilon$-number }\epsilon\\
			\omega^{\omega^\alpha} & \text{ otherwise}
		\end{accolade}$}
\end{theorem}

\section{Surreal numbers}
\label{sec:introsurreals}


We assume some familiarity with the ordered field of surreal numbers (refer to  \cite{conway2000numbers,gonshor1986introduction} for presentations) which we denote by $\No$. In this section we give a brief presentation of the basic definitions and results, and we fix the notations that will be used in the rest of the paper.

\subsection{Order and simplicity}

The class $\No$ of surreal numbers can be defined either by transfinite recursion, as in  \cite{conway2000numbers} or by transfinite length sequences of $+$ and $-$ as done in 
\cite{gonshor1986introduction}. We  will mostly follow \cite{gonshor1986introduction}, as well as \cite{berarducci2018surreal} for their presentation.

We introduce the class $\No = 2^{<\On}$ of all binary sequences of some ordinal length $\alpha \in \On$, where $\On$ denotes the class of the ordinals. In other words,  $\No$ corresponds to functions of the form $x : \alpha \to \{-,+\}$. The \textbf{length} (sometimes also called \textbf{birthday} in  literature) of a surreal number $x$ is the ordinal number $\alpha = \dom(x)$. We will also write $\alpha=\length{x}$ (the point of this notation is to ``count'' the number of pluses and minuses).
Note that $\No$ is not a set but a proper class, and all the relations and functions we shall define on $\No$ are going to be class-relations and class-functions, usually constructed by transfinite induction.

We say that $x \in \No$ is \textbf{simpler} than $y \in \No$, denoted $x \simpler y$, i.e., if $x$ is a strict \textbf{initial segment} (also called \textbf{prefix}) of $y$ as a binary sequence. We say that $x$ is simpler than or equal to $y$, written $x \simplereq y$, if $x \simpler y$ or $x = y$ i.e., $x$ is an initial segment of $y$. The simplicity relation is a binary tree-like partial order on $\No$, with the immediate successors of a node $x\in\No$ being the sequences $x_-$ and $x_+$ obtained by appending $-$ or $+$ at the end of the signs sequence of $x$. Observe in particular that the simplicity relation $\simpler $ is well-founded, and the empty sequence, which will play the role of the number zero, is simpler than any other surreal number. 

We can introduce a total order $<$ on $\No$ which is basically the lexicographic order over the corresponding sequences: More precisely,  we consider the order  $-<\square<+$ where $\square$ is the blank symbol. Now to compare two signs sequences,  append blank symbols to the shortest so that they have the same length. Then,  just compare them with the corresponding lexicographic order to get the total order $<$.

Given two sets $A \subseteq \No$ and $B \subseteq \No$ with $A < B$ (meaning that $a < b$ for all $a \in A$ and $b \in B$),
it is quite easy to understand why there is a simplest surreal number, denoted $\crotq AB$ such that $A<\crotq AB < B$.  However, a formal proof is long. See \cite[Theorem 2.1]{gonshor1986introduction} for details.
If $x=\crotq AB$, we say that Such a pair $\crotq{A}{B}$ is \textbf{representation} of $x$.

Every surreal number $x$ has several different representations $x = \crotq{A}{ B} = \crotq{A'}{ B'}$, for instance, if $A$ is cofinal with $A'$ and $B$ is coinitial with $B'$. In this situation, we shall say that $\crotq{A}{ B} = \crotq{A'}{ B'}$ by cofinality. On the other hand, as discussed in \cite{berarducci2018surreal}, 
it may well happen that $\crotq{A}{ B} = \crotq{A'}{ B'}$ even if $A$ is not cofinal with $A'$ or $B$ is not coinitial with $B'$. The \textbf{canonical representation} $x = \crotq{A}{ B}$ is the unique one such that $A \cup B$ is exactly the set of all surreal numbers strictly simpler than $x$. Indeed it turns out that is $A=\enstq{y\sqsubset x}{y<x}$ and $B=\enstq{y\sqsubset x}{y>x}$, then $x=\crotq AB$.

\begin{remark} 
By definition, if $x = \crotq{A}{ B}$ and $A < y < B$, then $x \simplereq y$.
\end{remark}
To make the reading easier we may forget $\{\}$ when writing explicitly $A$ and $B$. For instance $\crotq xy$ will often stand for $\crotq{\{x\}}{\{y\}}$ when $x,y\in\No$.


\subsection{Field operations}


Ring operations $+$, $·$ on $\No$ are defined by transfinite induction on simplicity as follows:
$$x+y:=\crotq{x' +y, x+y'}{x'' +y, x+y''}$$
$$
xy := \crotq{\begin{array}{c}
		x'y + xy' - x'y' \\ x''y + xy'' - x''y''
	\end{array}}{\begin{array}{c}
		x'y + xy'' - x'y'' \\ x''y + xy' - x''y'
	\end{array}}
$$
where $x'$ (resp. $y'$) ranges over the numbers simpler than $x$ (resp. $y$) such that $x' < x$ (resp. $y'<y$) and $x''$ (resp. $y''$) ranges over the numbers simpler than $x$ (resp. $y$) such that $x < x''$ (resp. $y<y''$); in other words, when $x = \crotq{x'}{x''}$ and $y = \crotq{y'} {y''}$ are the canonical representations of $x$ and $y$ respectively. The expression for the product may seem not intuitive, but actually, it is basically inspired by the fact that we expect $(x-x')(y-y')>0$, $(x-x'')(y-y'')>0$, $(x-x')(y-y'')<0$ and $(x-x'')(y-y')<0$.

\begin{remark}
The definitions of sum and product are uniform in the sense of \cite[page 15]{gonshor1986introduction}. Namely the equations that define $x + y$ and $xy$ does not require the canonical representations of $x$ and $y$ but any representation. In particular, if $x=\crotq AB$ and $y=\crotq CD$, the variables $x', x'', y', y''$ may range over $A$, $B$, $C$, $D$ respectively.
\end{remark}

It is an early result that these operations, together with the order, give $\No$ a structure of ordered field, and even a structure of real closed field (see \cite[Theorem 5.10]{gonshor1986introduction}). 
Consequently, there is a unique embedding of the rational numbers in $\No$ so we can
identify $\Qbb$ with a subfield of $\No$. 
Actually, the subgroup of the dyadic rationals $m/2^{n}\in \Qbb$, 
with $m\in\Zbb$ and $n \in \Nbb$, correspond exactly to the surreal numbers $s : k \to \{-, +\}$ of finite length $k \in \Nbb.$

The field $\Rbb$ can be isomorphically identified with a subfield of $\No$ by sending $x\in\Rbb$ to the number $\crotq{A}{ B}$ where $A\subseteq\No$ is the set of rationals (equivalently: dyadics) lower than $x$ and $B\subseteq\No$ is the set of (equivalently: dyadics) greater than $x$. This embedding is consistent with the one of $\Qbb$ into $\No$. We may thus write $\Qbb\subseteq\Rbb\subseteq\No$. By \cite[page 33]{gonshor1986introduction}, the length of a real number is at most $\omega$ (the least infinite ordinal). There are however surreal numbers of length $\omega$ which are not real numbers, such as $\omega$ itself or its inverse that is a positive infinitesimal.  

The ordinal numbers can be identified with a subclass of $\No$ by sending the ordinal $\alpha$ to the sequence $s : \alpha \rightarrow\{+,-\}$ with constant value $+$. Under this identification, the ring operations of $\No$, when restricted to the ordinals $\Ord \subseteq \No$, coincide with the Hessenberg sum and product (also called natural operations) of ordinal numbers. Similarly, the sequence $s : \alpha\rightarrow\{+,-\}$ with constant value $-$ corresponds to the opposite (inverse for the additive law) of the ordinal $\alpha$, namely $-\alpha$. We remark that $x \in \Ord$ if and only if $x$ admits a representation of the form $x = \crotq AB$ with $B=\emptyset$, and similarly $x \in -\Ord$ if and only if we can write $x = \crotq AB$ with $A=\emptyset$.

Under the above identification of $\Qbb$ as a subfield of $\No$, the natural numbers $\Nbb \subseteq\Qbb$ are exactly the finite ordinals.

\subsection{Hahn series}
\label{sec:hahn}

%
%
%
%
\subsubsection{Generalities}
Let $\Kbb$ be a field, and let $G$ be a divisible ordered Abelian group.

\begin{definition}[Hahn series \cite{hahn1995nichtarchimedischen}]
The Hahn series (obtained from $\Kbb$ and $G$) are formal power series of the form $s=\sum_{g \in S} a_{g} t^{g}$, where $S$ is a well-ordered subset of $G$ and $a_{g} \in \Kbb .$ The support of s is $\supp(s)=\enstq{g \in S}{ a_{g} \neq 0}$ and the length of $s$ is the order type of $\supp(s)$.

We write $\HahnField{\Kbb}{G}$ for the set of Hahn series with coefficients in $\Kbb$ and terms corresponding to elements of $G$. 
\end{definition}

\begin{definition}[Operations on $\HahnField{\Kbb}{G}$]
The operations on $\HahnField{K}{G}$ are defined in the natural way:
 Let $s=\sum_{g \in S} a_{g} t^{g}, s^{\prime}=\sum_{g \in S^{\prime}} a_{g}^{\prime} t^{g}$, where $S, S^{\prime}$ are well
ordered.
\begin{itemize}
\item $s+s^{\prime}=\sum_{g \in S \cup S^{\prime}}\left(a_{g}+a_{g}^{\prime}\right) t^{g}$, where $a_{g}=0$ if $g \notin S$, and $a_{g}^{\prime}=0$ if $g \notin S^{\prime}$.
\item $s \cdot s^{\prime}=\sum_{g \in T} b_{g} t^{g}$, where $T=\enstq{g_{1}+g_{2}}{g_{1} \in S \wedge g_{2} \in S^{\prime}}$, and for
each $g \in T$, we set
$b_{g}=\Sum{g_{1}\in S, g_{2}\in S' | g_1+g_2=g}{} b_{g_{1}} \cdot b_{g_{2}}$
\end{itemize}
\end{definition}

Hahn fields inherits a lot of from the structure of the coefficient field. In particular if $\Kbb$ is algebraically closed, and if  $G$ is some divisible (i.e. for any $n\in\Nbb$ and $g\in G$ there is some $g'\in G$ such that $ng'=g$) ordered Abelian group, then the corresponding Hahn field is also algebraically closed. More precisely: 

\begin{theorem}[Generalized Newton-Puiseux Theorem, Maclane \cite{maclane1939}]
	\label{thm:macLane}
	Let $G$ be a divisible ordered Abelian group, and let $\Kbb$ be a field that is algebraically closed of characteristic $0$. Then $\HahnField{\Kbb}{G}$ is also algebraically closed.
\end{theorem}
As noticed in \cite{alling1987foundations}, we can deduce the following:
\begin{corollary}
	Let $G$ be a divisible ordered Abelian group, and let $\Kbb$ be a field that is real closed of characteristic $0$. Then $\HahnField{\Kbb}{G}$ is also real closed.
\end{corollary}

\begin{proof}
	$\Kbb$ is real closed. That is to say that $-1$ is not a square in $\Kbb$ and that $\Kbb[i]$ is algebraically closed. Notice that $\Kbb[i]((G))=\left(\Kbb((G))\right)[i]$. Therefore, Theorem \ref{thm:macLane} ensures that $\left(\Kbb((G))\right)[i]$ is algebraically closed. Also, $-1$ is not a square in $\Kbb((G))$. Therefore, $\Kbb((G))$ is real closed. 
\end{proof}


\subsubsection{Restricting length of ordinals}


In this article, will often restrict the class of ordinals allowed in the ordinal sum, namely by restricting to ordinals up to some ordinal $\lambda$. We then give the following notation:
\begin{definition}[$\HahnFieldOrd{\Kbb}{G}{\gamma}$]
Let $\lambda$ be some ordinal.  We define $\HahnFieldOrd{\Kbb}{G}{\gamma}$ for the restriction of $\HahnField{\Kbb}{G}$ to formal power series whose support has an order type in~$\gamma$ (that is to say, corresponds to some ordinal less than $\gamma$).
\end{definition}


%

\begin{theorem}
Assume $\gamma$ is some $\epsilon$-number. Then $\HahnFieldOrd{\Kbb}{G}{\gamma}$ is a field.
\end{theorem}

\begin{proof}
This basically relies on the observation that the length of the inverse of some Hahn series in this field remains in the field: This is basically a consequence of Proposition \ref{prop:orderTypeMonoid}.
\end{proof}


We also get:

\begin{proposition}[{\cite[Lemma 4.6]{DriesEhrlich01}}] 
	\label{prop:hahnFieldRealClosed}
Assume $\Kbb$ is some real closed field, and $G$ is some abelian divisible group.  Then $\HahnFieldOrd{\Kbb}{G}{\gamma}$ is real closed.
\end{proposition}

Actually, this was stated in \cite[Lemma 4.6]{DriesEhrlich01} for the case $\Kbb=\Rbb$, but the proof ony uses the fact that $\Rbb$ is real-closed. 

\subsubsection{Normal form theorem for surreal numbers}

\begin{definition}
	For $a$ and $b$ two surreal numbers, we define the following relations: 
	\begin{itemize}
		\item $a\prec b$ if for all $n\in\Nbb$, $n|a|<|b|$.
		\item $a\preceq b$ if there is some natural number $n\in\Nbb$ such that $|a|<n|b|$.
		\item $a\asymp b$ if $a\preceq b$ and $b\preceq a$.
	\end{itemize}
\end{definition}
With this definition, $\preceq$ is a preorder and $\prec$ is the corresponding strict preorder. The associated equivalence relation is $\asymp$ and the equivalence classes are the Archimedean classes.

\begin{theorem}[{\cite[Theorem 5.1]{gonshor1986introduction}}]
	For all surreal number $a$ there is a unique positive surreal $x$ of minimal length such that $a\asymp x$.
\end{theorem}
The unique element of minimal length in its Archimedean class has many properties  similar to those of exponentiation:

\begin{definition}
	For all surreal number $a$ written in canonical representation $a=\crotq{a'}{a''}$, we define
	$$
		\omega^a=\crotq{0,\enstq{n\omega^{a'}}{n\in\Nbb}}{\enstq{\frac1{2^n}\omega^{a''}}{n\in\Nbb}}
	$$
	we call such surreal numbers \textbf{monomials}.
\end{definition}
Actually this definition is uniform (\cite[Corollary 5]{gonshor1986introduction}) and therefore, we can use any representation of $a$ in this definition. Another point is that we can easily check that this notation is consistent with the ordinal exponentiation. More precisely, if $a$ is an ordinal, $\omega^a$ is indeed the ordinal corresponding to the ordinal exponentiation (see \cite[Theorem 5.4]{gonshor1986introduction}). Finally, as announced, this definition gives the simplest elements among the Archimedean classes.

\begin{theorem}[{\cite[Theorem 5.3]{gonshor1986introduction}}]
	A surreal number is of the form $\omega^a$ if and only if it is simplest positive element in its Archimedean class. More precisely,
	$$
		\forall a\in\No\qquad (\exists c\in\No\quad a=\omega^c)\implies (\forall b\in\Nobf\quad b\asymp a\implies a\sqsubseteq |b|)
	$$
\end{theorem}

Elements of the form $\omega^a$ are by definition positive and have the following property: 

\begin{proposition}[{\cite[Theorem 5.4]{gonshor1986introduction}}]
	We have 
	\begin{itemize}
		\item $\omega^0=1$
		\item $\forall a,b\in\No\qquad \omega^a\omega^b=\omega^{a+b}$
	\end{itemize}
\end{proposition}

Thanks to this definition of the $\omega$-exponentiation, we are now ready to expose a normal form for surreal numbers which is analogous to the Cantor normal form for ordinal normal. 

\begin{definition}[{\cite[Section 5C, page 59]{gonshor1986introduction}}]
	For $\nu$ an ordinal number, $\suitelt{r_i}i\nu$ a sequence of non-zero real numbers and $\suitelt{a_i}i\nu$ a decreasing sequence of surreal numbers, we define $\aSurreal$ inductively as follows:
	\begin{itemize}
		\item If $\nu=0$, then $\aSurreal=0$
		\item If $\nu=\nu'+1$ then $\aSurreal=\aSurrealPrefix+r_{\nu'}\omega^{a_{\nu'}}$
		\item If $\nu$ is a limit ordinal, $\aSurreal$ is defined as the following bracket:
		$$
			\crotq{\enstq{\aSurrealPrefix + s\omega^{a_{\nu'}}}{
					\begin{array}{c}
						\nu'<\nu\\ s<r_{\nu'}
					\end{array}}}{\enstq{\aSurrealPrefix+s\omega^{a_{\nu'}}}{\begin{array}{c}
					\nu'<\nu\\ s>r_{\nu'}
				\end{array}}}
		$$
	\end{itemize}
\end{definition}
Note that if $0$ is seen as a limit ordinal, then both definition are consistent.

\begin{theorem}[{\cite[Theorem 5.6]{gonshor1986introduction}}]
	\label{thm:normalForm}
	Every surreal number can has a unique writing of the form $\aSurreal$. This expression will be called its \textbf{normal form}. 
\end{theorem}
Note that  if $a$ is an ordinal number, then its normal form coincides with its Cantor normal form. In such a sum, elements $r_i\omega^{a_i}$ will be called the \textbf{terms} of the series.

\begin{definition}
	\label{def:nu}
	The length of the series in the normal form of a surreal number $x$ is denoted $\nu(x)$.
\end{definition}

\begin{definition}
	A surreal number $a$ in normal form $a=\aSurreal$ is
	\begin{itemize}
		\item \textbf{purely infinite} if for all $i<\nu$, $a_i>0$. $\Nobf_\infty$ will stand for the class of purely infinite numbers.
		\item \textbf{infinitesimal} if for all $i<\nu$, $a_i<0$ (or equivalently if $a\prec 1$). 
		\item \textbf{appreciable} if for all $i<\nu$, $a_i\leq0$ (or equivalently if $a\preceq 1$).
	\end{itemize}
	If $\nu'\leq\nu$ is the first ordinal such that $a_i\leq0$, then $\aSurrealPrefix$ is called the \textbf{purely infinite part} of $a$. Similarly, if $\nu'\leq\nu$ is the first ordinal such that $a_i<0$, $\Sumlt {\nu'\leq i}\nu r_i\omega^{a_i}$ is called the \textbf{infinitesimal part} of $a$.
\end{definition}

\begin{theorem}[{\cite[Theorems 5.7 and 5.8]{gonshor1986introduction}}]
	\label{thm:normalFormOp}
	Operation over surreal numbers coincides with formal addition and formal multiplication over the normal forms. More precisely,
	
	$$
		\aSurreal+\Sumlt i{\nu'}s_i\omega^{b_i} = \Sum{x\in\No}{}t_x\omega^x
	$$
	where
	\begin{itemize}
		\item $t_x=r_i$ if $i$ is such that $a_i=x$ and there is no $i$ such that $b_i=x$.
		\item $t_x=s_i$ if $i$ is such that $b_i=x$ and there is no $i$ such that $a_i=x$.
		\item $t_x=r_i+s_j$ if $i$ is such that $a_i=x$ and $j$ is such that $b_j=x$
	\end{itemize}
	and
	$$
		\pa{\aSurreal}\pa{\Sumlt i{\nu'}s_i\omega^{b_i}}= \Sum{x\in\No}{}\pa{\Sum{\tiny\enstq{\begin{array}{c}
						i<\nu\\ j<\nu'
				\end{array}}{a_i+b_j=x}}{}r_is_j}\omega^x
	$$
\end{theorem}

We stated that every surreal number has a normal form. However, in the other direction, it is possible to get back the sign expansion from a normal form. 

\begin{definition}[Reduced sign expansion, Gonshor, \cite{gonshor1986introduction}]
	\label{def:reducedSignExpansion}
	Let $x=\Sum{i<\nu}{}r_i\omega^{a_i}$ be a surreal number. The reduced 
	sign expansion of $a_i$, denoted $a_i^\circ$ is inductively defined as follows:
	\begin{itemize}
		\item $a_0^\circ=a_0$
		\item For $i>0$, if $a_i(\delta)=-$ and if there is there is $j<i$ 
		such that for $\gamma\leq\delta$, $a_j(\gamma)=a_i(\gamma)$, then we discard
		the minus in position $\delta$ in the sign expansion of $a_i$.
		\item If $i>0$ is a non-limit ordinal and $(a_{i-1})_-$ (as a sign expansion)
		is a prefix of $a_i$, then we discard this minus after $a_{i-1}$ if $r_{i-1}$
		is not a dyadic rational number. 
	\end{itemize}
\end{definition}
More informally, $a_i^\circ$ is the sign expansion obtained when copying $a_i$ omitting 
the minuses that have already been treated before, in an other exponent of the serie. 
We just keep the new one brought by $a_i$. However, the later case give a condition where
even a new minus can be omitted.

\begin{theorem}[\cite{gonshor1986introduction}, Theorems 5.11 and 5.12]
	\label{thm:serieToSignExp}
	For $\alpha$ an ordinal and a surreal $a$, we write $|a[:\alpha]|_+$ for the (ordinal) number of pluses in $\alpha[:\alpha]$ the prefix of length of $\alpha$ of $x$. Then,
	\begin{itemize}
		\item The sign expansion of $\omega^a$ is as follows: we start with a plus and the for any ordinal $\alpha<|a|$ we add $\omega^{|a[:\alpha]|_++1}$ occurrences of $a(\alpha)$ (the sign in position $\alpha$ in the signs sequence of $a$).
		\item The sign expansion of $\omega^an$ is the signs sequence of $\omega^a$ followed by $\omega^{|a|_+}(n-1)$ pluses.
		\item The sign expansion of $\omega^a\f1{2^n}$ is the sign expansion of $\omega^a$ followed by $\omega^{|a|_+}n$ minuses.
		\item The sign expansion of $\omega^ar$ for $r$ a positive real is the sign expansion of $\omega^a$ to which we add each sign of $r$ $\omega^{|a|_+}$ times excepted the first plus which is omitted. 
		\item The sign expansion of $\omega^ar$ for $r$ a negative real is the sign expansion of $\omega^a(-r)$ in which we change every plus in a minus and conversely.
		\item The sign expansion of $\Sum{i<\nu}{}r_i\omega^{a_i}$ is the juxtaposition of the sign expansions of the $\omega^{a_i^\circ}r_i$
	\end{itemize}
\end{theorem}

As a final note of this subsection, we give some bounds on the length of monomials and terms.

\begin{lemma}[{\cite[Lemma 4.1]{DriesEhrlich01}}]
	\label{lem:lengthOmegaA}
	For all surreal number $a\in\Nobf$, $$\length a\leq \length{\omega^a}\leq\omega^{\length a}$$
\end{lemma}

\begin{lemma}[{\cite[Lemma 6.3]{gonshor1986introduction}}]
	\label{lem:lengthTerm}
	Let $x=\aSurreal$ a surreal number. We have for all $i<\nu$, $\length{r_i\omega^{a_i}}\leq\length x$.
\end{lemma}

\subsubsection{Hahn series and surreal numbers}
As a consequence of Theorems \ref{thm:normalForm} and \ref{thm:normalFormOp}, the field $\No$ in in fact a Hahn serie field. More precisely,

\begin{corollary}
	The fields $\No$ and $\Rbb((t^{\No}))$ are isomorphic.
\end{corollary}

\begin{proof}
	Sending $t^a$ to $\omega^{-a}$ for all surreal number $a$, we notice that all the definitions match to each other.
\end{proof}

Notice that we have of course $\Nobf= \HahnFieldOrd{\Rbb}{\Nobf}{\Ord}$.

\section{Surreal subfields}
\label{sec:erhlichandco}

\subsection{Subfields defined by Gonshor's representation}


%
  Let $\Nolambda$ denote the set surreal number whose signs 
  sequences have length less than $\lambda$ where $\lambda$ is some ordinal.
  We have of course $\Nobf = \bigcup_{\lambda \in \On} \Nolambda$.  
  
Van den Dries and Ehrlich have proved the following: 

  \begin{theorem}[\cite{DriesEhrlich01,van2001erratum}]
   The ordinals $\lambda$
  such that $\Nolambda$ is closed under the various fields operations of $\No$ can be
  characterised as follows:
  \begin{itemize}
  \item $\Nolambda$ is an additive subgroup of $\No$ iff
    $\lambda=\omega^\alpha$ for some ordinal $\alpha$.
  \item $\Nolambda$ is a subring of $\No$ iff
    $\lambda=\omega^{\omega^\alpha}$ for some ordinal $\alpha$.
  \item $\Nolambda$ is a subfield of $\No$ iff
    $\omega^\lambda=\lambda$.
  \label{ou}
\end{itemize}
\end{theorem}
The ordinals $\lambda$ satisfying first (respectively: second) item
are often said to be additively (resp. multiplicatively)
indecomposable but for the sake of brevity we shall just call them
\textbf{additive} (resp. \textbf{multiplicative}). Multiplicative ordinals are exactly
the ordinals $\lambda>1$ such that $\mu \nu <\lambda$ whenever
$\mu,\nu < \lambda$. The ordinal satisfying third item are called
\textbf{$\epsilon$-numbers}. The smallest $\epsilon$-number is usually denoted
by $\epsilon_0$ and is given by
$$\epsilon_0:=\sup\{\omega,\omega^\omega,\omega^{\omega^\omega},\dots\}.$$

\begin{remark} \label{rqdouze}
	Since rational numbers have length at most $\omega$, we have that if $\lambda$ is multiplicative, then $\Nolt\lambda$ is a divisible group.
\end{remark}

If $\lambda$ is an $\epsilon$-number, $\Nolambda$ is actually more than only a field: 

\begin{theorem}[\cite{DriesEhrlich01,van2001erratum}]
Let $\lambda$ be any $\epsilon$-number. Then 
  $\Nolambda$ is a real closed field. 
\end{theorem}

\subsection{Subfields defined from Hahn's series representation}
 
%
%


As we will often play with exponents of formal power series consided in the Hahn series, we propose to introduce the following notation: 

\begin{definition}
	If $\lambda$ is an $\epsilon$-number and $\Gamma$ a divisible Abelian group, we denote  $$\SRF{\lambda}{\Gamma}=\HahnFieldOrd\Rbb{\Gamma}\lambda$$
\end{definition}
As a consequence of Proposition \ref{prop:hahnFieldRealClosed} we have 

\begin{corollary}
$\SRF\lambda{\Nolt\mu}$ is a real-closed field when $\mu$ is a multiplicative ordinal and $\lambda$ an $\epsilon$-number. 
\end{corollary}
This fields are somehow the atoms constituting the fields $\Nolambda$.

\thEhrlichquatresept*


\begin{remark}
The fact that if $\lambda$ is not a regular cardinal, then $\Nolambda \neq\SRF\lambda{\Nolt\lambda}$ can be seen as follows: Suppose that $\lambda$ is not a regular cardinal. This means that we can take some strictly increasing sequence $(\mu_{\alpha})_{\alpha < \beta}$ that is cofinal in $\lambda$ with $\beta<\lambda$. Then $\sum_{\alpha <\beta} \omega^{-\mu_{\alpha}}$ is in $\SRF{\lambda}{\Nolt\lambda}$ by definition, but is not in $\Nolambda$.
\end{remark}

\section{Exponentiation and logarithm}
\label{sec:expoLn}

\subsection{Gonshor's exponentiation}


The field surreal numbers $\No$ admits an exponential function $\exp$ defined as follows.

\begin{definition}[{Function $\exp$, \cite[page 145]{gonshor1986introduction}}]
	Let $x = \crotq{x'}{x''}$ be the canonical representation of $x$. We define inductively
	$$\exp x = \crotq {\begin{array}{c}
			0, \exp(x')[x-x']_{n},\\
			\exp(x'')[x-x'']_{2n+1}
		\end{array}} {\f{\exp(x')}{[x'-x]_{2n+1}}, \f{\exp(x'')}{[x''-x]_{2n+1}}}$$
	where $n$ ranges in $\Nbb$ and
	$$[x]_{n}  = 1+ \frac{x}{1!} + \dots + \frac{x}{n!},$$
	with the further convention that the expressions containing terms of the form
	$[y]_{2n+1}$ are to be considered only when $[y]_{2n+1} > 0$.
\end{definition}
It can be shown that the function $\exp$ is a surjective homomorphism from $(\No, +)$ to $(\No^{>0}, ·)$ which extends $\exp$ on $\Rbb$ and makes $(\No, +, ·, \exp)$ into an elementary extension of $(\Rbb,+,·,\exp)$ (see \cite[Corollaries 2.11 and 4.6]{van1994elementary}, \cite{DriesEhrlich01} and \cite{ressayre1993integer}). As $\exp$ is surjective, and from its properties, it can be shown that it has some inverse $\ln : \No^{>0} \to \No$ (called logarithm).

\begin{definition}[Functions $\log$, $\log_{n}$, $\exp_{n}$]
 Let $\ln : \No^{>0} \to \No$ (called logarithm) be the inverse of $\exp$. We let $\exp_{n}$ and $\ln_{n}$ be the $n$-fold iterated compositions of $\exp$ and $\ln$ with
themselves.
\end{definition}

We recall some other basic properties of the exponential functions:

\begin{theorem}[{\cite[Theorems 10.2, 10.3 and 10.4]{gonshor1986introduction}}]
	\label{thm:expAppreciables}
	For all $r\in\Rbb$ and $\epsilon$ infinitesimal, we have 
	\centre{$\exp r = \Sum{k=0}\infty{\f{r^k}{k!}} \qandq \exp \epsilon = \Sum{k=0}{\infty}\f{\epsilon^k}{k!}$}
 	\lc{and}{$
		\exp(r+\epsilon) = \exp(r)\exp(\epsilon) = \Sum{k=0}{\infty}\f{(r+\epsilon)^k}{k!}$}
	Moreover for all purely infinite number $x$,
		$$\exp(x+r+\epsilon) = \exp(x)\exp(r+\epsilon)$$
\end{theorem}

\begin{proposition}[{\cite[Theorem 10.7]{gonshor1986introduction}}]
	\label{prop:formeExpXPurelyInfiniteOmegaA}
	If $x$ is purely infinite, then $\exp x=\omega^a$ for some surreal number $a$. 
\end{proposition}
More precisely:
\begin{proposition}[Function $g$, {\cite[Theorem 10.13]{gonshor1986introduction}}]
	\label{prop:formeExpXPurelyInfiniteOmegaAg}
	If $x$ is purely infinite, \textit{i.e.} $x=\Sum{i<\nu}{}r_i\omega^{a_i}$ with $a_i>0$ for all $i$, then 
	$$\exp x=\omega^{\Sum{i<\nu}{}r_i\omega^{g(a_i)}},$$
	for some function $g: \No^{>0 }\to \No$. 	Function $g$ satisfies for all $x$, 
	$$g(x) = \crotq{c(x),g(x')}{g(x'')}$$
	where $c(x)$ is the unique number such that $\omega^{c(x)}$ and $x$ are in the same Archimedean class \cite[Thm. 10.11]{gonshor1986introduction} (i.e. such that 
	$x\asymp\omega^{c(x)}$),  where $x'$ ranges over the lower non-zero prefixes of $x$ and $x''$ over the upper prefixes of $x$.
\end{proposition}

\subsection{About some properties of function $g$}

\begin{proposition}[{\cite[Theorem 10.14]{gonshor1986introduction}}]
	\label{prop:gord}
	If $a$ is an ordinal number then
	$$ g(a)=\begin{accolade}
		a+1 & \text{if } \lambda\leq a <\lambda+\omega \text{ for some }\epsilon\text{-number }\lambda\\
		a& \text{otherwise}
	\end{accolade}$$
\end{proposition}

Note that in the previous proposition, $a\neq 0$ since $g$ is defined only for positive elements.

\begin{proposition}[{\cite[Theorem 10.15]{gonshor1986introduction}}]
	\label{prop:gMonomeInfinitesimal}
	Let $n$ be a natural number and $b$ be an ordinal. We have
	$ g(2^{-n}\omega^{-b})=-b+2^{-n}.$
\end{proposition}

\begin{proposition}[{\cite[Theorems 10.17, 10.19 and 10.20]{gonshor1986introduction}}]
	If $b$ is a surreal number such that for some $\epsilon$-number $\epsilon_i$, some ordinal $\alpha$ and for all natural number $n$,
	$\epsilon_i+n<b<\alpha<\epsilon_{i+1}$, then $g(b)=b$. This is also true if there is some ordinal $\alpha<\epsilon_0$ such that for all natural number $b$, $n\omega^{-1}<b<\alpha <\epsilon_0$.
\end{proposition}

\begin{proposition}[{\cite[Theorem 10.18]{gonshor1986introduction}}]
	If $\epsilon\leq b\leq\epsilon + n$ for some $\epsilon$-number $\epsilon$ and some integer $n$. In particular, the sign expansion of $b$ is the sign expansion of $\epsilon$ followed by some sign expansion $S$. Then, the sign expansion of $g(b)$ is the sign expansion of $\epsilon$ followed by a $+$ and then $S$. In particular, $g(b)=b+1$. 
\end{proposition}

It is possible to bound the length of $g(a)$ depending on the length of $a$. 

\begin{lemma}[{\cite[Lemma 5.1]{DriesEhrlich01}}]
	\label{lem:lengthGA}
	For all $a\in\Nobf$, $\length{g(a)}\leq \length a +1$.
\end{lemma}

The function $g$ has a inverse function, $h$ defined as follows

$$h(b) = \crotq{0,h(b')}{h(b''),\f{\omega^b}{n}}$$
This expression is uniform (see \cite{gonshor1986introduction}) and then does not depend of the expression of $b$ as $\crotq{b'}{b''}$.

\begin{corollary}
	\label{cor:hmord}
	If $a$ is an ordinal number then $h(-a)=\omega^{-a-1}$.
\end{corollary}

\begin{proof}
	It is a direct consequence of Proposition \ref{prop:gMonomeInfinitesimal} and the fact that $h=g^{-1}$.
\end{proof}

As for $g$, we can bound the length of $h(a)$ in function of the length of $a$.

\begin{lemma}[{\cite[Proposition 3.1]{aschenbrenner:hal-02350421}}]
	\label{lem:lengthH}
	For all $a\in\Nobf$ we have, 
	$$\length{h(a)}\leq\omega^{\length a +1} $$
\end{lemma}
We will also prove another lemma, Lemma \ref{lem:lengthOmegaGA}, that looks like the previous lemma but that is better in many cases but not always. To do so we first prove another technical lemma.

\begin{lemma}
	\label{lem:gomegacCasSpec}
	For all $c$, denote $c_+$ the surreal number whose signs sequence is the one of $c$ followed by a plus. Assume $g(a)<c$ for all $a\sqsubset\omega^c$ such that $0<a<\omega^c$. Then $g(\omega^c)$ is $c_+$ if $c$ does not have a longest prefix greater than itself, otherwise, $g(\omega^c)=c''$ where $c''$ is the longest prefix of $c$ such that $c''>c$.
\end{lemma}

\begin{proof}
	By induction on $c$:
	
	\begin{itemize}
		\item For $c=0$, $g(\omega^0)=g(1)=1$ whose signs sequence is indeed the one of $0$ followed by a plus.
		
		\item Assume the property for $b\sqsubset c$. Assume $g(a')<c$ for all $a'\sqsubset\omega^c$ such that $0<a'<\omega^c$. Then,
		$$
			g(\omega^c)=\crotq{c}{g(a'')}
		$$
		where $a''$ ranges over the elements such that $a''\sqsubset \omega^c$ and $a''>\omega^c$.
		
		\begin{itemize}
			\item First case: $c$ has a longest prefix $c_0$ such that $c_0>c$. Then, for all $a''$ such that $a''\sqsubset \omega^c$ and $a''>\omega^c$, $a''\succeq \omega^{c_0}$, hence $g(a'')> c_0$. Since $c<c_0<g(a'')$, the simplicity property ensures $g(\omega^c)\sqsubseteq c_0\sqsubset c$. Then $g(\omega^c)$ is some prefix $c''$ of $c$, greater than $c$. 	We look at $\omega^{c''}$. Notice that for all $b\sqsubset c''$ is such that $0<b<c''$, $b\sqsubset c$ and $b<c$, hence $g(b)<c<c''$. Therefore we can apply the induction hypothesis to $c''$ and $g(\omega^{c''})$ is $c''_+$ if the signs sequence of $c$ does not end with only minuses, otherwise, $g(\omega^{c''})$ is the last (strict) prefix of $c''$ greater than $c''$. 
			
			\begin{itemize}
				\item First subcase: $g(\omega^{c''})=c''_+$. If there is some $b$ such that $c''\sqsubset b\sqsubset c$ and $b>c$, then $g(\omega^b)$ is a prefix of $g(\omega^c)=c''$. But, $c''=g(\omega^c)<g(\omega^b)<g(\omega^{c''})=c''_+$. Then $c''$ must be a strict prefix of $g(\omega^b)$ which is a contradiction. Then $c''$ is indeed the last strict prefix of $c$ greater than $c$.
				
				\item Second subcase:  $g(\omega^{c''})$ is the last (strict) prefix of $c''$ greater than $c''$. If there is some $b$ such that $c''\sqsubset b\sqsubset c$ and $b>c$, then $g(\omega^b)$ is a prefix of $g(\omega^c)=c''$. Since $g(b)<g(c'')$, $g(b)$ is prefix of $c''$ smaller than $c''$. But this contradicts the fact that $g(\omega^b)>g(\omega^c)=c''$. Therefore, $c''$ is the last prefix of $c$ greater than $c$.
			\end{itemize}

			\item Second case: $c$ does not have a longest prefix greater than $c$. Then,
			$$
				g(\omega^c) = \crotq c{g(\omega^{c''})}
			$$
			where $c''$ ranges over the prefixes of $c$ greater than $c$. Let $d\sqsubset c$ such that $d>c$. Then there is $d_1$ or minimal length such that $d\sqsubset d_1\sqsubset c$ and $d_1>c$. By minimality of $d_1$, $d$ is the longest prefix of $d_1$ greater than $d_1$. As in the first case, we can apply the induction hypothesis on $d_1$ and get $g(\omega^{d_1})=d$. Therefore, again by induction hypothesis,
			$$
				g(\omega^c)=\crotq c{c'',c''_+} = \crotq c{c''}
			$$
			where $c''$ ranges over the prefixes of $c$ greater than $c$. We finally conclude that $g(\omega^c)=c_+$.
		\end{itemize}
	\end{itemize}
\end{proof}

In the following we denote $\oplus$ the usual addition over the ordinal numbers and~$\otimes$ the usual product over ordinal numbers.

\begin{lemma}
	\label{lem:lengthOmegaGA}
	For all $a>0$, $\length a\leq\length{\omega^{g(a)}}\otimes(\omega+1)$.
\end{lemma}

\begin{proof}
	We proceed by induction on $\length a$. 
	\begin{itemize}
		\item For $a=1$, $g(a)=1$ and we indeed have $1\leq \omega^2$.
		
		\item Assume the property for all $b\sqsubset a$. Let $c$ such that $\omega^c\asymp a$. Then 
		$$
			g(a)=\crotq{c,g(a')}{g(a'')}
		$$
		We split into two cases:
		\begin{itemize}
			\item If there is some $a_0\sqsubset a$ such that $a_0<a$ and $g(a_0)\geq c$ then
			$$
			g(a)=\crotq{g(a')}{g(a'')}
			$$
			and if $S$ stand for the signs sequence such that $a$ is the signs sequence of $a_0$ followed by $S$, $g(a)$ is the signs sequence of $g(a_0)$ followed $S$. Let $\alpha$ the length of $S$. Therefore using Theorem \ref{thm:serieToSignExp},
			$$
			\length{\omega^{g(a)}}\geq\length{\omega^{g(a_0)}}\oplus(\omega\otimes\alpha)
			$$
			and then,
			\begin{align*}
				\length{\omega^{g(a)}}\otimes(\omega+1) &\geq \length{\omega^{g(a_0)}}\otimes\omega\oplus \length{\omega^{g(a_0)}}\oplus\alpha \\
				&\geq\length{\omega^{g(a_0)}}\otimes(\omega+1)\oplus\alpha
			\end{align*}
			and by induction hypothesis on $a_0$,
			$$
			\length{\omega^{g(a)}}\otimes(\omega+1)\geq\length{a_0}\oplus\alpha=\length a
			$$
			
			\item Otherwise, for any $a_0\sqsubset a$ such that $a_0<a$, $g(a_0)<c$. Therefore,
			$$ g(a)=\crotq{c}{g(a'')}$$
			Also, since $a>0$, we can write the signs sequence of $a$ as the one of $\omega^c$ followed by some signs sequence $S$. If $S$ contains a plus, then there is a prefix of $a$, $a_0$ such that $a_0<a$ and still $a_0\asymp\omega^c$ and then $g(a_0)>c$ what is not the case by assumption. Then, $S$ is a sequence of minuses. If $S$ is not the empty sequence, let $\alpha$ be the length of $S$. Then the signs sequence of $g(a)$ is the one of $g(\omega^c)$ followed by $S$. Hence,
			$$
				\length{\omega^{g(a)}}\geq \length{\omega^{g(\omega^c)}}\oplus(\omega\otimes\alpha)
			$$ 
			As in the previous case, but using the induction hypothesis on $\omega^c$,
			$$
				\length{\omega^{g(a)}}\otimes(\omega+1)\geq\length{\omega^c}\oplus\alpha=\length a
			$$
			Now if $S$ is the empty sequence, $a=\omega^c$. Applying Lemma \ref{lem:gomegacCasSpec} to $c$ we get that either $g(a)=c_+$ or $g(a)$ is the last prefix of $c$ greater than $c$. If the first case occurs then $a$ is a prefix of $\omega^{g(a)}$ and then $\length{\omega^{g(a)}}\geq\length a$. Now assume that the second case occurs. Then for any $b$ such that $g(a)\sqsubset b\sqsubset c$, $b<c$.
			If for all $b'\sqsubset b$ such that $b'<b$, $g(b')<b$, then Lemma \ref{lem:gomegacCasSpec} applies. Since $b$ has a last prefix greater than itself, $g(a)$, $g(\omega^b)=g(a)$ and we reach a contradiction since $b<c$ and therefore $\omega^b<\omega^c=a$. Then for all  $b$ such that $g(a)\sqsubset b\sqsubset c$, there is some $b'\sqsubset b$, $b'<b$ such that $g(\omega^{b'})>b$. Since the signs sequence of $b$ consists in the one of $g(a)$ a minus and then a bunch of pluses, and since $g(\omega^{b'})$ must also a a prefix of $c$, $g(\omega^{b'})\sqsubseteq g(a)\sqsubset b$. Therefore to ensure $g(b')>b$, we must have $g(\omega^{b'})\geq g(a)$. Since $\omega^{b'}$ is a prefix of $a$ lower than $a$, it is a contradiction. Therefore, there is no $b$ such that $g(a)\sqsubset b\sqsubset c$ and $b<c$, and finally, the signs sequence of $c$ is the one $g(a)$ followed by a minus. In particular, $g(a)$ and $c$ have the same amount of pluses, say $\alpha$. Then, using Theorem \ref{thm:serieToSignExp},
			\begin{align*}
				\length a &= \length{\omega^{g(a)}}\oplus\omega^{\alpha+1}\\
				&\leq \length{\omega^{g(a)}}\oplus \length{\omega^{g(a)}}\otimes\omega = \length{\omega^{g(a)}}\otimes\omega\\
				&\leq \length{\omega^{g(a)}}\otimes(\omega+1)
			\end{align*}
		\end{itemize}
		The induction principle concludes.
	\end{itemize}
\end{proof}

\begin{corollary}
	\label{cor:lengthOmegaGA}
	For all $a>0$ and for all multiplicative ordinal greater than $\omega$, if $\length a\geq\mu$, then $\length{\omega^{g(a)}}\geq\mu$.
\end{corollary}

\begin{proof}
	Assume the that $\length{\omega^{g(a)}}<\mu$. Then using Lemma \ref{lem:lengthOmegaGA}, $\mu\leq \length{\omega^{g(a)}}\otimes(\omega+1)$. Since $\mu$ is a multiplicative ordinal greater than $\omega$, we have $\omega+1<\mu$. $\mu$ is a multiplicative ordinal, hence $\length{\omega^{g(a)}}\otimes(\omega+1)<\mu$ and we reach a contradiction.
\end{proof}

\subsection{Gonshor's logarithm}

We already know that a logarithm exist over positive surreal numbers. Nevertheless we were very elliptical and we now get deeper into it.

\begin{definition}
	For a surreal number $a$ in canonical representation $a=\crotq{a'}{a''}$, we define
	$$
		\ln\omega^a=\crotq{\begin{array}{c}
				\enstq{\ln\omega^{a'}+n}{\begin{array}{c}n\in\Nbb\\ a'\sqsubset a\\ a'<a\end{array}}\\
				\enstq{\ln\omega^{a''}-\omega^{\frac{a''-a}n}}{\begin{array}{c}n\in\Nbb\\ a''\sqsubset a\\ a<a''\end{array}}
			\end{array}\!\!}{\!\!\begin{array}{c}
				\enstq{\ln\omega^{a''}-n}{\begin{array}{c}n\in\Nbb\\ a''\sqsubset a\\ a<a''\end{array}}\\
				\enstq{\ln\omega^{a'}+\omega^{\frac{a-a'}{n}}}{\begin{array}{c}n\in\Nbb\\ a'\sqsubset a\\ a'<a\end{array}}
			\end{array}}
	$$
\end{definition}
As often with this kind of definitions, the uniformity property holds.
\begin{lemma}[{\cite[Lemma 10.1]{gonshor1986introduction}}]
	The definition of $\ln\omega^a$ does not require $a$ in canonical representation.
\end{lemma}

\begin{proposition}[{\cite[Theorem 10.8]{gonshor1986introduction}}]
	\label{prop:lnOmegaA}
	For all surreal number $a$, $\ln\omega^a$ is purely infinite.
\end{proposition}
Purely infinite numbers are a special case in the definition of the exponential function. We can state the previous definition of $\ln$ is consistent with the one of $\exp$.

\begin{theorem}[{\cite[Theorem 10.9]{gonshor1986introduction}}]
	\label{thm:expLnOmegaA}
	For all surreal number $a$, 
	$$\exp\ln\omega^a=\omega^a$$
\end{theorem}

\begin{theorem}[{\cite[Theorem 10.12]{gonshor1986introduction}}]
	\label{thm:lnOmegaOmegaA}
	For all surreal number $a$, 
	$$\ln\omega^{\omega^a}=\omega^{h(a)}$$
\end{theorem}
The above theorem is not actually stated like this in \cite{gonshor1986introduction} but this statement follows from the proof there.

As a consequence of Theorems \ref{thm:expLnOmegaA} and \ref{thm:lnOmegaOmegaA} and Propositions \ref{prop:lnOmegaA} and \ref{prop:formeExpXPurelyInfiniteOmegaAg}, we have 

\begin{corollary}
	For all surreal number $a=\aSurreal$, we have
	$$
		\ln\omega^a = \Sumlt i\nu r_i\omega^{h(a_i)}
	$$ 
\end{corollary}

Finally, since for appreciable numbers $\exp$ is defined by its usual serie, $\ln(1+x)$ is also defined by its usual serie when $x$ in infinitesimal. More precisely,

\begin{definition}
	\label{def:lnAppreciables}
	For $x$ an infinitesimal, 
	$$
		\ln(1+x) = \Sum{i=1}{\infty}\frac{(-1)^{i-1}x^i}{i}
	$$
\end{definition}
And thanks to Theorem \ref{thm:expAppreciables},

\begin{corollary}
	Let $a=\aSurreal$ a positive surreal number. Then
	$$
		\ln a = \ln\omega^{a_0} + \ln r_0 + \ln\pa{1+\Sumlt{1\leq i}\nu \frac{r_i}{r_0}\omega^{a_i-a_0}}
	$$
	where the last term is defined in Definition \ref{def:lnAppreciables}.
\end{corollary}

\subsection{Stability of $\Nolambda$ by exponential and logarithm}

We first recall some result by van den Dries and Ehrlich.

\begin{lemma}[{\cite[Lemmas 5.2, 5.3 and 5.4]{DriesEhrlich01}}]
	\label{lem:lengthExpLog}
	For all surreal number $a\in\Nobf$,
	\begin{itemize}
		
		\item $\length{\exp a}\leq\omega^{\omega^{2\length a\oplus3}}$
		
		\item $\length{\ln\omega^a} \leq \omega^{4\omega\length{a}\length{a}}$
		
		\item $\length{\ln a}\leq\omega^{\omega^{3\length a\oplus3}}$
	\end{itemize}
\end{lemma}

\begin{corollary}[{\cite[Corollary 5.5]{DriesEhrlich01}}]
	\label{cor:NolambdaStableExpLn}
	For $\lambda$ an $\epsilon$-number, $\Nolt\lambda$ is stable under $\exp$ and $\ln$.
\end{corollary}

We have 

\begin{theorem}[{\cite[Theorem 1.3]{bournez2022surreal}}]
	The following are equivalent:
	\label{thm:NoltStableExpLn} 
	\begin{itemize}
		\item $\Nolambda$ is a subfield of $\No$ stable by $\exp$, and $\ln$ 
		\item   $\Nolambda$ is a subfield of $\No$ 
		\item  $\lambda$ is some $\epsilon$-number.
	\end{itemize}
\end{theorem}

\subsection{A hierarchy of subfields of $\No$ stable by exponential and logarithm}
In this subsection we recall our previous work on a hierarchy of surreal subfields stable under exponential and logarithm.

We start by Theorem \ref{thm:SRFGammaUpStableExpLn} repeated here for readability:

\thmSRFGammaUpStableExpLn*

This result is actually a consequence of a more general proposition which is the following.
\begin{proposition}[{\cite[Proposition 5.1]{bournez2022surreal}}]
	\label{prop:UnionSRFStableExpLn}
	Let $\lambda$ be an $\epsilon$-number and $\suite{\Gamma_i}iI$ be a family of Abelian subgroups of $\Nobf$. Then
	$\RlGI$ is stable under $\exp$ and $\ln$ if and only if $$\Unionin iI\Gamma_i=\Unionin iI\SRF\lambda{g\pa{\pa{\Gamma_i}^*_+}}$$
\end{proposition}

Note that a consequence of Proposition \ref{prop:UnionSRFStableExpLn} is also the following:

\begin{corollary}[{\cite[Corollary 5.2]{bournez2022surreal}}]
	Let $\lambda$ be an $\epsilon$-number and $\Gamma$ be an abelian subgroup of $\Nobf$. Then
	$\RlG$ is stable under $\exp$ and $\ln$ if and only if $\Gamma=\SRF\lambda{g\pa{\Gamma^*_+}}$.
\end{corollary}
This result is quite similar to Theorem \ref{thm:SRFGammaUpStableExpLn} but in the particular very particular case where $\Unionin G{\Gamma^{\uparrow\lambda}}G=\Gamma$. This apply for instance when $\Gamma=\{0\}$. In this case, we get $\RlG=\Rbb$. If $\lambda$ is a regular cardinal we get an other example considering $\SRF\lambda\Gamma=\Gamma=\Nolt\lambda$.

Theorem \ref{thm:SRFGammaUpStableExpLn} enables us to consider a lot of fields stable under exponential and logarithm and enabled us to prove that we can express $\Nolt\lambda$ as a strictly increasing hierarchy of fields stable under $\exp$ and $\ln$.

\thmNolambdaDecompCorpsStables*

%
%

\thmhierarchieUparrow*

\section{The class of log-atomic numbers, derivation and anti-derivation}
\label{sec:deriv}
\subsection{Log-atomic numbers}

We now introduce the concept of \textbf{log-atomic numbers}. Log-atomic numbers were first introduced by Schmeling in \cite[page 30]{schmeling2001corps} about transseries. Such number are basically number whose series of iterated logarithm have all length $1$.

\begin{definition}[Log-atomic]
	A positive surreal number $x\in\No^*_+$ is said \textbf{log-atomic} \tiff for all $n\in\Nbb$, there is a surreal number $a_n$ such that $\ln_nx=\omega^{a_n}$. We denote $\Lbb$ the class of log-atomic numbers.
\end{definition}

For instance, $\omega$ is a log-atomic number and we can check that for all $n\in\Nbb$, $\ln_n \omega = \omega^{\f1{\omega^n}}$. Log-atomic number are the number we cannot divide into simpler numbers when considering exponential and logarithm and are the fundamental blocs we end up with when writing $x=\aSurreal$ and then each $\omega^{a_i}$ as $\omega^{a_i}=\exp x_i$ with $x_i$ a purely infinite number and then doing the same thing with each of the $x_i$s. The use of the word ``simpler'' is not innocent. Indeed, log-atomic numbers are also the simplest elements for some equivalence relation introduced by Beraducci and Mantova \cite{Berarducci_2018}.

\begin{definition}[{\cite[Definition 5.2]{Berarducci_2018}}]
	\label{def:asympPrecL}
	Let $x,y$ be two positive infinite surreal numbers. We write
	\begin{itemize}
		\item $x\asymp^L y$\index{$\asymp^L$, Definition \ref{def:asympPrecL}} \tiff there are some natural numbers $n,k$ such that $$\exp_n\pa{\f1k\ln_n y}\leq x\leq \exp_n\pa{k\ln_ny}$$
		Equivalently, we ask that the is a natural number $n$ such that  $\ln_n x \asymp \ln_n y$. For such $n$ we notice that $\ln_{n+1}x\sim\ln_{n+1}y$.
		\item $x\prec^L y$\index{$\prec^L,\preceq^L$, Definition \ref{def:asympPrecL}} \tiff for all natural numbers $n$ and $k$, 
		$$x<\exp_n\pa{\f1k\ln_n y}$$
		Equivalently, we ask that for all $n\in\Nbb$, $\ln_n x\prec\ln_n y$.
		\item $a\preceq^L b$ \tiff there are some natural numbers $n$ and $k$, 
		$$x\leq\exp_n\pa{\f1k\ln_n y}$$
		Equivalently, we ask that for some $n\in\Nbb$, $\ln_n x\preceq\ln_n y$.
	\end{itemize}
\end{definition}

Log-atomic number are closely related to this equivalence relation since they representatives of each equivalence classes.

\begin{proposition}[{\cite[Propositions 5.6 and 5.8]{Berarducci_2018}}]
	\label{prop:lamdaMapLogAtomic}
	For all positive infinite $x$ there is unique log-atomic number $y\in\Lbb$ such that $y\sqsubseteq x$ and such that $y\asymp^L x$. In particular, if $x,y\in\Lbb$ with $x<y$ then $x\prec^L y$.
\end{proposition}

This proposition shows in particular that not even log-atomic are representative of the equivalence classes of $\asymp^L$, they also are the simplest element (\ie{} the shortest in terms of length) in their respective equivalence classes. This make them a canonic class of representatives.

As we can parametrize additive ordinal, multiplicative ordinal or even $\epsilon$-numbers (for which a generalization for surreal numbers exists in Gonshor's book \cite{gonshor1986introduction}), we can parametrized epsilon numbers by a an increasing function $\lambda_\cdot$. A first conjecture was to consider $\kappa$-numbers which are defined by Kuhlmann and Matusinski as follows:

\begin{definition}[{\cite[Definition 3.1]{Kuhlmann_2014}}]\label{def:kappaMap}
	Let $x$ be a surreal number and write it in canonical representation as $x=\crotq{x'}{x''}$. Then we define
	$$
	\kappa_x=\crotq{\Rbb,\exp_n \kappa_{x'}}{\ln_n \kappa_{x''}}
	$$
\end{definition}

Intuitively, $x<y$ \tiff every iterated exponential of $\kappa_x$ is less than $\kappa_y$ and we try to build them as simple as possible. As an example, it is quite easy to see that $\kappa_0=\omega$, $\kappa_{-1}=\omega^{\omega^{-\omega}}$ and $\kappa_1=\epsilon_0$. It was conjectured that $\Lbb$ consists in $\kappa$-number and there iterated exponentials and logarithms. As shown by Berarducci and Mantova, it turns out that it is not true. They then suggest a more general map which is the following:

\begin{definition}[{\cite[Definition 5.12]{Berarducci_2018}}]\label{def:lambdaMap}
	Let $x$ be a surreal number and write it in canonical representation $x=\crotq{x'}{x''}$. Then we define
	$$
	\lambda_x=\crotq{\Rbb,\exp_n\pa{k\ln_n \lambda_{x'}}
	}{\exp_n\pa{\f1k\ln_n\lambda_{x''}}}
	$$
	where $n,k\in\Nbb^*$.
\end{definition}

\begin{proposition}[{\cite[Proposition 5.13 and Corollary 5.15]{Berarducci_2018}}]
	The function $x\mapsto\lambda_x$ is well defined, increasing, satisfies the uniformity property\index{Unifromity property} and if $x<y$ then $\lambda_x\prec^L\lambda_y$.
\end{proposition}

\begin{proposition}[{\cite[Proposition 5.16]{Berarducci_2018}}]
	For every $x\in \No$ with $x > \Rbb$ there is a unique $y\in\No$ such
	that $x\asymp^L\lambda_y$ and $\lambda_y\sqsubseteq x$. In particular, $\lambda_y$ is the simplest number in its equivalence class for $\asymp^L$. As a consequence, $\lambda_\No=\Lbb$.
\end{proposition}

Moreover, the $\lambda$ map behaves very nicely with exponential and logarithm.

\begin{proposition}[{\cite[Proposition 2.5]{aschenbrenner:hal-02350421}}]
	For all surreal number $x$, 
	$$
	\exp\lambda_x=\lambda_{x+1}\qqandqq\ln\lambda_x=\lambda_{x-1}
	$$
\end{proposition}

\begin{lemma}[{\cite[Lemma 2.6]{aschenbrenner:hal-02350421}}]
	For all ordinal $\alpha$, $\lambda_{-\alpha}=\omega^{\omega^{-\alpha}}$.
\end{lemma}

\begin{lemma}[{\cite[Aschenbrenner, van den Dries and  van der Hoeven, Corollary 2.9]{aschenbrenner:hal-02350421}}]
	\label{lem:kappaMoinsAlpha}
	For all ordinal number $\alpha$, \centre{$\kappa_{-\alpha}=\lambda_{-\omega\otimes\alpha}=\omega^{\omega^{-\omega\otimes\alpha}}$}
\end{lemma}
\subsection{Nested truncation rank}
\subsubsection{Definition}
Log-atomic number are the base case (up to minor changes) of a notion of rank over surreal numbers, the \textbf{nested truncation rank}. As expected, it is based on some well partial order. This one has been defined by Berarducci and Mantova as follows:

\begin{definition}[{\cite[Definition 4.3]{Berarducci_2018}}]
	\label{def:nestedn}
	For all natural number $n\in\Nbb$, we define the relation $\nestedneq$ as follows:
	\begin{itemize}
		\item Writing $y\nestedneq[0] x$ if any only if $y=\aSurrealPrefix$ and $x=\aSurreal$ with $\nu'\leq\nu$. We say that $y$ is a \textbf{truncation} of $x$.
		
		\item Let $x=\aSurreal$  Since $\omega^\Nobf=\exp(\No_\infty)$, we can write 
		$$
		x=\Sumlt i\nu r_i\exp(x_i)
		$$
		where $\exp(x_i)=\omega^{a_i}$. For a surreal number $y$, we say $y\nestedneq[n+1] x$ if there is $\nu'<\nu$ and $y'\nestedneq x_{\nu'}$ such that
		$$
		y=\Sumlt i{\nu'}r_i\exp(x_i) + \sign(r_{\nu'})\exp y'
		$$ 
		We say that $y$ is a \textbf{nested truncation} of $x$.
	\end{itemize}
	We also write $y\nestedeq x$ is there is some natural number $n$ such that $y\nestedneq x$. We also introduce the corresponding strict relations $\nestedn$ and $\nested$.
\end{definition}

\begin{definition}[Nested truncation rank {\cite[Definition 4.27]{Berarducci_2018}}]
	\label{def:nestedTruncRk}
	The \textbf{nested truncation rank} of $x\in\Nobf^*$ is defined by
	$$
	\NR(x) = \sup\enstq{\NR y + 1}{y\nested x}
	$$
	By convention, we also set $\NR(0)=0$.
\end{definition}

\subsubsection{Properties}
We know investigate some properties of the nested truncation rank. More precisely, we provide compatibility properties with the operations over surreal numbers and bounds on some particular nested truncation ranks. First of all, the nested truncation rank is unaffected by the exponential.

\begin{proposition}[{\cite[Proposition 4.28]{Berarducci_2018}}]
	\label{prop:NRStableExp}
	If $\gamma\in\Nobf_\infty$, then $$\NR(\pm\exp \gamma) = \NR(\gamma)$$
\end{proposition}

\begin{corollary}
	\label{cor:NRminus}
	For all $a\in\Nobf^*$, $\NR(a)=\NR\pa{-a}$
\end{corollary}
\begin{proof} Without loss of generality, we assume that $a>0$. Then
	\begin{calculs}
		& \NR\pa{a} &=& \NR\pa{\ln a} & (Proposition \ref{prop:NRStableExp})\\
		&&=& \NR(-\exp\ln a) &(Proposition \ref{prop:NRStableExp}) \\
		&&=& \NR\pa{-a}\\
	\end{calculs}
\end{proof}

\begin{corollary}
	\label{cor:NRinverse}
	For all $a\in\Nobf^*$, $\NR(a)=\NR\pa{\f1a}$
\end{corollary}
\begin{proof}\ 
	\begin{calculs}
		& \NR\pa{\f1a} &=& \NR\pa{\ln\f1a} & (Proposition \ref{prop:NRStableExp})\\
		&&=& \NR(-\ln a)\\
		&&=& \NR\pa{\ln a} &(Corollary \ref{cor:NRminus})\\\\
		&&=& \NR(a) & (Proposition \ref{prop:NRStableExp})\\
	\end{calculs}
\end{proof}

\begin{lemma}\label{lem:NR0}
	For all $x\in\No$, $\NR(x)=0$ \tiff either $x\in\Rbb$ or $x=\pm\lambda^{\pm1}$ for some log-atomic number $\lambda$.
\end{lemma}

\begin{proof}
	\begin{itemize}
		\item[\CS] Note that if $x\in\Rbb$ then there is no $y\in\Nobf$ such that $y\nested x$. Therefore $\NR(x)=0$. 
		Now assume that there is some $x=\pm\lambda^{\pm 1}$ with $\lambda\in\Lbb$ such that $\NR(x)\neq 0$. Therefore there is some $y\in\Nobf$ such that $y\nested x$. Let $n\in\Nbb$ minimal such that there is $y\in\Nobf$ and $\lambda\in\Lbb$ such that $y\nestedn \pm\lambda^{\pm 1}$.
		Note that since $\pm\lambda^{\pm1}$ is a term, $n>0$. Then $y=\pm\exp(\pm y')$ with $y'\nestedn[n-1]\ln\lambda\in\Lbb$. But this contradicts the minimality of $n$. hence, for all $\lambda\in\Lbb$, $\NR\pa{\pm\lambda^{\pm1}}=0$.
		
		\item[\CN] Assume $\NR(x)=0$ and $x$ is not a real number. If $x$ is not a term, then there is $y\nestedn[0] x$ and in particular $\NR(x)\geq 1$, what is impossible. Therefore there is some $r\in\Rbb^*$ and some $x_1\in\Jbb$ such that $x=r\exp(x_1)$. If $r\neq \pm1$ then $\sign(x)\exp(x')\nested x$ what is again impossible. Hence, $x=\pm\exp(x_1)$. Proposition \ref{prop:NRStableExp} ensures that $\NR(x_1)=0$. We then can apply the same work to $x_1$ so that there is some $x_2\in\Jbb$ such that $x_1=\pm\exp(x_2)$. By induction, we can always define $x_n=\pm\exp(x_{n+1})$ with $x_{n+1}\in\Jbb$. For $n\geq 1$ we have $x_n\in\Jbb$, therefore $x_{n+1}>0$. In particular
		$$
		\forall n\geq 2\qquad x_n=\exp(x_{n+1})
		$$
		So, for all $n\in\Nbb$,  $\ln_n x_2$ is a monomial, this means that $x_2\in\Lbb$. We also have 
		$$x=\pm\exp\pa{\pm\exp x_2}=\pm\pa{\exp_2(x_2)}^{\pm1}$$
		Since $\exp_2 x_2\in\Lbb$, we have the expected result.
	\end{itemize}
\end{proof}

\begin{lemma}
	\label{lem:NRAjoutNouveauTerme}
	Let $x=\aSurreal$ and $r\in\Rbb^*, a\in\Nobf$ such that for all $i<\nu$, $r\omega^a\prec\omega^{a_i}$. Then 
	\centre{$\NR(x+r\omega^a) = \NR(x)\oplus1\oplus\NR(\omega^a)\oplus\ind_{r\neq\pm1}$}
	where the $\oplus$ is the usual sum on ordinal numbers.
\end{lemma}

\begin{proof}
	Let $y\nested x+r\omega^a$. Then $y\nestedeq x$ or $y=x+\sign(r)\exp(\delta)$ with $\delta\nested\ln\omega^a$ or, if $r\neq\pm 1$,  $y=x+\sign(r)\omega^a$.
	Let \centre{$A=\enstq{y}{y\nestedeq x}\qqandqq B=\enstq{x+\sign(r)\exp(\delta)}{\delta\nested \ln\omega^a}$}
	\lc{and}{$C=\begin{accolade}
			\emptyset & r=\pm 1 \\
			x+\sign(r)\omega^a & r\neq\pm 1
		\end{accolade}$}
	One can easily see that
	\centre{$\forall y\in A\quad \forall y'\in B\quad\forall y''\in C \qquad y\nested y'\wedge y\nested y'' \wedge y'\nested y''$}
	We now proceed by induction on $\NR(\omega^a)$.
	\begin{itemize}
		\item If $\NR(\omega^a)=0$, using Lemma \ref{lem:NR0}, either $\omega^a=\pm\lambda^{\pm 1}$ for some log-atomic number $\lambda$ or $a=0$. In both cases, there is no $\delta\nested\ln\omega^a$. 
		\begin{calculs}
			& \NR(x+r\omega^a) &=& \sup\enstq{\NR(y)+1}{y\in A\cup C}\\
			&&=& \sup\pa{\enstq{\underbrace{\NR(y)+1}_{\leq\NR(x)}}{y\nested x}\right.\\&&&\left.\qquad\vphantom{\underbrace{\NR(y)+1}_{\leq\NR(x)}}\cup\{\NR(x)+1\} \cup\enstq{\NR(y)+1}{y\in C}} \\
			&&=& \begin{accolade}
				\NR(x) + 1 & r=\pm 1 \\
				\NR(x+\sign(r)\omega^a) & r\neq \pm 1 
			\end{accolade}\\
			&&=& \begin{accolade}
				\NR(x) + 1 & r=\pm 1 \\
				\NR(x)+2 & r\neq \pm 1 
			\end{accolade}\\
			& \NR(x+r\omega^a) &=& \NR(x)+1+\NR(\omega^a) + \ind_{r\neq\pm 1}
		\end{calculs}
		\item For heredity now. Let $\delta\nested\ln\omega^a$. Since $\ln\omega^a$ is a purely infinite number, so is $\delta$. Then $\exp\delta$ is of the form $\omega^b$ for some surreal $b\in\Nobf$. Moreover
		\centre{$\NR(\omega^b) \underset{\text{Proposition }\ref{prop:NRStableExp}}{=} \NR(\delta) < \NR(\ln\omega^a) \underset{\text{Proposition }\ref{prop:NRStableExp}}{=} \NR(\omega^a)$}
		From the induction hypothesis, we have that for any $\delta\nested\ln\omega^a$
		\centre{$\NR(x+\sign(r)\exp(\delta)) = \NR(x)\oplus1\oplus\NR(\exp\delta)$}
		\begin{calculs}
			& \NR(x+r\omega^a) &=& \sup\enstq{\NR(y)+1}{y\in B\cup C}\\
			&&=& \sup\pa{\enstq{\underbrace{\NR(x+\sign(r)\exp\delta)+1}_{\leq\NR(x+\sign(r)\omega^a)}}{\delta\nested \ln\omega^a}\right. \\ &&&\left.\qquad\cup\enstq{\NR(y)+1}{y\in C} \vphantom{\underbrace{\NR(x+\sign(r)\exp\delta)+1}_{\leq\NR(x+\sign(r)\omega^a)}} } \\
			&&=& \sup\enstq{\NR(x)\oplus1\oplus\NR(\exp\delta)\oplus1}{\delta\nested\ln\omega^a}+\ind_{r\neq\pm1}\\
			&&=& \NR(x)\oplus1\oplus\sup\enstq{\NR(\exp\delta)+1}{\delta\nested\ln\omega^a}\oplus\ind_{r\neq\pm1}\\
			& \NR(x+r\omega^a) &=& \NR(x)\oplus1\oplus\NR(\omega^a)\oplus\ind_{r\neq\pm1}
		\end{calculs}
	\end{itemize}
\end{proof}

\begin{lemma}
	\label{lem:NRSommeLogAtomiques}
	Let $x=\aSurreal$ such that for all $i<\nu$, $r_i=\pm 1$ and $\omega^{a_i}=\lambda_i^{\pm1}$ for some $\lambda\in\Lbb$. Then
	\centre{$\NR(x) = \begin{accolade}
			\nu +1 & \nu<\omega\\ \nu & \nu\geq\omega
		\end{accolade}$}
\end{lemma}
\begin{proof}
	If $\nu<\omega$, we just proceed by induction using Lemma \ref{lem:NRAjoutNouveauTerme}. Now we prove by induction the remaining. 
	\begin{itemize}
		\item If $\nu=\omega$. Then 
		\centre{$\NR(x)=\sup\enstq{\NR\pa{\aSurrealPrefix}+1}{\nu'<\nu}=\sup\enstq{\nu'+2}{\nu'<\omega}=\omega$}
		
		\item Assume for $\omega\leq\nu'<\nu$, $\NR\pa{\aSurrealPrefix} = \nu'$. If $\nu$ is a non-limit ordinal, then Lemma \ref{lem:NRAjoutNouveauTerme} concludes. Otherwise
		\centre{$\NR(x)=\sup\enstq{\NR\pa{\aSurrealPrefix}+1}{\nu'<\nu}=\sup\enstq{\nu'+^1}{\omega\leq\nu'<\nu}=\nu$}
	\end{itemize}
\end{proof}

\begin{lemma}
	\label{lem:NRvsNu}
	Let $x=\Sumlt i\nu r_i(x)\omega^{a_i(x)}\in\Nobf$. Then $\nu\leq\NR(x)+1$. The equality stands \tiff $x$ is a finite sum of numbers of the form $\pm y^{\pm1}$ with $y\in\Lbb$ and possibly one non-zero real number.
\end{lemma}

\begin{proof}
	Using induction on $\nu$ it is trivial. For $0$, $\nu=0=\NR(0)$. Now assume $\nu\neq 0$. Then, by definition
	\centre{$\NR(x)+1\geq\sup\enstq{\NR(y)+1}{y\nestedn[0]x\quad y\neq 0}+1\underset{\text{induction hypothesis}}{\geq}
		\sup\enstq{\nu(y)}{y\nestedn[0]x\quad y\neq0} +1 \geq \nu(x)$}
	
	Now assume $\nu(x)=\NR(x)+1$ and write $x=\Sumlt{i}{\nu(x)}r_i\omega^{a_i}$. We use induction on $\Nobf^*$ with the well partial order $\nestedn[0]$.
	
	\begin{itemize}
		\item If $x$ is a monomial, $\nu(x)=1$ and $\NR(x)=0$. That is $x=\pm y^{\pm1}$ for some $y\in\Lbb$ or $x\in\Rbb$ (using Lemma \ref{lem:NR0}).
		
		\item If $x$ is not a monomial.  Assume $r_i\omega^{a_i}\notin\pm\Lbb^{\pm 1}\cup\Rbb^*$ with $i$ minimal for that property. Let $x'=\Sumlt ji r_j\omega^{a_j}$. 
		\begin{itemize}
			\item If $i=0$ then $\NR(r_0\omega^{a_0})\geq 1$. A simple induction shows that $\NR\pa{\Sumlt i{\nu'}r_i\omega^{a_i}}\geq \nu'$ for all $\nu'\leq\nu$. What is a contradiction.
			\item Otherwise $x'\neq 0$ and $x'\nestedn[0]x$. If $\NR(x')+1\neq i$ then $\NR(x')\geq i$ and 
			\centre{$\NR(x)\geq\NR(x') \oplus (\nu\ominus i)\geq \nu$}
			where $\nu\ominus i$ is the ordinal such that $i\oplus(\nu\ominus i)=\nu$.
			what is a contraction. Then by induction hypothesis, $i=\NR(x')+1$ is finite. Now consider
			$y\nested x'+r_i\omega^{a_i}$. Then $y\nestedneq[0] x'$ ($y\nestedn x'$ with $n\geq 1$ is impossible since $x'$ has only terms in $\pm\Lbb^{\pm 1}\cup\Rbb$) or $y=x'+\sign(r_i)\exp(\delta)$ with $\delta\nestedeq \ln(\omega^{a_i})$. Since $r_i\omega^{a_i}\notin\pm\Lbb^{\pm 1}\cup\Rbb$, there is such a $y$ of the later form such that $y\neq x'+r_i\omega^{a_i}$. From Lemma \ref{lem:NRAjoutNouveauTerme}, we have $\NR(y)\geq\NR(x')+1$. Then $\NR(x'+r_i\omega^{a_i})\geq\NR(y)+1\geq\NR(x')+2$. By induction we then can show that 
			\centre{$\NR(x)\geq \NR(x'+r_i\omega^{a_i})\oplus(\nu-(i+1))\geq\NR(x')\oplus2\oplus(\nu\ominus(i\oplus 1)) = i\oplus1+(\nu\ominus(i\oplus1))=\nu$}
			and we get a contradiction.
		\end{itemize}
		Then, every term of $x$ is in $\pm\Lbb^{\pm1}\cup\Rbb$ and by definition only one can be a non-zero real number. It remains to show that there are finitely many terms, what follows from Lemma \ref{lem:NRSommeLogAtomiques}.
	\end{itemize}
\end{proof}

\begin{remark}
	\label{rk:NRvsLength}
	For all $x\in\Nobf$, $\NR(x)\leq \length x$
\end{remark}

\begin{proof}
	Assume the converse and take $x$ with minimal length that contradicts the property then there is $y\nested x$ such that $\NR(y)\geq \length x$. Since $\length x > \length y$, then $y$ reaches contradiction with the minimality of $x$.
\end{proof}

\begin{proposition}[{\cite[Berarducci and Mantova, Proposition 4.29]{Berarducci_2018}}]
	\label{prop:NRComparaisonCoefReel}
	For all $a\in\Nobf^*$, for all $r\in\Rbb\setminus\{\pm1\}$, we have $\NR(r\omega^a)=\NR(\omega^a)+1$.
\end{proposition}
\begin{proposition}[{\cite[Berarducci and Mantova, Proposition 4.30]{Berarducci_2018}}]
	\label{prop:NRComparaisonTerme}
	Let $x=\aSurreal\in\Nobf^*$. Then 
	\begin{itemize}
		\item $\forall i<\nu\qquad \NR(r_i\omega^{a_i})\leq\NR(x)$
		\item $\forall i<\nu\qquad i+1<\nu\Rightarrow \NR(r_i\omega^{a_i})<\NR(x)$
	\end{itemize}
\end{proposition}

We can also say something about the nested truncation rank of a sum of surreal number.

\begin{lemma}
	\label{lem:NRsum}
	For $a,b\in\Nobf$, $\NR(a+b)\leq\NR(a)+\NR(b)+1$ (natural sum of ordinal, which correspond to the surreal sum).
\end{lemma}

\begin{proof}
	We prove it by induction on the couple $(\NR(a),\NR(b))$.
	\begin{itemize}
		\item If $\NR(a)=\NR(b)=0$ then, by Lemma \ref{lem:NR0} both $a,b$ are in $\pm\Lbb^{\pm1}\cup\Rbb$. If $a\in\Rbb$ or $b\in\Rbb$ then $\NR(a+b)\leq 1$ by Lemmas \ref{lem:NRAjoutNouveauTerme} and \ref{lem:NR0}. Otherwise, either $a=\pm b$ and then $\NR(a+b)=0$ or $a\neq\pm b$ and Lemma \ref{lem:NRvsNu} ensure that $\NR(a+b)=1$.
		
		\item Assume the property for all $x,y$ such that 
		$$(\NR(x),\NR(y))<_{lex}(\NR(a),\NR(b))$$
		Then, consider $y\nested a+b$. Write $a+b=\aSurreal$.
		\begin{itemize}
			\item If $y=\aSurrealPrefix$ with $\nu'<\nu$. Let $z_a$ be the series constituted of the terms of $a$ which asolute value is infintely larger than $\omega^{a_\nu'}$. We define the same way $z_b$. Then $y=z_a+z_b$. We have $(\NR(z_a),\NR(z_b))<_{lex}(\NR(a),\NR(b))$ since there is term with order of magnitude $\omega^{a_{\nu'}}$ in either $a$ or $b$. Then, applying induction hypothesis, 
			\centre{$\NR(y)\leq\NR(z_a)+\NR(z_b)+1$}
			
			Since we have at least one of the following inequalities $z_a\nestedn[0]a$ or $z_b\nestedn[0]b$, then $\NR(z_a)+1\leq \NR(a)$ or $\NR(z_b)+1\leq \NR(b)$. In all cases
			\centre{$\NR(y)+1\leq\NR(a)+\NR(b)+1$}
			
			\item If $y=\aSurrealPrefix+\sign(r_{\nu'})\exp(y')$ with $\nu'<\nu$ and $y'\nestedeq \ln\omega^{a_{\nu'}}$ (and $y\nested\ln\omega^{a_{\nu'}}$ if $r_{\nu'}=\pm1$). Let $z_a$ be the series constituted of the terms of $a$ which absolute value is infinitely larger than $\omega^{a_\nu'}$. We define the same way $z_b$. Then $y=z_a+z_b+\sign(r_{\nu'})\omega^{a_{\nu'}}$. 
			Since there is term with order of magnitude $\omega^{a_{\nu'}}$ with the same sign as $r_{\nu'}$ in either $a$ or $b$. Without loss of generality, assume it is $a$. Then $z_a+\sign(r_{\nu'})\exp y'\nestedeq a$. 
			We have 
			$$(\NR(z_a+\sign(r_{\nu'})\exp y'),\NR(z_b))<_{lex}(\NR(a),\NR(b))$$
			otherwise $y=a+b$ what is not the case. Then, applying induction hypothesis, 
			\centre{$\NR(y)\leq\NR(z_a+\sign(r_{\nu'})\exp y')+\NR(z_b)+1$}
			Since we have at least one of the following inequalities $z_a+\sign(r_{\nu'})\exp y'\nested a$ or $z_b\nestedn[0]b$, then we have either 
			\centre{$\NR(z_a+\sign(r_{\nu'})\exp y')+1\leq \NR(a)$}
			\lc{or}{$\NR(z_b)+1\leq \NR(b)$}
			In all cases
			\centre{$\NR(y)+1\leq\NR(a)+\NR(b)+1$}
		\end{itemize}
		Then, for any $y\nested a+b$, $\NR(y)+1\leq\NR(a)+\NR(b)+1$. This proves that
		\centre{$\NR(a+b)\leq\NR(a)+\NR(b)+1$}
	\end{itemize}
\end{proof}

\begin{corollary}
	\label{cor:NRprod}
	For all $a,b\in\Nobf$, $\NR(ab)\leq\NR(a)+\NR(b)+1$.
\end{corollary}

\begin{proof}\ 
	\begin{calculs}
		We have & \NR(ab) &=& \NR(\ln\pa{ab}) & (Proposition \ref{prop:NRStableExp})\\
		&&=&\NR(\ln a + \ln b)\\
		&&\leq&\NR(\ln a) + \NR(\ln b)+1 & (Lemma \ref{lem:NRsum}) \\
		&&\leq&\NR(a)+\NR(b)+1& (Proposition \ref{prop:NRStableExp})\\
	\end{calculs}
\end{proof}

\subsubsection{Paths}
Surreal numbers can be seen as trees. More precisely, it is possible to associate to each surreal number a tree (with an ordinal numbers of node at each layers) whose leaves are labeled by log-atomic numbers or $0$. This gives us some information about the structure of the surreal number. With this notion of tree we can look at the \textbf{paths} from the root (labeled by the surreal number itself) to the leaves that are not labeled by $0$ (actually there is at most one such a leaf). More precisely, the tree associated to a surreal number $x$ is built as follows:
\begin{itemize}
	\item Base case: if $x\in\Lbb$ or $x=0$ just create a node labeled by $x$ and stop the construction.
	
	\item Otherwise: 
		\begin{enumerate}
			\item Put a node at the root and label is t $x$. Write $x$ under the form $x=\Sumlt i\nu r_i\exp(x_i)$ where $r_i\in\Rbb^*$, $\nu$ is an ordinal and $x_i\in\Nobf_\infty$ form a decreasing sequence.
			
			\item For all $i<\nu$ create built the tree for $x_i$ and link its root to $x$ by an edge labeled by $r_i$.
		\end{enumerate}
\end{itemize}

With the a notion, it is possible to have a geometric interpretation of the well partial order $\nested$.

\centre{\includegraphics[scale=.6]{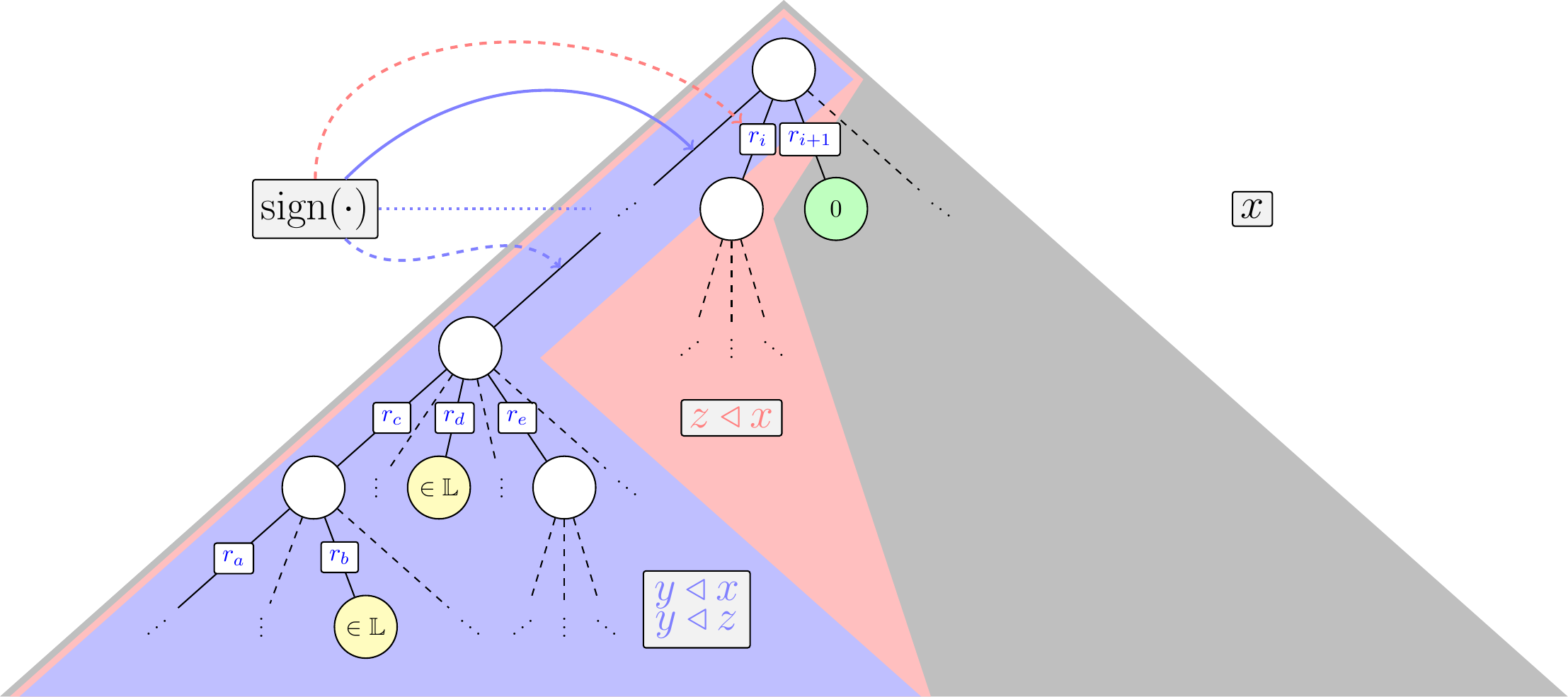}}

The dotted arrows from ``$\sign$'' are to be understood by the fact that we can apply the sign function or not to this arrow. The plain one means that we must apply it. Thanks to this figure we can understand $y\nested x$ by the fact that the tree representation of $y$ is a left-part of the tree representation of $x$.

\begin{remark}
	The reason why we stop the construction on log-atomic numbers is because if we proceed the construction, we would get an infinite path where each node as exactly one child and where every edge is labeled by $1$.
\end{remark}

This notion of tree comes with a notion of path\index{Path} inside the tree.

\begin{definition}\label{def:path}
	Let $x$ be a surreal number. A path $P$ of $x$ is sequence $P:\Nbb\to\Nobf$ such that
	\begin{itemize}
		\item $P(0)$ is a term of $x$
		\item For all $i\in\Nbb$, $P(i+1)$ is an infinite term of $\ln |P(i)|$
	\end{itemize}
	
	We denote $\Pcal(x)$ the set of all paths of $x$.
	
	We also denote $\ell(x)$ to be the purely infinite part of $\ln|x|$. Then $P(i+1)$ is an infinite term of $\ell(P(i))$.
\end{definition}

\begin{definition} The \textbf{dominant path} of $x$ is the path such that
	\begin{itemize}
		\item $P(0)$ is the leading term of $x$
		\item $P(i+1)$ is the leading term of $\ln|P(i)|$.
	\end{itemize}
\end{definition}

In a more graphical point of view, the dominant path of $x$ is the left most path in the tree of $x$ that does not end on the lead $0$. This reduce to the left most path if $x\not\asymp 1$.

\begin{proposition}
	\label{prop:cheminMajorationNu}
	Let $x\in\Nobf$ and $P\in\Pcal(x)$. Then for any $n\in\Nbb$, the length of the serie of $\ell(P(n))$, $\nu(\ell(P(n)))$ satisfies
	\centre{$\nu(\ell(P(n)))\leq \NR(x)+1$}
\end{proposition}

\begin{proof}
	For any $x\in\Nobf$ we write $x=\Sumlt i{\nu(x)}r_i(x)\omega^{a_i(x)}$ in Gonshor's normal form. Now fix $x\in\Nobf$.
	Let $P\in\Pcal(x)$. We set $x_0=x$, and $\alpha_0<\nu(x)$ such $P(0)=r_{\alpha_0}(x)\omega^{a_{\alpha_0}(x_0)}$ and for any natural number $n$,
	\centre{$x_{n+1}=\ln\omega^{a_{\alpha_n}(x_n)} = \ell(P(n))$}
	\lc{and}{$P(n+1)=r_{\alpha_{n+1}}\omega^{a_{\alpha_{n+1}}(x_{n+1})}$}
	Using Proposition \ref{prop:NRStableExp}, we get
	\centre{$\NR(x_{n+1})=\NR\pa{\omega^{a_{\alpha_n}(x_n)}}$}
	By definition $x_{n+1}$ is purely infinite. Then $a_{\alpha_{n+1}}(x_{n+1})>0$ for all natural number $n$. Since $P$ is path, $P(0)\notin\Rbb$ (otherwise $P(1)$ is not defined) and then $a_{\alpha_0}(x_0)\neq 0$. We then can apply Proposition \ref{prop:NRComparaisonCoefReel} and get for all natural number $n$
	\centre{$\NR(x_{n+1})\leq \NR\pa{r_{\alpha_n}(x_n)\omega^{a_{\alpha_n}(x_n)}}$}
	\lc{Now using Proposition \ref{prop:NRComparaisonTerme},}{$\NR(x_{n+1})\leq\NR(x_n)$}
	Then for any natural number $n$ we have $\NR(x_n)\leq \NR(x_0)=\NR(x)$. Applying Lemma \ref{lem:NRvsNu}, we get
	\centre{$\forall n\in\Nbb\qquad \nu(x_n)\leq\NR(x_n)+1\leq\NR(x)+1$}
\end{proof}

\begin{remark}
	Actually, we often have $\nu(\ell(P(n)))\leq\NR(x)$. Indeed, using the notations of the proof and assuming that~$\nu(x_{n+1})=\NR(x)+1$, we have
	\centre{$\NR(x)+1=\nu(x_{n+1})\underset{\text{Proposition \ref{prop:cheminMajorationNu}}}{\leq}\NR(x_{n+1})+1 \leq\cdots\leq \NR(x)+1$}
	Then, all the inequalities are equalities and from Proposition \ref{prop:cheminMajorationNu} we get that $x_{n+1}$ is a finite sum of terms of the form $\pm\Lbb^{\pm1}$, in particular $\nu(x_{n+1})<\omega$ and $\NR(x)$ is finite.
\end{remark}

\subsection{Derivative of a surreal number}

\begin{definition}[Summable family]
	\label{def:summableFam}\index{Summable family!Definition \ref{def:summableFam}}
	Let $\famille{x_i}iI$ be a family of surreal numbers. For $i\in I$ write 
	$$
	x_i=\Sumin a\Nobf r_{i,a}\omega^a
	$$The family $\famille{x_i}iI$ is \textbf{summable} \tiff
	\begin{enumerate}[label=(\roman*)]
		\item $\Unionin iI\supp x_i$ is a reverse well ordered set.
		\item For all $a\in\Unionin iI\supp x_i$, $\enstq{i\in I}{a\in\supp x_i}$ is a finite set.
	\end{enumerate}
	In this case, its sum is defined as $\Sumin iI x_i=\Sumin a\Nobf s_a\omega^a$ where for all $a\in\Nobf$,
	$$
	s_a=\Sum{i\in I\tq a\in\supp x_i}{} r_{i,a}
	$$
	which is a finite sum.
\end{definition}

\begin{definition}[{\cite[Berarducci and Mantova, Definition 6.1]{Berarducci_2018}}]
	A derivation $D$ over a totally ordered exponential (class)-field $\Kbb\supseteq\Rbb$ is a function $D:\Kbb\to\Kbb$ such that
	\begin{enumerate}[label={\textbf{D\arabic*.}}]
		\item\label{ax:D1} \lcr{It satisfies}{$\forall x,y\in\Kbb\qquad D(xy)=xD(y)+D(x)y$}{(Liebniz Rule)}
		
		\item\label{ax:D2} \lcr{If $\famille{x_i}iI$ is summable,}{$D\pa{\Sum{i\in I}{}x_i} = \Sum{i\in I}{}D(x_i)$}{(Strong additivity)\index{Strong additivity}}
		
		\item\label{ax:D3} $\forall x\in\Kbb\qquad D(\exp x)=D(x)\exp x$ 
		
		\item\label{ax:D4} $\ker D = \Rbb$
		
		\item $\forall x>\Nbb\qquad D(x)>0$
	\end{enumerate}
\end{definition}

\begin{remark}
	We can replace Axiom \ref{ax:D2} by
	\begin{enumerate}[label={\textbf{D\arabic*'.}}]
		\setcounter{enumi}{1}
		\item\label{ax:D2p} If $\famille{x_i}iI$ is summable and $\famille{r_i}iI$ is a family of real numbers,
		\lcr{}{$D\pa{\Sum{i\in I}{}r_i x_i} = \Sum{i\in I}{}r_iD(x_i)$}{(Strong lineraity)\index{Strong lineraity}}
	\end{enumerate}
	Indeed, we have
	\lc{}{$
		\ref{ax:D2p}\implies\ref{ax:D2}\qqandqq \ref{ax:D1}\wedge \ref{ax:D2}\wedge \ref{ax:D4}\implies\ref{ax:D2p}$}
\end{remark}

Berarducci and Mantova \cite{Berarducci_2018} provided a general way to define derivation over the class-field $\Nobf$. We recall quickly some of their results.

\begin{proposition}[{\cite[Berarducci and Mantova, Proposition 6.4]{Berarducci_2018}}]
	\label{prop:compareDerivative}
	We have the following properties for a derivation $D$:
	\begin{itemize}
		\item $\forall x,y\in\Kbb\qquad 1\not\asymp x\succ y\Rightarrow D(x)\succ D(y)$
		\item $\forall x,y\in\Kbb\qquad 1\not\asymp x\sim y\Rightarrow D(x)\sim D(y)$
		\item $\forall x,y\in\Kbb\qquad 1\not\asymp x\asymp y\Rightarrow D(x)\asymp D(y)$
	\end{itemize}
\end{proposition}

If $\Kbb\subseteq\Nobf$ is stable under $\exp$ and $\ln$, we can get a nice property satisfied by a general derivation.

\begin{proposition}[{\cite[Berarducci and Mantova, Proposition 6.5]{Berarducci_2018}}]
	Let $\Kbb\subseteq\Nobf$ be a field of surreal number stable by $\exp$ and $\ln$. Let $D$ be a derivation over $\Kbb$. For all $x,y>\Nbb$ such that $x-y>\Nbb$, 
	$$
	\ln D(x) - \ln D(y) \prec x-y \preceq \max(x,y)
	$$
\end{proposition}

To define the derivation, Berarducci and Mantova started by defining it on log-atomic numbers and then extending it on all surreal numbers. More precisely, a derivation on log-atomic number must satisfy the following:

\begin{definition}[{\cite[Berarducci and Mantova, Definition 9.1]{Berarducci_2018}} Prederivation]
	Let $\Kbb$ be a field of surreal numbers stale under $\exp$ and $\ln$ and  such that for all $x\in\Kbb$, for all path $P\in\Pcal(x)$, for all $k\in\Nbb$, if $P(k)\in\Lbb$, then $P(k)\in\Kbb$. A \textbf{prederivation} over $\Kbb$ is a function $D_\Lbb:\Lbb\cap\Kbb\to\Kbb$ such that
	\begin{itemize}
		\item[\textbf{D3.}] $\forall \lambda\in\Lbb\cap\Kbb\qquad 
		D_\Lbb\exp\lambda = (D_\Lbb\lambda)\exp\lambda$
		\item[\textbf{PD1.}] For all $\lambda\in\Lbb\cap\Kbb$, $D_\Lbb\lambda$ is a positive term.
		
		\item[\textbf{PD2.}] $\forall\lambda,\mu\in\Lbb\cap\Kbb\qquad 
		\ln D_\Lbb\lambda - \ln D_\Lbb\ln\mu \prec \max(\lambda,\mu)$
	\end{itemize}
\end{definition}

They key notion to define the derivation from the the prederivation is the notion of \textbf{path derivative}. This notion look at all the paths of the surreal number to say how it contributes to the derivative of the surreal number.

\begin{definition}[{\cite[Berarducci and Mantova, Definition 6.13]{Berarducci_2018}} Path derivative]
	\label{def:pathDeriv}\index{Path derivative!Definition  \ref{def:pathDeriv}}
	Let $P$ be a path. We define the \textbf{path derivative} $\partial P\in\Rbb\omega^\Nobf$ by
	\centre{$\partial P = \begin{accolade}
			P(0)\cdots P(k-1)D_\Lbb P(k) & P(k)\in\Lbb \\ 0 & \forall k\in\Nbb\quad P(k)\notin\Lbb
		\end{accolade} $}
	We denote $\Pcal_\Lbb(x)=\enstq{P\in\Pcal(x)}{\partial P\neq 0}$, which is the set of paths that indeed reach log-atomic numbers at some point.
\end{definition}

One can notice that for any $P\in\Pcal_\Lbb(x)$, $\partial P=r\omega^a$ for some $r\in\Rbb^*$ and $a\in\Nobf$. Indeed, every $P(k)$ is a term and $D_{\Lbb}P(k)$, when $P(k)\in\Lbb$ is an exponential of a purely infinite number, hence, it is a monomial.
For $P\in\Pcal_\Lbb(x)$ there is a minimum $k_P\in\Nbb$ such that $P(k_P)\in\Lbb$. Then $P$ is entirely determined by $P(0),\dots,P(k_P)$. 
We then define $\alpha_0(P),\dots,\alpha_{k_P}(P)$ as follows :
\begin{itemize}
	\item Writing $x=\Sumlt i{\nu(x)} r_i(x)\omega^{a_i(x)}$, then define $\alpha_0(P)<\nu(x)$ such that $P(0)=r_{\alpha_0(P)}(x)\omega^{a_{\alpha_0(P)}}$.
	
	\item For $0\leq i<k$, write $P(i)=r\omega^a$. Then $P(i+1)$ is a term of $\ln\omega^{a}$. Write $\ln\omega^a=\Sumlt i{\nu(a)}r_i(a)\omega^{h(a_i(a))}$. Then set $\alpha_{i+1}(P)$ such that $P(i+1)=r_{\alpha_{i+1}(P)}(a)\omega^{h(a_{\alpha_{i+1}(P)}(a))}$
\end{itemize}

Using Proposition \ref{prop:cheminMajorationNu}, we get that $\suite{\alpha_i(P)}i{\intn0{k_P}}$ is a finite sequence over ordinal less than $\NR(x)+1$. In particular, we can give $\Pcal_\Lbb(x)$ a lexicographic order inherited from the one over finite sequences. 

\begin{definition}
	We define the order $<_{lex}$ on paths by
	\centre{$P<_{lex}Q\Longleftrightarrow 
		(\alpha_0(P),\dots,\alpha_{k_P}(P))<_{lex} (\alpha_0(Q),\dots,\alpha_{k_Q}(Q))$}
\end{definition}

This order will be useful later when we will try to understand better what is going on to get some bounds about the derivatives. For now, the path-derivative being defined, we can recall a theorem by Berarducci and Mantova which explains how to build a general derivation from a prederivation.

\begin{lemma}[{\cite[Berarducci and Mantova, Corollary 6.17]{Berarducci_2018}}]
	\label{lem:partialPQprec}
	Let $P,Q\in\Pcal(x)$ such that $\partial P,\partial Q\neq0$. If there is $i\in\Nbb$ such that
	\begin{enumerate}
		\item $\forall j\leq i\qquad P(i)\preceq Q(i)$
		\item $P(i+1)$ is not a term of $\ell(Q(i))$, 
	\end{enumerate}
	\lc{then}{$\partial P\prec\partial Q$}
\end{lemma}

\begin{lemma}[{\cite[Berarducci and Mantova, Lemma 6.18]{Berarducci_2018}}]
	\label{lem:NRpath}
	Given $P\in\Pcal(x)$ a path of $x$ we have for all $i$ $\NR(P(i+1))\leq\NR(P(i))$ with equality if and only if $P(i)$ is the last term of $\ell(P(i))$. We also have $\NR(P(0)\leq\NR(x)$ with equality if and only if $P(0)$ is the last term of $x$.)
\end{lemma}

\begin{theorem}[{\cite[Berarducci and Mantova, Proposition 6.20, Theorem 6.32]{Berarducci_2018}}]\label{thm:prederivToDeriv}
	Let $D_\Lbb$ be a prederivation over a surreal field $\Kbb$ stable under $\exp$ and $\ln$. Then $D_\Lbb$ extends to a derivation $\partial:\Kbb\to\Nobf$ such that 
	$$
	\forall x\in\Kbb\qquad  \partial x = \Sumin P{\Pcal(x)}\partial P
	$$
	In particular, $\famille{\partial P}P{\Pcal(x)}$ is summable (see Definition \ref{def:summableFam}).
\end{theorem}

The study would not be complete without an example. Berarducci and Mantova provided such a derivation and even more: it is the simplest in some sense.

\begin{definition}[{\cite[Berarducci and Mantova, Definition 6.7]{Berarducci_2018}}]
	\label{def:partialLbb}
	We define $\partial_\Lbb:\Lbb\to\Nobf$ by
	$$	
	\forall\lambda\in\Lbb\qquad \partial_\Lbb\lambda = \exp\pa{-\Sum{\alpha\in\Ord|\kappa_{-\alpha}\succeq^K\lambda}{}\ \Sum{n=1}{\pinf}\ln_n\kappa_\alpha + \Sum{n=1}\pinf \ln_n\lambda}
	$$
\end{definition}

For example, we have:
	\begin{align*}
		\partial_\Lbb \omega &= 1 &\partial_\Lbb \exp\omega &= \exp\omega\\
		\partial_\Lbb \ln\omega &= \exp(-\ln\omega) = \f1\omega 
		&\partial_\Lbb\ln_n\omega &= \f1{\Prod{k=0}{n-1}\ln_k\omega}\\
		\partial_\Lbb\kappa_1=\partial_\Lbb\epsilon_0 &= \exp\pa{\Sum{n=1}\pinf\ln_n\kappa_1} &\partial_\Lbb\kappa_{-1} &= \exp\pa{-\Sum{n=1}\pinf \ln_n\omega}
	\end{align*}
In fact, $\kappa_1$ is intuitively $\exp_\omega\omega$. Therefore it is also quite intuitive that $\partial_\Lbb\kappa_1 = \kappa_1\ln(\kappa_1)\ln\ln(\kappa_1)\cdots$. The same happens for $\kappa_{-1}$ which is intuitively $\ln_\omega\omega$. We indeed have $\partial_\Lbb\kappa_{-1} = \f1{\omega\ln(\omega)\ln\ln(\omega)\cdots}$.

\begin{proposition}[{\cite[Berarducci and Mantova, Propositions 6.9 and 6.10]{Berarducci_2018}}]
	$\partial_\Lbb$ is a prederivation. 
\end{proposition}

The previous proposition ensures that the associated function $\partial$ defined by Theorem \ref{thm:prederivToDeriv} is indeed a derivation over surreal numbers. It turns out that it the simplest for the order $\sqsubseteq$.

We now explain what is meant when saying that $\partial$ is the simplest derivation. In fact, we mean that $\partial_\Lbb$ is the simplest prederivation with respect to the order $\sqsubseteq$.

\begin{theorem}[{Berarducci and Mantova\cite[Theorem 9.6]{Berarducci_2018}}]
	Let $D_\lambda$ be a prederivation. Let $\lambda\in\Lbb$, minimal (in $\Lbb$) for $\sqsubseteq$ such that $D_\Lbb\lambda \neq \partial_\Lbb\lambda$. Then $\partial_\Lbb\lambda\sqsubset D_\Lbb\lambda$. 
\end{theorem}

\subsection{A first bound about the derivative}
We give here some bound on the length of the series of a derivative.

\propmajorationNuPartial*

\begin{proof}
	We know that $\famille{\partial P}P{\Pcal(x)}$ is summable (see Definition \ref{def:summableFam})\index{Summable family}. In particular $\famille{\partial P}P{\Pcal_\Lbb(x)}$ is summable. By definition of summability (in this context) for any $P\in\Pcal_\Lbb(x)$, there are finitely many $Q\in\Pcal_\Lbb(x)$ such that $\partial P\asymp \partial Q$. 
	
	By definition of summability, $<_\Pcal$ is a well total order over $\Pcal_\Lbb(x)$ and if $\beta$ is its order type, then $\omega\otimes\nu(\partial x)<\beta$ (usual ordinal product). Then, to complete the proof, we just need to show that $\beta<\omega^{\omega^{\omega(\NR(x)+1)}}$. We proceed by induction on $\NR(x)$.	
	\begin{itemize}
		\item \underline{$\NR(x)$=0} : then $x=0$ or $x=\pm y^{\pm 1}$ for some $y\in\Lbb$ and $\nu\pa{\partial x} \leq 1 <\omega^{\omega^\omega}$ and we conclude the proof.
		
		\item Assume that for any $y$ such that $\NR(y)<\NR(x)$, $\Pcal_\Lbb(y)$ has order type less than $\omega^{\omega^{\omega(\NR(y)+1)}}$. Assume for contradiction that $\beta\geq\omega^{\omega^{\omega(\NR(x)+1)}}$. Then for any multiplicative ordinal $\mu<\omega^{\omega^{\omega(\NR(x)+1)}}$, there is some $P_\mu\in\Pcal_\Lbb(x)$, minimum with respect to $<_{lex}$, such that the set
		\centre{$\Ecal_\mu(x)=\enstq{Q\in\Pcal_\Lbb(x)}{Q<_\Pcal P_\mu}$}
		has order type $\beta_\mu\geq\mu$. Let us select any $\mu$ such that $\mu\geq\omega^{\omega^{\omega\NR(x)+1}}$. Now define
		\begin{calculs}
			& \Ecal_\mu^{(1)}(x) &=& \enstq{Q\in\Pcal_\Lbb(x)}{Q<_\Pcal P_\mu\quad Q<_{lex} P_\mu}\\ [.4cm]
			& \Ecal_\mu^{(2)} &=& \enstq{Q\in\Pcal_\Lbb(x)}{Q >_{lex} P_\mu}
		\end{calculs}
		\lc{Theses sets are disjoints and}{$\Ecal_\mu=\Ecal_\mu^{(1)}\cup\Ecal_\mu^{(2)}$}
		Let $\beta_\mu^{(i)}$ be the order type of $\Ecal_\mu^{(i)}$. We then have
		\centre{$\mu\leq\beta_\mu\leq\beta_\mu^{(1)}+\beta_\mu^{(2)}$}
		where the addition is the surreal addition of ordinal numbers. Since $\mu$ is multiplicative ordinal, hence, an additive one, at least one of the $\beta_\mu^{(i)}\geq\mu$.
		\begin{itemize}
			\item \underline{First case }: $\beta_\mu^{(2)}\geq\mu$. Since $\mu$ is additive, there is an $i\in\intn0{k_P}$ such that the well ordered set
			\centre{$\Ecal_\mu^{(2,i)}=\enstq{Q\in\Ecal_\mu^{(2)}}{\forall j<i\ Q(j)=P_\mu(j)\quad Q(i)\prec P_\mu(i)}$}
			has order type at least $\mu$. We take such an $i$. For $Q\in\Ecal_\mu^{(2,i)}$, we consider the path $Q'(n)=Q(n+i+1)$. Since $\partial Q\succeq \partial P_\mu$, Lemma \ref{lem:partialPQprec} gives us that $Q(i+1)$ is a term of $\ell(P_\mu(i))$.  We then have $Q'\in\Pcal\pa{\ell(P_\mu(i))}$ and 
			\centre{$\partial Q' = \f{\partial Q}{Q(0)\cdots Q(i)}=\f{\partial Q}{P_\mu(0)\cdots P_\mu(i-1)Q(i)}$}
			In particular $Q'\in\Pcal_\Lbb\pa{\ell(P_\mu(i))}$. 
			Since $Q(i)\prec P_\mu(i)$, $P_\mu(i)$ is not the last term of $\ell(P_\mu(i-1))$ (or $x$ if $i=0$). Then Proposition \ref{prop:NRComparaisonTerme} ensures that 
			\centre{$\NR(\ell(P_\mu(i)))\leq \NR(P_\mu(i))<\NR(x)$}
			Applying the induction hypothesis on $\ell(P_\mu(i))$, the order type of $\Pcal_\Lbb(\pa{\ell(P_\mu(i)))}$ has order type $\gamma$ such that
			\centre{$\gamma<\omega^{\omega^{\omega\pa{\NR(\ell(P(i)))+1}}}\leq\omega^{\omega^{\omega\NR(x)}}<\omega^{\omega^{\omega\NR(x)+1}}\leq\mu$}
			For $Q,R\in\Ecal_\mu^{(2,i)}$, $Q<_\Pcal R$ \tiff 
			\begin{calculs}
				& (Q(i)\partial Q' \succ R(i)\partial R') &\vee& (Q(i)\partial Q'\asymp R(i)\partial R' \wedge Q(i)\partial Q' > R(i)\partial R')\\
				&&& \vee (Q(i)\partial Q' = R(i)\partial R\wedge Q<_{lex}R)
			\end{calculs}
			what we can also write
			\begin{calculs}
				& Q<_\Pcal R &\Leftrightarrow&
				\big(\ell(Q(i)) + \ell(\partial Q') > \ell(R(i))+\ell(\partial R')\big)\\
				&&&\vee \big(\ell(Q(i))+\ell(\partial Q')=\ell(R(i))+\ell(\partial R') \wedge Q(i)\partial Q'>R(i)\partial R'\big)\\
				&&&\vee (Q(i)\partial Q' = R(i)\partial R\wedge Q<_{lex}R) 
			\end{calculs}
			where the two later cases occur finitely may times for $Q$ or $R$ fixed. 
			Let $\delta$ denote the order type of the possible values for $Q(i)$ and $\beta_\mu^{(2,i)}$ the order type of $\Ecal_\mu^{(2,i)}$. Since $\ell$ is non-decreasing, the set $\enstq{\ell(\partial Q')}{Q\in\Ecal_\mu^{(2,i)}}$ has order type at most $\gamma$ and the set $\enstq{\ell(Q(i))}{Q\in\Ecal_\mu^{(2,i)}}$ has order type at most $\NR(x)$. Using Proposition \ref{prop:sommeEnsBienOrd},

			\centre{$\beta_\mu^{(2,i)}\leq (\gamma\NR(x))\otimes\omega<\mu$}
			\lc{Finally}{$\mu\leq\beta_\mu^{(2,i)}<\mu$}
			and we reach the contradiction.
			
			\item \underline{Second case }: $\beta_\mu^{(2)}<\mu$. Then $\beta_\mu^{(1)}\geq\mu$.  Let us define for $i\in\intn0{k_P}$ 
			\centre{$\Ecal_\mu^{(1,i)}=\enstq{Q\in\Ecal_\mu^{(1)}}{\forall j<i\ P_\mu(j)=Q(j)\quad P_\mu(i)\prec Q(i)}$}
			Since there are finitely many of them, that they form a partition of $\Ecal_\mu^{(1)}$ and $\mu$ is multiplicative, hence additive, there is at least one of them which has order type at least $\mu$. We consider such an $i\in\intn0{k_P}$. Now define
			\centre{$x_j = \begin{accolade}
					x & i=j \\ \ell(P(i-j-1)) & j<i 
				\end{accolade}$} 
			Writing $x_0=\Sumlt n{\nu(x_0)} r_n(x_0)\omega^{a_n(x_0)}$ and $P_\mu(i)=r_{\alpha_0}(x_0)\omega^{a_{\alpha_0}(x_0)}$ we set
			\centre{$y_0 = \Sumlt n{\alpha_0} r_n(x_0)\omega^{a_n(x_0)}$}
			Now for $0\leq j<i$, we define $y_{j+1}$ has follows. $P_\mu(i-j-1)$ is a term of $x_{j+1}$. Write $P_\mu(i-j-1)=r_{\alpha_{j+1}}(x_{j+1})\omega^{a_{\alpha_{j+1}}(x_{j+1})}$ for some ${\alpha_{j+1}}<\nu(x_{j+1})$. Then set
			\centre{$y_{j+1} = \Sumlt n{\alpha_{j+1}} r_n(x_{j+1})\omega^{a_n(x_{j+1})} + \sign(r_{\alpha_{j+1}}(x_{j+1})) \exp(y_j)$}
			Denote $y=y_i$. For $Q\in\Ecal_\mu^{(1,i)}$. For any $Q\in\Ecal_\mu^{(1,i)}$ we will build $Q'\in\Pcal_\Lbb(y)$. We expect to use the induction hypothesis on $y$. First we prove that $\NR(y)<\NR(x)$. In fact, by trivial induction, we have $y_j\nestedn[j] x_j$. So $y\nestedn[i] x$ and by definition of $\NR$ we have $\NR(y)<\NR(x)$. Now consider the path $Q'$ defined as follows :
			\begin{itemize} 
				\item $\forall j<i\qquad  Q'(j)=\sign(r_{\alpha_j}(x_{i-j}))\exp(y_{i-j-1})$
				\item $\forall j\geq i\qquad Q'(j)=Q(j)$
			\end{itemize}
			
			We then have $Q'\in\Pcal(y)$. We can even say $Q'\in\Pcal_\Lbb(y)$. Moreover, since we change only the common terms of the path, and the changes do not depend on $Q$, we have
			\centre{$\forall Q,R\in\Ecal_\mu^{(1,i)}\qquad Q<_\Pcal R\Leftrightarrow Q'<_\Pcal R'$}
			We then have an increasing function
			\centre{$\fct{\Phi}{\Ecal_\mu^{(1,i)}}{\Pcal_\Lbb(y)}{Q}{Q'}$}
			The induction hypothesis give that the order type of $P'(y)$ is less than $\omega^{\omega^{\omega(\NR(y)+1)}}$. Then
			\centre{$\omega^{\omega^{\omega\NR(x)}}\leq \mu < \omega^{\omega^{\omega(\NR(y)+1)}}\leq\omega^{\omega^{\omega\NR(x)}}$}
			and we get the contradiction.
		\end{itemize}
		This completes the proof.
	\end{itemize}
\end{proof}

\begin{corollary}
	\label{cor:majorationNuPartial}
	If $\NR(x)<\lambda$ then $\nu(\partial x)<\lambda$
\end{corollary} 

\begin{proposition}
	\label{prop:majorationNRPartial}
	For all $x\in\Nobf$, let $\alpha$ the minimum ordinal such that $\kappa_{-\alpha}\prec^K t$ for all log-atomic $t$ such that there is some path $P\in\Pcal_\Lbb(x)$ and some index $k\in\Nbb$ such that $P(k)=t$. Then, for all path $P$, 
	\centre{$\NR(\partial P)\leq  k(\NR(x)+1)+\omega(\alpha+1) $}
	\lc{and}{$\NR(\partial x)\leq \omega(\NR(x)+\alpha+2)\otimes \nu(\partial x) \leq\omega^{\omega^{\omega(\NR(x)+1)}+\alpha}$} 
\end{proposition}

\begin{proof}
	Let $P$ be a path of such that $\partial P\neq 0$. Then there is some $k\in\Nbb$ such that $\partial P=P(0)\cdots P(k-1)\partial_\Lbb P(k)$. With Corollary \ref{cor:NRprod}, we get
	\begin{calculs}
		& \NR(\partial P) &\leq& \Sum{i=0}{k-1}\NR(P(i)) + \NR(\partial_\Lbb P(k))+k\\
		&&\leq& k\NR(x)+\NR(\partial_\Lbb P(k))+k&(Lemma \ref{lem:NRpath})\\
		&&\leq& k\NR(x)+k+\NR\pa{\exp\pa{-\Sum{\kappa_{-\beta}\succeq^K P(k)}{}\Sum{n\geq 1}{} \ln_n\kappa_{-\beta} + \Sum{n\geq 1}{}\ln_nP(k)}}\\
		&&\leq& k\NR(x)+k+\NR\pa{-\Sum{\kappa_{-\beta}\succeq^K P(k)}{}\Sum{n\geq 1}{} \ln_n\kappa_{-\beta} + \Sum{n\geq 1}{}\ln_nP(k)} & \\&&&&(Proposition \ref{prop:NRStableExp})\\
		&&\leq& k\NR(x)+k+(\omega\otimes(\alpha\oplus 1)) &(Lemma \ref{lem:NRSommeLogAtomiques})\\
		&&\leq& \omega(\NR(x)+1)+\omega(\alpha+1)
	\end{calculs}
	This bound does not depend on $P$. Then applying Proposition \ref{prop:majorationNuPartial} and Lemma \ref{lem:NRAjoutNouveauTerme} we get
	\begin{calculs}
		&\NR(\partial x) &\leq& (\omega(\NR(x)+\alpha+2))\otimes \nu(\partial x)\\
		&&<& (\omega(\NR(x)+\alpha+2)) \otimes \omega^{\omega^{\omega(\NR(x)+1)}}\\
		&&\leq&\omega^{\omega^{\omega(\NR(x)+1)}+\alpha}
	\end{calculs}
\end{proof}

\subsection{Anti-derivative of a surreal number}

Berarducci and Mantova provided a derivation, $\partial$. An other strong property of this derivation is that is as a compositional inverse, an anti-derivation. The first thing to prove it is to prove that there is an asymptotic anti derivation.

\begin{proposition}[{\cite[Berarducci and Mantova, Proposition 7.4]{Berarducci_2018}}]
	There is a class function $A : \Nobf^*\to \Rbb\omega^{\Nobf^*}$ such that 
	\centre{$\forall x\in\Nobf^*\qquad x\sim\partial A(x)$}
\end{proposition}

Basically, $A(x)$ is the leading term of $x\f{xu/\partial u}{\partial(xu/\partial u)}$ where $u=\kappa_\alpha$ for $\alpha$ a sufficiently large ordinal. Actually we can be even more precise and give a more explicit formula for Berarducci and Mantova's asymptotic anti-derivation \cite{Berarducci_2018}.

\begin{lemma}
	\label{lem:antiDerivee1}
	Let $u=\ln_n\kappa_{-\alpha}$ for some $n\in\Nbb$ and some ordinal $\alpha$. Let $x=\partial u\exp\epsilon$. If $\epsilon\succ \ln u$, then
	\centre{$\partial\pa{\f x{\partial\epsilon}}\sim x$}
\end{lemma}

\begin{proof}
	Let $y=\f x{\partial\epsilon} = \f{\partial u}{\partial\epsilon}\exp(\epsilon)$. Since $\epsilon\succ\ln u$, Proposition \ref{prop:compareDerivative} ensures that $\partial\epsilon \succ \f{\partial u} u$.  Then, $\f{\partial u}{\partial\epsilon}\prec u\not\asymp 1 $
	\begin{calculs}
		& \partial y &=& \f{\partial u}{\partial\epsilon}\partial\epsilon\exp(\epsilon) + \partial\pa{\f{\partial u}{\partial\epsilon}}\exp(\epsilon) \\
		
		&&=& x + \partial\pa{\f{\partial u}{\partial\epsilon}}\exp(\epsilon)\\ 
	\end{calculs}
	
	Proposition \ref{prop:compareDerivative} gives that $\partial\pa{\f{\partial u}{\partial\epsilon}}\prec\partial u$. Then
	\centre{$\partial y\sim x$}
\end{proof}

\begin{lemma}
	\label{lem:antiDerivee2}
	Let $u=\ln_n\kappa_{-\alpha}$ for some $n\in\Nbb$ and some ordinal $\alpha$. Let $x=\partial u\exp(\epsilon)$. If $\epsilon\sim r\ln u$ for some $r\in\Rbb\setminus\{0,-1\}$, then
	\centre{$\partial\pa{\f1{r+1}\f {ux}{\partial u}}\sim x$}
\end{lemma}

\begin{proof}
	Let us compute the above derivative.
	\centre{$\partial\pa{\f1{r+1}\f {ux}{\partial u}} =
		\partial\left(\f{u\exp(\epsilon)}{r+1}\right) = 
		\f x{r+1} + \f{u\partial\epsilon\exp(\epsilon)}{r+1}$}
	Using Proposition \ref{prop:compareDerivative}, we get that $\partial\epsilon \sim\partial(r\ln u) =  r\f{\partial u}u$. Then, since $r\neq -1$, we get that
	\centre{$\partial\pa{\f1{r+1}\f {ux}{\partial u}}\sim x$}
\end{proof}

\begin{lemma}
	\label{lem:antiDerivee3}
	Let $u=\ln_n\kappa_{-\alpha}$ for some $n\in\Nbb$ and some ordinal $\alpha$. Let $x=\partial u\exp(\epsilon)$. If $\epsilon\prec \ln u$, then
	\centre{$\partial\pa{\f {ux}{\partial u}}\sim x$}
\end{lemma}

\begin{proof}
	Let us compute the above derivative.
	\centre{$\partial\pa{\f {ux}{\partial u}} =
		\partial\left(u\exp(\epsilon)\right) = 
		x + u\partial\epsilon\exp(\epsilon)$}
	Using Proposition \ref{prop:compareDerivative}, we get that $\partial\epsilon\prec\partial\ln u =  \f{\partial u}u$. Then, $u\partial\epsilon\exp(\epsilon)\prec x$ and we get that
	\centre{$\partial\pa{\f {ux}{\partial u}}\sim x$}
\end{proof}

\begin{theorem}
	\label{thm:formeAntiDeriveeAsymp}
	Let $x$ be a term. Write $|x|=\partial u\exp(\epsilon)$ with $u=\ln_n\kappa_{-\alpha}=\lambda_{-\omega\alpha-n}$ with $\omega\alpha+n$ such minimal that $\epsilon\not\sim-\ln u$. Then,
	\centre{$A(x)\sim\begin{accolade}
			\f x{\partial\epsilon} & \epsilon \succ\ln u\\
			\f{ux}{(r+1)\partial u} & \epsilon\sim r\ln u\quad r\neq 0,-1\\
			\f{ux}{\partial u} & \epsilon\prec\ln u
		\end{accolade}$}
\end{theorem}

In this theorem, the quantities $\kappa_a$ and $\lambda_a$ are defined in Definitions \ref{def:kappaMap} and \ref{def:lambdaMap}. 

\begin{proof}
	Since $A(x)=-A(-x)$, we may assume that $x>0$. Then, we just need to apply Lemmas \ref{lem:antiDerivee1}, \ref{lem:antiDerivee2}, and \ref{lem:antiDerivee3}.
\end{proof}

\begin{corollary}
	\label{cor:formeAntiDeriveeAsymp}
	Let $x$ be a non-zero surreal number. Write $|x|=\partial u\exp(\epsilon)$ with $u=\ln_n\kappa_{-\alpha}=\lambda_{-\omega\alpha-n}$ with $\omega\alpha+n$ such minimal that $\epsilon\not\sim-\ln u$. Then,
	\centre{$A(x)= \begin{accolade}
			\f{t}{s} & \epsilon \succ\ln u\\ [.4cm]
			\f{ut}{(r+1)\partial u} & \epsilon=r\ln u + \eta \quad r\neq -1, \eta\prec\ln u
		\end{accolade}$}
	\vskip .2cm
	where $t$ is the leading term of $x$ and $s$ the leading term of $\partial\epsilon$.
\end{corollary}

\begin{proof}
	Just use Theorem \ref{thm:formeAntiDeriveeAsymp} and the definition of $A$.
\end{proof}

We are now ready to build the anti-derivation for surreal numbers.
We start with a useful lemma due to Aschenbrenner, van den Dries and van der Hoeven. We give it in a form that matches our notations.

\begin{definition}\label{def:strongLin}\index{Strongly linear function!Definition \ref{def:strongLin}}
	A function $\Phi$ is strongly linear is for all summable family $\famille{x_i}iI$, 
	$$
	\Phi\pa{\Sumin iI x_i} = \Sumin iI \Phi(x_i)
	$$
\end{definition}

\begin{lemma}[{\cite[Aschenbrenner, van den Dries, van der Hoeven, Corollary 1.4]{vdH:dagap}}]
	\label{lem:inversePhi}
	Let $\Phi$ a strongly linear map defined over a field $\Kbb$ of surreal numbers.  Assume that for any monomial
	$\omega^a\in\Kbb$, we have $\Phi(\omega^a)\prec\omega^a$. Then $\Sumin n\Nbb\Phi^n(x)$ makes sense as a surreal number (\ie{} $\famille{\Phi^n(x)}n\Nbb$ is summable\index{Summable family}) and if it belongs to $\Kbb$ for all $x$, we have
	\centers{$(\id-\Phi)^{-1}=\Sumin n\Nbb\Phi^n$}
\end{lemma}

\begin{definition}
	\label{def:phi}
	We define an extension of $A$, denoted $\Acal$, to all surreal numbers by
	
	\centre{$\Acal\pa{\aSurreal} = \Sumlt i\nu r_i A(\omega^{a_i})$}
	
	We also introduce the function $\Phi=\id-\partial\circ\Acal$.
\end{definition}

Proposition \ref{prop:compareDerivative} ensures that the function $\Acal$ is well defined. Moreover, this function is obviously strongly linear\index{Strongly linear function}. We now consider, given a surreal number $x$, the sequence
\centre{$\begin{accolade}
		&x_0 = x\\
		&x_{n+1} = x_n-\partial\Acal(x_n) = \Phi(x_n)
	\end{accolade}$}

Note that if $\omega^a=\partial u\exp\epsilon$ with $u=\lambda_{-\omega\otimes\alpha-n} =\ln_n\kappa_{-\alpha}$ and $\epsilon\prec \ln \lambda_{-\omega\otimes\beta-m}$ for $\omega\otimes\beta+m<\omega\otimes\alpha+n$, and $\omega\otimes\alpha+n$
maximum for that property, we have
\centre{$\Phi(\omega^a) = \begin{accolade}
		\pa{1-\f{\partial\epsilon}s}\omega^a - \partial\pa{\f{\partial u}s}\exp\epsilon 
		& \epsilon\succ\ln u\quad s\text{ dominant term of }\partial\epsilon\\
		\f{\omega^a}{r+1}\partial\eta\f u{\partial u} & \epsilon=r\ln u+\eta\quad r\neq -1
	\end{accolade}$}

\begin{corollary}
	\label{cor:existencePrimitive}
	The operator $\id-\Phi$ is invertible with inverse $\Sum{i\in\Nbb}{}\Phi^{i}$. Moreover $\Acal\circ\Sum{i\in\Nbb}{}\Phi^i$ is an operator that sends every $x$ to some anti-derivative of $x$.
\end{corollary}

\begin{proof}
	Lemma \ref{lem:inversePhi} ensure that $\id-\Phi$ has a inverse expressed by $\Sumin i\Nbb\Phi^i$. We also have that $\id-\Phi=\partial\circ\Acal$. Then,
	\centre{$\partial\circ\pa{\Acal\circ\Sumin i\Nbb\Phi^i} = (\partial\circ\Acal)\circ\pa{\partial\circ\Acal}^{-1}=\id$}
	In particular, for all $x$, $\pa{\Acal\circ\Sumin i\Nbb\Phi^i}(x)$ is a anti-derivative of $x$.
\end{proof}

\section{Field stable under exponential, logarithm, derivation and anti-derivation}
\label{sec:stable}
\subsection{Length of the series of a derivative of a monomial}

\subsubsection{Case $\epsilon\preceq\ln u$}

\begin{lemma}
	\label{lem:formeEtaPhiOmegaA}
	Assume $x=\omega^a=\partial u\exp\epsilon$ with $\epsilon=r\ln u+\eta$ and $u=\ln_n\kappa_{-\alpha}$. 
	Let $b\in\supp\Phi(\omega^a)$. Then, there is a path $P\in\Pcal_\Lbb(\eta)$ such that
	\begin{calculs}
		& \omega^b &\asymp& \partial u\exp\pa{r\ln u + \eta -\Sum{m=n+2}{\pinf}\ln_m\kappa_{-\alpha}
			-\hspace{-2.5em}\Sum{\footnotesize\begin{array}{c}
					\beta>\alpha, m\in\Nbb^*\\ \beta\tq\kappa_{-\beta}\succeq^KP(k_P)
			\end{array}}{}\hspace{-2.5em} \ln_m\kappa_{-\beta} + \Sum{i=0}{\pinf}\ln |P(i)|}
	\end{calculs}
\end{lemma}

\begin{proof}
	It is just a calculation. First notice that $\f{\omega^a}{r+1}\f u{\partial u}$ is a term as a product of terms. Then, let $b\in\supp \Phi(\omega^a)$. There is path $P$ of $\eta$ such that
	\centers{$\omega^b\asymp\omega^a\f u{\partial u}\partial P = u\partial P\exp(r\ln u + \eta)$}
	\lc{write}{$\partial P=P(0)\cdots P(k_P-1)\partial_\Lbb P(k_P)$}
	
	Since $P(0)$ is a term of $\eta\prec\ln u$, we also have $P(0)\prec\ln u$. Moreover since $\eta$ consists in
	purely infinite  term, so is $P(0)$ and then  $\ln |P(0)|\prec P(0)$. Since $P(1)$ is a purely infinite 
	term of $\ln|P(0)|$, we get that $P(1)\prec P(0)$. By induction, for all $i$, $P(i+1)\preceq P(i)\leq P(0)$.
	In particular, $P(k_P)\preceq P(0)\preceq^k\kappa_{-\alpha}$. Then, $\kappa_{-\alpha}\succeq^K P(k_P)$. That leads to
	\begin{calculs}
		& \partial_\Lbb(P(k_P)) &=& \exp\pa{-\Sum{\beta\leq\alpha,\ m\in\Nbb^*}{}\ln_m\kappa_{-\beta} 
			-\hspace{-2.5em}\Sum{\footnotesize\begin{array}{c}
					\beta>\alpha, m\in\Nbb^*\\ \beta\tq\kappa_{-\beta}\succeq^KP(k_P)
			\end{array}}{}\hspace{-2.5em} \ln_m\kappa_{-\beta} + \Sum{m=1}{\pinf}\ln_mP(k_P)} \\
		& \partial_\Lbb(P(k_P)) &=& \partial u \exp\pa{-\Sum{m=n+1}{\pinf}\ln_m{\kappa_{-\alpha}}
			-\hspace{-2.5em}\Sum{\footnotesize\begin{array}{c}
					\beta>\alpha, m\in\Nbb^*\\ \beta\tq\kappa_{-\beta}\succeq^KP(k_P)
			\end{array}}{}\hspace{-2.5em} \ln_m\kappa_{-\beta} + \Sum{m=1}{\pinf}\ln_mP(k_P)}\\
		Since $P(k_P)\in\Lbb$,\\
		& \partial_\Lbb(P(k_P)) &=& \partial u \exp\pa{-\Sum{m=n+1}{\pinf}\ln_m{\kappa_{-\alpha}}
			-\hspace{-2.5em}\Sum{\footnotesize\begin{array}{c}
					\beta>\alpha, m\in\Nbb^*\\ \beta\tq\kappa_{-\beta}\succeq^KP(k_P)
			\end{array}}{}\hspace{-2.5em} \ln_m\kappa_{-\beta} + \Sum{i=k_P}{\pinf}\ln |P(i)|}\\
		Then & \partial P &=& \partial u\exp\pa{-\Sum{m=n+1}{\pinf}\ln_m{\kappa_{-\alpha}}
			-\hspace{-2.5em}\Sum{\footnotesize\begin{array}{c}
					\beta>\alpha, m\in\Nbb^*\\ \beta\tq\kappa_{-\beta}\succeq^KP(k_P)
			\end{array}}{}\hspace{-2.5em} \ln_m\kappa_{-\beta} + \Sum{i=0}{\pinf}\ln |P(i)|}\\ \end{calculs}Finally, \begin{calculs}
		& \omega^b &\asymp & \partial u\exp(r\ln u+\eta)u\\ &&& \qquad\times\exp\pa{-\Sum{m=n+1}{\pinf}\ln_m{\kappa_{-\alpha}}
			-\hspace{-2.5em}\Sum{\footnotesize\begin{array}{c}
					\beta>\alpha, m\in\Nbb^*\\ \beta\tq\kappa_{-\beta}\succeq^KP(k_P)
			\end{array}}{} \hspace{-2.5em}\ln_m\kappa_{-\beta} + \Sum{i=0}{\pinf}\ln |P(i)|}\\
		& \omega^b &\asymp & \partial u\exp\pa{r\ln u+\eta-\Sum{m=n+2}{\pinf}\ln_m{\kappa_{-\alpha}}
			-\hspace{-2.5em}\Sum{\footnotesize\begin{array}{c}
					\beta>\alpha, m\in\Nbb^*\\ \beta\tq\kappa_{-\beta}\succeq^KP(k_P)
			\end{array}}{}\hspace{-2.5em} \ln_m\kappa_{-\beta} + \Sum{i=0}{\pinf}\ln |P(i)|}\\
	\end{calculs}
\end{proof}

\begin{proposition}\label{prop:supportPhi1}	
	Assume $x=\omega^a=\partial u\exp\epsilon$ with $\epsilon=r\ln u+\eta$ and $u=\ln_n\kappa_{-\alpha}$. 
	We denote for $P_0,\dots,P_k\in\Pcal_\Lbb(\eta)$ and $i_1,\dots,i_k\in\Nbb^*$,
	\begin{calculs}
		&e\left(\begin{array}{c}P_0,\dots,P_k\\ i_1,\dots,i_k\end{array}\right) &=& 
		-(k+1)\Sum{m=n+2}{\pinf}\ln_m{\kappa_{-\alpha}}\\&&&\qquad
		-\Sum{j=0}{k}\hspace{-1em}\Sum{\footnotesize\begin{array}{c}
				\beta>\alpha, m\in\Nbb^*\\ \beta\tq\kappa_{-\beta}\succeq^KP_j(k_{P_j})
		\end{array}}{}\hspace{-2em} \ln_m\kappa_{-\beta}
		+ \Sum{j=0}{k}\Sum{i=i_j}{\pinf}\ln |P_j(i)|
	\end{calculs}
	with $i_0=0$. For $k\in\Nbb$ define $E_{1,k}$ by:
	\begin{calculs}
		& e\left(\begin{array}{c}P_0,\dots,P_k\\ i_1,\dots,i_k\end{array}\right)\in E_{1,k}&\Leftrightarrow& P_0,\dots,P_k\in\Pcal_\Lbb(\eta)\quad\wedge\quad i_1,\dots,i_k\in\Nbb^*\\
		&&&\quad\wedge\quad \forall j\in\intn1k\quad \exists j'\in\intn 0{j-1}\\
		&&&\qquad\qquad \forall i\in\intn0{i_j-1}\qquad P_{j'}(i)=P_j(i)\\
		&&&\quad\wedge\quad \forall j \in \intn1k\\ &&&\qquad\quad \supp P_j(i_j)\subseteq\supp e\left(\begin{array}{c}P_0,\dots,P_{k-1}\\ i_1,\dots,i_{k-1}\end{array}\right)
			
	\end{calculs}
	\begin{calculs}
		& E_1 &=& \Union{k\in\Nbb}{} E_{1,k} \\ [.5cm]
		&E_2 &=& \enstq{\begin{array}{c}-\Sum{m=n+2}{\pinf}\ln_m\kappa_{-\alpha}
			\\-\hspace{-1em}\Sum{\gamma<\beta,\ m\in\Nbb^*}{}\hspace{-1em}\ln_m\kappa_{-\gamma} - \Sum{m=1}{p}\ln_m\kappa_{-\beta}\end{array}}{
			\begin{array}{c}
				\beta>\alpha \\ \exists P\in\Pcal_\Lbb(\eta)\quad \kappa_{-\beta}\succeq^K P(k_P) \\ p\in\Nbb
		\end{array}}\\ [1cm]
		&E_3 &=& \enstq{-\Sum{m=n+2}{p}\ln_m{\kappa_{-\alpha}}}{p\geq n+2}\\ [.4cm]
		& E &=& E_1\cup E_2\cup E_3
	\end{calculs}
	and $\inner E$ be the monoid it generates. Let $b\in\Union{\ell=0}{\pinf}\supp\Phi^\ell(\omega^a)$. Then, there is $y\in\inner E$ such that
	\centre{$\omega^b\asymp\partial u\exp(r\ln u + \eta + y)$}
\end{proposition}

\begin{proof}
	We prove it by induction on $\ell$.
	\begin{itemize}
		\item If $b\in\supp\omega^a$, then $y=0$ works.
		
		\item Assume the property for $\ell\in\Nbb$ and let $b\in\sup\Phi^{\ell+1}(\omega^a)$. Then there is 
		$c\in\supp\Phi^\ell(\omega^a)$ such that $b\in\supp\Phi(\omega^c)$. Apply the induction hypothesis on $c$
		and on $y$ associated to $c$. Since any element $e\in E$ is such that $e\prec\ln u$, we have $y\prec\ln u$
		then 
		Apply Lemma \ref{lem:formeEtaPhiOmegaA} to get that there is $P\in\Pcal_\Lbb(\eta+y)$ such that
		\begin{calculs}
			& \omega^b &\asymp& \partial u\exp\pa{r\ln u + \eta + y -\Sum{m=n+2}{\pinf}\ln_m\kappa_{-\alpha}
				\vphantom{\Sum{\footnotesize\begin{array}{c}
							\beta>\alpha, m\in\Nbb^*\\ \beta\tq\kappa_{-\beta}\succeq^KP(k_P)
					\end{array}}{}}\right.\\
			&&&\qquad\qquad\left. -\hspace{-2em}\Sum{\footnotesize\begin{array}{c}
					\beta>\alpha, m\in\Nbb^*\\ \beta\tq\kappa_{-\beta}\succeq^KP(k_P)
			\end{array}}{}\hspace{-2em} \ln_m\kappa_{-\beta} + \Sum{i=0}{\pinf}\ln |P(i)|}
		\end{calculs}
		If $P(0)$ a term of $\eta$, up to some real factor, then there is a real number $s$ and some $e\in E_{1,0}$ such that
		\centers{$\exp\pa{-\Sum{m=n+2}{\pinf}\ln_m\kappa_{-\alpha} -\hspace{-2em}\Sum{\footnotesize\begin{array}{c}
						\beta>\alpha, m\in\Nbb^*\\ \beta\tq\kappa_{-\beta}\succeq^KP(k_P)
				\end{array}}{}\hspace{-2em} \ln_m\kappa_{-\beta} + \Sum{i=0}{\pinf}\ln |P(i)|} = s\exp e$} 
		Then $y+e\in \inner E$ and $\omega^b\asymp\partial u\exp(r\ln u+\eta+y+e)$. If not, then $P(0)$ is a term
		of $y$ (not up to a real factor, an actual term). Hence, we have the following cases :
		\begin{itemize}
			\item $P(0)=s\ln_p\kappa_{-\alpha}$ for some $s\in\Rbb^*_-$ and $p\geq n+2$. Then,
			\begin{calculs}
				&-\Sum{m=n+2}{\pinf}\ln_m\kappa_{-\alpha} -\hspace{-2em}\Sum{\footnotesize\begin{array}{c}
						\beta>\alpha, m\in\Nbb^*\\ \beta\tq\kappa_{-\beta}\succeq^KP(k_P)
				\end{array}}{} \hspace{-2em}\ln_m\kappa_{-\beta} + \Sum{i=0}{\pinf}\ln |P(i)| &=&  
				\ln|s| - \Sum{m=n+2}{p}\ln_m\kappa_{-\alpha} \\ &&\in& \ln|s|+E_3
			\end{calculs}
			Then,
			\centre{$y-\Sum{m=n+2}{\pinf}\ln_m\kappa_{-\alpha} -\hspace{-2em}\Sum{\footnotesize\begin{array}{c}
						\beta>\alpha, m\in\Nbb^*\\ \beta\tq\kappa_{-\beta}\succeq^KP(k_P)
				\end{array}}{} \hspace{-2em}\ln_m\kappa_{-\beta} + \Sum{i=0}{\pinf}\ln |P(i)| \in\Rbb+\inner E$}
			
			\item $P(0)=s\ln_p\kappa_{-\beta}$ with $\beta>\alpha$ and $p\in\Nbb^*$ such that there is some path $Q\in\Pcal_\Lbb(\eta)$ such that $\kappa_{-\beta}\succeq^K Q(k_Q)$. Then
			\centre{$-\Sum{m=n+2}{\pinf}\ln_m\kappa_{-\alpha} -\hspace{-2em}\Sum{\footnotesize\begin{array}{c}
						\beta>\alpha, m\in\Nbb^*\\ \beta\tq\kappa_{-\beta}\succeq^KP(k_P)
				\end{array}}{} \hspace{-2em}\ln_m\kappa_{-\beta} + \Sum{i=0}{\pinf}\ln |P(i)| \in \ln|s| + E_2$}
			Then,
			\centre{$y-\Sum{m=n+2}{\pinf}\ln_m\kappa_{-\alpha} -\hspace{-2em}\Sum{\footnotesize\begin{array}{c}
						\beta>\alpha, m\in\Nbb^*\\ \beta\tq\kappa_{-\beta}\succeq^KP(k_P)
				\end{array}}{} \hspace{-2em}\ln_m\kappa_{-\beta} + \Sum{i=0}{\pinf}\ln |P(i)| \in\Rbb+\inner E$}
			
			\item There are some paths $P_0,\dots,P_k\in \Pcal_\Lbb(\eta)$ and some non-zero integers $i_1,\dots,i_k$ such that 
			\centers{$\forall j\in\intn1k\ \exists j'\in\intn 0{j-1}\ \forall i\in\intn0{i_j-1}\quad P_{j'}(i)=P_j(i)$}
			and
			\begin{calculs}
				&\exists y'\in\inner E\quad y&=&y'-(k+1)\Sum{m=n+2}{\pinf}\ln_m{\kappa_{-\alpha}}\\
				&&& \quad-\Sum{j=0}{k}\Sum{\footnotesize\begin{array}{c}
						\beta>\alpha, m\in\Nbb^*\\ \beta\tq\kappa_{-\beta}\succeq^KP_j(k_{P_j})
				\end{array}}{}\hspace{-2em} \ln_m\kappa_{-\beta} + \Sum{j=0}{k}\Sum{i=i_j}{\pinf}\ln |P_j(i)|
			\end{calculs}
			and such that $P(0)\in\Rbb z$ for some $z$ a term of some $\ln|P_j({i_{k+1}}')|$ with $j\in\intn0k$ and ${i_{k+1}}'\geq i_j$. 
			Let $P_{k+1}$ be the following path :
			\centre{$P_{k+1}(i) = \begin{accolade}
					P_j(i) & i\leq {i_{k+1}}' \\ z & i={i_{k+1}}'+1 \\ P(i-{i_{k+1}}'-1) & i>{i_{k+1}}'+1 
				\end{accolade}$}
			Then, $P_{k+1}\in\Pcal(\eta)$. Moreover, $\partial P_{k+1}=\underbrace{P_j(0)\cdots P_j({i_{k+1}}')}_{\neq 0}\underbrace{\partial P}_{\neq 0}$. Then $P_{k+1}\in\Pcal_\Lbb(\eta)$. Note also that for all $\beta$, 
			
			\centre{$\kappa_{-\beta}\succeq^K P_{k+1}(k_{P_{k+1}})\Longleftrightarrow\kappa_{-\beta}\succeq^KP(k_P)$}
			Finally,
			
			\begin{calculs}
				&&&-\Sum{m=n+2}{\pinf}\ln_m\kappa_{-\alpha} -\hspace{-2.5em}\Sum{\footnotesize\begin{array}{c}
						\beta>\alpha, m\in\Nbb^*\\ \beta\tq\kappa_{-\beta}\succeq^KP(k_P)
				\end{array}}{} \hspace{-2.5em}\ln_m\kappa_{-\beta} + \Sum{i=0}{\pinf}\ln |P(i)|\\
				&&=&-\Sum{m=n+2}{\pinf}\ln_m\kappa_{-\alpha} -\hspace{-2.5em}\Sum{\footnotesize\begin{array}{c}
						\beta>\alpha, m\in\Nbb^*\\ \beta\tq\kappa_{-\beta}\succeq^KP_{k+1}(k_{P_{k+1}})
				\end{array}}{}\hspace{-2.5em} \ln_m\kappa_{-\beta} + \Sum{i={i_{k+1}}'+1}{\pinf}\ln |P_{k+1}(i)| + \ln\underbrace{\left|\f{P(0)}{z}\right|}_{\in\Rbb^*_+}\\
			\end{calculs}
			From that we derive that
			\begin{calculs}
				&&& y-\Sum{m=n+2}{\pinf}\ln_m\kappa_{-\alpha} -\hspace{-2em}\Sum{\footnotesize\begin{array}{c}
						\beta>\alpha, m\in\Nbb^*\\ \beta\tq\kappa_{-\beta}\succeq^KP(k_P)
				\end{array}}{} \hspace{-2em}\ln_m\kappa_{-\beta} + \Sum{i=0}{\pinf}\ln |P(i)| \\
				&&=& y'-(k+2)\Sum{m=n+2}{\pinf}\ln_m{\kappa_{-\alpha}}\\
				&&&\quad-\Sum{j=0}{k+1}\Sum{\footnotesize\begin{array}{c}
						\beta>\alpha, m\in\Nbb^*\\ \beta\tq\kappa_{-\beta}\succeq^KP_j(k_{P_j})
				\end{array}}{}\hspace{-2em} \ln_m\kappa_{-\beta} + \Sum{j=0}{k+1}\Sum{i=i_j}{\pinf}\ln |P_j(i)| 
				+ \ln\left|\f{P(0)}{z}\right|\\
				&&\in& \Rbb + \inner E
			\end{calculs}
			where $i_{k+1}={i_{k+1}}'+1$ and $P_{k+1}(i_k)=z$ has indeed its support (which is reduced to a singleton) included
			in the one of $e\left(\begin{array}{c}P_0,\dots,P_k\\ i_1,\dots,i_k\end{array}\right)$.
		\end{itemize}
		Then there is a real number $s$, and $e\in\inner E$ such that 
		\centre{$\omega^b\asymp\partial u\exp(r\ln u + \eta + e + s)\asymp \partial u\exp(r\ln u + \eta + e)$}
		Then we get the property at rank $\ell+1$.
	\end{itemize}
	By the induction principle, we conclude that the proposition is true for any $\ell\in\Nbb$.
\end{proof}

\begin{corollary}\label{cor:propsupportPhi1}	Let $x=\aSurreal$ such that
	\centre{$\exists u=\ln_n\kappa_{-\alpha}\ \exists r\in\Rbb\ \forall a\in\supp x\ \exists \eta\prec\ln u\quad \omega^a=\partial(u)\exp(r\ln u+\eta)$}
	We denote for $P_0,\dots,P_k\in\Pcal_\Lbb(x)$ and $i_1,\dots,i_k\in\Nbb\setminus\{0,1\}$,
	\begin{calculs}
		&e\left(\begin{array}{c}P_0,\dots,P_k\\ i_1,\dots,i_k\end{array}\right) &=& 
		-(k+1)\Sum{m=n+2}{\pinf}\ln_m{\kappa_{-\alpha}}\\
		&&&\quad-\Sum{j=0}{k}\Sum{\footnotesize\begin{array}{c}
				\beta>\alpha, m\in\Nbb^*\\ \beta\tq\kappa_{-\beta}\succeq^KP_j(k_{P_j})
		\end{array}}{} \hspace{-2em}\ln_m\kappa_{-\beta}
		+ \Sum{j=0}{k}\Sum{i=i_j}{\pinf}\ln |P_j(i)|
	\end{calculs}
	with $i_0=0$. For $k\in\Nbb$ define $E_{1,k}$ by:
	\begin{calculs}
		& e\left(\begin{array}{c}P_0,\dots,P_k\\ i_1,\dots,i_k\end{array}\right)\in E_{1,k}&\Leftrightarrow& P_0,\dots,P_k\in\Pcal_\Lbb(\eta)\quad\wedge\quad i_1,\dots,i_k\in\Nbb\setminus\{0,1\}\\
		&&&\quad\wedge\quad \forall i\in\intn0k\qquad P_i(1)\prec\ln u\\
		&&&\quad\wedge\quad \forall j\in\intn1k\quad \exists j'\in\intn 0{j-1}\\
		&&&\qquad\qquad \forall i\in\intn0{i_j-1}\qquad P_{j'}(i)=P_j(i)\\
		&&&\quad\wedge\quad \forall j \in \intn1k\\ &&&\qquad\quad \supp P_j(i_j)\subseteq\supp e\left(\begin{array}{c}P_0,\dots,P_{k-1}\\ i_1,\dots,i_{k-1}\end{array}\right)
		
	\end{calculs}
	Let also
	\begin{calculs}
		& E_1 &=& \Union{k\in\Nbb}{} E_{1,k} \\ [.5cm]
		&E_2 &=& \enstq{\begin{array}{c}
				-\Sum{m=n+2}{\pinf}\ln_m\kappa_{-\alpha}
			-\hspace{-1em}\Sum{\gamma<\beta,\ m\in\Nbb^*}{}\hspace{-1em}\ln_m\kappa_{-\gamma}\\ \quad - \Sum{m=1}{p}\ln_m\kappa_{-\beta}\end{array}\!\!}{\!\!
			\begin{array}{c}
				\beta>\alpha \\ \exists P\in\Pcal_\Lbb(x)\ \kappa_{-\beta}\succeq^K P(k_P) \\ p\in\Nbb
		\end{array}}\\ [1cm]
		&E_3 &=& \enstq{-\Sum{m=n+2}{p}\ln_m{\kappa_{-\alpha}}}{p\geq n+2}\\ [.4cm]
		& E &=& E_1\cup E_2\cup E_3
	\end{calculs}
	and $\inner E$ be the monoid it generates. Let $b\in\Union{\ell=0}{\pinf}\supp\Phi^\ell(x)$. Then, there is $y\in\inner E$ such that
	\centre{$\omega^b\asymp\exp(y)$}
\end{corollary}

\begin{proof}
	Since $\Phi$ is strongly linear\index{Strongly linear function}, we just need to apply Proposition \ref{prop:supportPhi1} to each term of $x$. For each term, $\partial u \exp(r\ln u + \eta)$ is term we add at the beginning $P_0$. Each path involved is shifted one rank. 
\end{proof}

\begin{proposition}
	\label{prop:bonOrderPhi1}
	Let $x=\aSurreal$ such that
	\centre{$\exists u=\ln_n\kappa_{-\alpha}\ \exists r\in\Rbb\ \forall a\in\supp x\ \exists \eta\prec\ln u\quad \omega^a=\partial(u)\exp(r\ln u+\eta)$} 
	Consider $E_1$, $E_2$ and $E_3$ as defined Corollary \ref{cor:propsupportPhi1}.
	Let $\gamma$ be the smallest ordinal such that $\kappa_{-\gamma}\prec^K P(k_P)$ for all path $P\in\Pcal_\Lbb(\eta)$.
	Let $\lambda$ the least $\epsilon$-number greater than $\NR(x)$ and $\gamma$. Then $E=E_1\cup E_2\cup E_3$ is reverse well-ordered with order type at most $2\lambda+\omega(\gamma+1)$.
\end{proposition}

\begin{proof}
	First notice that $E_3$ is reverse well-ordered with order type $\omega$. $E_2$ is also reverse well-ordered with order at most $\omega+\omega\otimes\gamma+n\leq \omega\otimes (\gamma+1)$. We then focus on $E_1$. We denote again
	\begin{calculs}
		&e\left(\begin{array}{c}P_0,\dots,P_k\\ i_1,\dots,i_k\end{array}\right) &=& 
		-(k+1)\Sum{m=n+2}{\pinf}\ln_m{\kappa_{-\alpha}}\\
		&&&\qquad-\Sum{j=0}{k}\Sum{\footnotesize\begin{array}{c}
				\beta>\alpha, m\in\Nbb^*\\ \beta\tq\kappa_{-\beta}\succeq^KP_j(k_{P_j})
		\end{array}}{} \ln_m\kappa_{-\beta}
		+ \Sum{j=0}{k}\Sum{i=i_j}{\pinf}\ln |P_j(i)|
	\end{calculs}
	
	\begin{enumerate}[label=(\roman*)]
		\item We first claim that for all $i\geq 3$ and all path $P\in\Pcal(x)$ such that $P(1)\prec\ln u$, $P(i)\prec P(2)\preceq \ln_2 u$. Let $P\in\Pcal(x)$ such that $P(1)\prec\ln u$. Assume $P(2)\succ\ln_2 u$. Then, since $P(2)$ is a term of $\ln |P(1)|$, we also have $\ln|P(1)|\succ\ln_2(u)$. Then, either $\ln|P(1)|<-m\ln_2 u$ for all $m\in\Nbb$, or $\ln|P(1)|>m\ln_2(u)$ for all $m\in\Nbb$. By definition, $P(1)$ is purely infinite. In particular, $\ln |P(1)|$ cannot be negative. Then,
		\centre{$\forall m\in\Nbb\qquad \ln|P(1)|>m\ln_2 u$}
		\lcr{and}{$\forall m\in\Nbb\qquad |P(1)| > \pa{\ln u}^m$}{($\exp$ is increasing)}
		which is a contradiction with $P(1)\prec\ln u$, since $\ln u$ is infinitely large. Since, for $i\geq 2$, $P(i)$ is infinitely large, $\ln|P(i)|\prec P(i)$, and since $P(i+1)\preceq\ln|P(i)|$, we have for all $i\geq 1$, $P(i+1)\prec P(i)$.
		By induction, we get
		\centers{$\forall i\geq3\qquad P(i)\prec P(2) \preceq\ln_2 u$}
		
		\item We claim that for all path $P\in\Pcal(x)$ such that $P(1)\prec\ln u$, if $P(2)\asymp\ln_2u$, then, denoting $r$  the real number such that 
		$P(2)\sim r\ln_2 u$, we have $0<r\leq1$ . Let $P\in\Pcal(x)$ such that $P(1)\prec\ln u$ and
		assume $P(2)\asymp\ln_2u$. Since $P(2)$ is a term there is a non-zero real number $r$ such that $P(1)=r\ln_2u$. From (i), we know that $P(2)$ is the dominant term of $\ln|P(1)|$ so that
		\centre{$\ln|P(1)|\sim r\ln_2 u$}
		If $r<0$, Proposition \ref{prop:formeExpXPurelyInfiniteOmegaAg} ensures that $|P(1)|\prec 1$ what is impossible since $P(1)$ is infinite. Then $r>0$. If now $r>1$ then again with Proposition \ref{prop:formeExpXPurelyInfiniteOmegaAg}, $|P(1)|\succ\ln u$ what is not true. Then, $0<r\leq1$.
		
		\item For all $j$ and $i\geq 2$, $\ln |P_j(i)|\preceq\ln_3 u\prec \ln_2u$. Indeed, using (i), we know that $P_j(i)\preceq \ln_2 u$. Then, there is a natural number $m\geq 1$ such that $|P_j(i)|\leq m\ln_2u$. Using the fact that $\ln$ is increasing,
		\centre{$\ln|P_j(i)| \leq \ln_3u+\ln m\preceq \ln_3u\prec \ln_2u$} 
		
		\item We now claim that $E_{1,k}>E_{1,k+2}$. Indeed, using (ii) and (iii) if $e_1\in E_{1,k}$, then there is $s\in\intff{-(k+1)}{-k}$ such that 
		$e_1\sim s\ln_2 u$. Similarly, for $e_2\in E_{1,k+2}$, there is $s'\in\intff{-(k+3)}{-(k+2)}$ such that $e_2\sim s'\ln_2u$. 
		
		\item We define the following sequence : 
		\begin{itemize}
			\item $a_0=\omega^{\omega^{\omega(\NR(x)+1)}}$
			\item $a_{k+1} = \omega^{\omega^{\omega(\omega(\NR(x)+\gamma+4)a_k+1)}}$
		\end{itemize}
		We show that $E_{1,k}$ is reverse well-ordered with order type less than $a_k$. We also claim that the equivalence classes of $E_{1,k}/{\asymp}$ are finite and that 
		\centre{$\NR\pa{\Sum{t\in E_{1,k}}{}\exp t}\leq \omega (\NR(x)+\gamma+4)a_k$}
		We show it by induction on $k\in\Nbb$.
		\begin{itemize}
			\item For $k=0$, let $t\in E_{1,0}$. Take $P\in\Pcal_\Lbb(x)$, minimal for $<_{lex}$ such that $P(1)\prec\ln u$ and $t=e(P;)$. Then
			\begin{calculs}
				&\partial(\ln u)\exp t &=& |P(0)\cdots P(k_P-1)|\exp\pa{
					-\Sum{\footnotesize\begin{array}{c}
							\beta\tq\kappa_{-\beta}\succeq^KP(k_P)\\ m\in\Nbb^*
					\end{array}}{} \ln_m\kappa_{-\beta}\right.\\ &&&\hspace{15em}\left. \vphantom{\Sum{\footnotesize\begin{array}{c}
						\beta\tq\kappa_{-\beta}\succeq^KP(k_P)\\ m\in\Nbb^*
				\end{array}}{}}
					+\Sum{i=k_P}{\pinf}\ln |P(i)|}\\ 
				&&=&|\partial P|
			\end{calculs}
			Since there are finitely many paths $Q\in\Pcal_\Lbb(x)$ such that $\partial P\asymp\partial Q$, there are finitely many $t'\in E_{1,0}$ such that
			\centre{$\partial (\ln u)\exp t \asymp\partial (\ln u)\exp t'$}
			Since $\exp$ is an increasing function and $\partial(\ln u)>0$, we get , using Proposition \ref{prop:majorationNuPartial}, that $E_{1,0}$ is reverse well-ordered with order type less than $\omega\otimes\omega^{\omega^{\omega(\NR(x)+1)}}=\omega^{\omega^{\omega(\NR(x)+1)}}=a_0$.
			Finally, it remains to compute the nested rank of $\Sum{t\in E_{1,0}}{}\exp t$. Write	
			\centre{$t=
				-\Sum{m=n+2}{\pinf}\ln_m{\kappa_{-\alpha}}
				-\Sum{\footnotesize\begin{array}{c}
						\beta>\alpha, m\in\Nbb^*\\ \beta\tq\kappa_{-\beta}\succeq^KP_0(k_{P_0})
				\end{array}}{} \ln_m\kappa_{-\beta}
				+\Sum{i=0}{\pinf}\ln |P_0(i)|$}
			\begin{calculs}
				& \NR\pa{t} &=& \NR\pa{
					-\Sum{m=n+2}{\pinf}\ln_m{\kappa_{-\alpha}}
					-\hspace{-2.5em}\Sum{\footnotesize\begin{array}{c}
							\beta>\alpha, m\in\Nbb^*\\ \beta\tq\kappa_{-\beta}\succeq^KP_0(k_{P_0})
					\end{array}}{}\hspace{-2.5em} \ln_m\kappa_{-\beta}
					+\Sum{i=0}{\pinf}\ln |P_0(i)|}\\
				&&\leq& \NR\pa{
					-\Sum{m=n+2}{\pinf}\ln_m{\kappa_{-\alpha}}
					-\hspace{-2.5em}\Sum{\footnotesize\begin{array}{c}
							\beta>\alpha, m\in\Nbb^*\\ \beta\tq\kappa_{-\beta}\succeq^KP_0(k_{P_0})
					\end{array}}{} \hspace{-2.5em}\ln_m\kappa_{-\beta}
					+\Sum{i=k_{P_0}}{\pinf}\ln |P_0(i)|}\\
				&&&+\Sum{i=0}{k_{P_0}-1}\NR\pa{\ln|P_0(i)|}+k_{P_0} & (Lemma \ref{lem:NRsum})\\
				&&\leq& (\omega\oplus \omega\otimes\gamma \oplus\omega)
				+\Sum{i=0}{k_{P_0}-1}\NR\pa{\ln|P_0(i)|}+k_{P_0}&(Lemma \ref{lem:NRSommeLogAtomiques})\\
				&&\leq& (\omega\oplus \omega\otimes\gamma \oplus\omega)+k_{P_0}(\NR(x)+1) & (using Proposition \ref{prop:NRComparaisonTerme})\\
				&&\leq& \omega (\NR(x)+\gamma+4)
			\end{calculs} 
			Then, since the equivalence classes of $E_{1,0}/{\asymp}$ are finite,
			\centre{$\NR\pa{\Sum{t\in E_{1,0}}{}\exp t}\leq \omega (\NR(x)+\gamma+4)a_0$}
			
			\item Assume the property for some $k\in\Nbb$. Let $t\in E_{1,k+1}$. Let \linebreak 
			$(P_0,0),\dots,(P_{k+1},i_{k+1})$ minimal for the order $(<_{lex},<)_{lex}$ such that $t=e\left(\begin{array}{c}P_0,\dots,P_{k+1}\\ i_1,\dots,i_{k+1}\end{array}\right)$. Then,
			\begin{calculs}
				&t&=&e\left(\begin{array}{c}P_0,\dots,P_k\\ i_1,\dots,i_k\end{array}\right)
				-\Sum{m=n+2}{\pinf}\ln_m{\kappa_{-\alpha}}\\
				&&&\qquad-\Sum{\footnotesize\begin{array}{c}
						\beta>\alpha, m\in\Nbb^*\\ \beta\tq\kappa_{-\beta}\succeq^KP_{k+1}(k_{P_{k+1}})
				\end{array}}{}\hspace{-2em} \ln_m\kappa_{-\beta}
				+ \Sum{i=i_{k+1}}{\pinf}\ln |P_{k+1}(i)|
			\end{calculs}
			Write $s=e\left(\begin{array}{c}P_0,\dots,P_k\\ i_1,\dots,i_k\end{array}\right)$. We then have,
			\begin{calculs}
				&\partial(\ln u)\exp t &=& 
					\exp(s)\exp\pa{-\Sum{\footnotesize\begin{array}{c}
							\beta\tq\kappa_{-\beta}\succeq^KP_{k+1}(k_{P_{k+1}})\\ m\in\Nbb^*
					\end{array}}{} \ln_m\kappa_{-\beta}\right.\\
					&&&\hspace{12em}\left.+\Sum{i=i_{k+1}}{\pinf}\ln |P_{k+1}(i)| \vphantom{\Sum{\footnotesize\begin{array}{c}
								\beta\tq\kappa_{-\beta}\succeq^KP_{k+1}(k_{P_{k+1}})\\ m\in\Nbb^*
						\end{array}}{}}}
			\end{calculs}
			Consider the following path:
			\centre{$\begin{accolade}
					R(0)=\exp s \\
					R(i)= P_{k+1}(i-1+i_{k+1}) & i>0
				\end{accolade}$}
			It is indeed a path since, by definition of $E_{1,k+1}$, $\supp P_{k+1}(i_{k+1})$ must be contained in $\supp s$. Then,
			\centre{$\partial(\ln u)\exp t=\partial R$} 
			Moreover, $R\in\Pcal_\Lbb\pa{\Sum{s\in E_{1,k}}{}\exp s}$. By induction hypothesis and Proposition \ref{prop:majorationNuPartial}, $E_{1,k+1}$ has order type less than 
			\centre{$\omega^{\omega^{\omega(\omega(\NR(x)+\gamma+4)a_k+1)}}=a_{k+1}$}	
			Since the equivalences classes of $\Pcal_\Lbb\pa{\Sum{s\in E_{1,k}}{}\exp s}/{\asymp}$ are finite, the 
			ones of $E_{1,k+1}/{\asymp}$ are also finite. Finally, using Lemmas \ref{lem:NRsum} and \ref{lem:NRSommeLogAtomiques},
			\begin{calculs}
				& \NR(t) &\leq& (\omega\oplus\omega\otimes\gamma\oplus\omega) + 
				\Sum{j=0}{k+1}\Sum{i=i_j}{k_{P_j}-1}\NR\pa{\ln|P_j(i)|}+\Sum{j=0}{k+1}\max(0,k_{P_j}-i_j)\\
				&&\leq& \omega(\NR(x)+\gamma+4)
			\end{calculs}
			\lc{Then,}{$\NR\pa{\Sum{t'\in E_{1,k+1}}{}\exp t'}\leq \omega(\NR(x)+\gamma+4)a_{k+1}$}
		\end{itemize}
		We conclude thanks to the induction principle.
		
		\item By easy induction, for all $k\in\Nbb$, $a_k<\lambda$.
		
		\item Using (iv), we get that for all $N\in\Nbb$, $\Union{k=0}{N}E_{1,2k}$ is an initial segment of
		$\Union{k\in\Nbb}{}E_{1,2k}$. We also have that $\Union{k=0}{N}E_{1,2k+1}$ is an initial segment of
		$\Union{k\in\Nbb}{}E_{1,2k+1}$. Using (v), we get that $\Union{k\in\Nbb}{}E_{1,2k}$ has order type at most
		\centre{$\sup\enstq{\Oplus{k=0}{N}a_{2k}}{N\in\Nbb} = \sup\enstq{a_{2N}}{N\in\Nbb}\underset{\text{by (vi)}}{\leq}\lambda$}
		Similarly,  $\Union{k\in\Nbb}{}E_{1,2k+1}$ has order type at most $\lambda$. Using Proposition \ref{prop:unionEnsBienOrd}, we conclude that $E_1$ has order type at most~$2\lambda$.
	\end{enumerate}
	Using again proposition \ref{prop:unionEnsBienOrd}, point (vii) above and the properties of $E_2$ and $E_3$ mentioned in the beginning of this proof, we get that $E$ is reverse well-ordered with order type at most $2\lambda+\omega(\gamma+1)$.
\end{proof}

\begin{Cor}
	\label{cor:supportPhi1} Let $x=\aSurreal$ such that
	\centre{$\exists u=\ln_n\kappa_{-\alpha}\quad\exists r\in\Rbb\quad \forall a\in\supp x\quad \exists \eta\prec\ln u\qquad \omega^a=\partial u\exp(r\ln u+\eta)$} 
	Let $\gamma$ be the smallest ordinal such that $\kappa_{-\gamma}\prec^K P(k_P)$ for all path $P\in\Pcal_\Lbb(\eta)$.
	Let $\lambda$ the least $\epsilon$-number greater than $\NR(x)$ and $\gamma$. Then $\Union{\ell=0}{\pinf}\supp\Phi^\ell(x)$ is reverse well-ordered with order type less at most $\omega^{\omega\pa{2\lambda+\omega(\gamma+1)+1}}$
\end{Cor}

\begin{proof}
	Just use Propositions \ref{prop:supportPhi1}, \ref{prop:bonOrderPhi1} and \ref{prop:orderTypeMonoid}.
\end{proof}

\subsubsection{Case $\epsilon\succ\ln u$}

\begin{Lem}
	\label{lem:comparePkPQkQdominantPath}
	Let $x$ be a surreal number. Let $P$ be the dominant path of $x$ and $Q\in\Pcal_\Lbb(x)$. Then, $P(k_P)\succeq^K Q(k_Q)$. In particular, for all ordinal $\beta$, if $\kappa_{-\beta}\succeq^K P(k_P)$, then $\kappa_{-\beta}\succeq^K Q(k_Q)$.
\end{Lem}

\begin{proof}
	\begin{enumerate}[label=(\roman*)]
		\item We first claim that for all $i\in\Nbb$, $P(i)\succeq Q(i)$. We prove it by induction.
		\begin{itemize}
			\item For $i=0$, $P(0)$ is the leading term of $x$ and $Q(0)$ is some term of $x$. Therefore, $P(0)\succeq Q(0)$.
			
			\item Assume $P(i)\succeq Q(i)$. $P(i+1)$ is the leading term of $\ln|P(i)|$. $P(i)$ and $Q(i)$ are both infinitely large. Then $\ln|P(i)|$ and $\ln|Q(i)|$ are both positive infinitely large. If $Q(i+1)\succ P(i+1)$ then, in particular, $\ln |Q(i)|\succ\ln|P(i)|$ what is impossible since $P(i)\succeq Q(i)$. Then $P(i+1)\succeq Q(i+1)$.
		\end{itemize}
		We conclude thanks to induction principle.
		
		\item Take $k=\max(k_P,k_Q)$. Using (i), we have :
		\centre{$P(k_P)\asymp^KP(k)\succeq Q(k)\asymp^K Q(k_Q)$}
		Hence, $P(k_P)\succeq^K Q(k_Q)$.
	\end{enumerate}
\end{proof}

\begin{Lem}
	\label{lem:formeEpsilonPhiOmegaA}
	Assume $x=\omega^a=\partial u\exp\epsilon$ with $\epsilon\succ\ln u$ and $u=\ln_n\kappa_{-\alpha}$. 
	Let $b\in\supp\Phi(\omega^a)$. Then, we have one of theses cases~:
	\begin{itemize}
		\item there is a path $P\in\Pcal(\eta)$ and $i\in\Nbb$ such that
		\centers{$\omega^b \asymp \partial u\exp\pa{\epsilon -\hspace{-2em}\Sum{\footnotesize\begin{array}{c}
						\beta\geq\alpha, m\in\Nbb^*\\ \beta\tq P_0(k_{P_0})\succ^K \kappa_{-\beta}\succeq^KP(k_P)
				\end{array}}{}\hspace{-2em} \ln_m\kappa_{-\beta} + \Sum{j=0}{\pinf}\ln \left|\f{P(i+j)}{P_0(j)}\right|}$}
		\lc{and}{$\forall j\in\intn0{i-1}\qquad P(j)=P_0(j)$}
		
		\item There is some $(\beta,m)<_{lex}(\alpha,n)$ such that there is some $\eta\prec\ln_m\kappa_{-\beta}$ such that
		\centre{$\omega^b\asymp\partial(\ln_m\kappa_{-\beta})\exp\eta$}
		where $\eta=\epsilon+\eta'$ and $\eta'$ only depends on $\alpha,\beta,n,m$ and $P_0$, the dominant path of $\epsilon$ :
		\centre{$\eta' = \Sum{\footnotesize\begin{array}{c}
					(\zeta,p)>_{lex}(\beta,m)\\ \zeta\tq \kappa_{-\zeta}\succeq^KP_0(k_{P_0})
			\end{array}}{}\hspace{-2em}\hspace{-1em} \ln_p\kappa_{-\zeta} - \Sum{(\beta,m)<_{lex}(\zeta,p)<_{lex}(\alpha,n)}{} \hspace{-1em}\ln_p\kappa_{-\zeta} - \Sum{i=0}{\pinf}\ln |P_0(i)|$}
		\lc{or}{$\eta' = \Sum{\footnotesize\begin{array}{c}
					(\zeta,p)\geq_{lex}(\alpha,n)\\ \zeta\tq\kappa_{-\zeta}\succeq^KP_0(k_{P_0})
			\end{array}}{}\hspace{-1em} \ln_p\kappa_{-\zeta}-\Sum{i=0}{\pinf}\ln|P_0(i)|$}
	\end{itemize}
\end{Lem}

\begin{proof}
	\lc{We have}{$\Phi(\omega^a)=\pa{1-\f{\partial\epsilon}s}\omega^a - \partial\pa{\f{\partial u}s}\exp\epsilon $}
	Let $b\in\supp\Phi(\omega^a)$. Then either 
	\centre{$b\in\supp\pa{\pa{1-\f{\partial\epsilon}s}\omega^a}$}
	\lc{or}{$b\in\supp\pa{\partial\pa{\f{\partial u}s}\exp\epsilon}$}
	\begin{itemize}
		\item \underline{First case} : $b\in\supp\pa{\pa{1-\f{\partial\epsilon}s}\omega^a}$. Then there is a path $P$, which is not the dominant path, such that
		\centre{$\omega^b\asymp \f{\partial P}{s}\omega^a
			\asymp \f{\exp\pa{
					-\Sum{\footnotesize\begin{array}{c}
							\beta\geq\alpha, m\in\Nbb^*\\ \beta\tq\kappa_{-\beta}\succeq^KP(k_P)
					\end{array}}{} \ln_m\kappa_{-\beta} + \Sum{i=0}{\pinf}\ln |P(i)|
			}}{
				\exp\pa{
					-\Sum{\footnotesize\begin{array}{c}
							\beta\geq\alpha, m\in\Nbb^*\\ \beta\tq\kappa_{-\beta}\succeq^KP_0(k_{P_0})
					\end{array}}{} \ln_m\kappa_{-\beta} + \Sum{i=0}{\pinf}\ln |P_0(i)|
			}}\omega^a$}
		where $P_0$ is the dominant path of $\epsilon$. Using Lemma \ref{lem:comparePkPQkQdominantPath}, we get
		\centre{$\omega^b\asymp \omega^a\exp\pa{-\Sum{\footnotesize\begin{array}{c}
						\beta\geq\alpha, m\in\Nbb^*\\ \beta\tq P_0(k_{P_0})\succ^K \kappa_{-\beta}\succeq^KP(k_P)
				\end{array}}{} \ln_m\kappa_{-\beta} + \Sum{i=0}{\pinf}\ln \left|\f{P(i)}{P_0(i)}\right|}$}
		
		\item\underline{Second case} : $b\in\supp\pa{\partial\pa{\f{\partial u}s}\exp\epsilon}$. First notice that $\partial\partial u=S_u\partial u$ where 
		\centre{$S_u = -\Sum{\beta<\alpha\  m\in\Nbb^*}{}\exp\pa{-\Sum{\zeta<\beta\ p\in\Nbb^*}{}\ln_p\kappa_{-\zeta} - \Sum{p=1}{m-1}\ln_p\kappa_{-\beta}} - \Sum{m=1}{n-1}\exp\pa{-\Sum{\beta<\alpha\ p\in\Nbb^*}{}\ln_p\kappa_{-\beta} - \Sum{p=1}{m-1}\ln_p\kappa_{-\alpha}}$}
		Hence, if $b\in\supp\pa{\f{\partial\partial u}{s}\exp\epsilon}$, there is some $(\beta,m)<_{lex}(\alpha,n)$ such that
		\centre{$\omega^b\asymp \omega^a\f{\exp\pa{-\Sum{\zeta<\beta\ p\in\Nbb^*}{}\ln_p\kappa_{-\zeta} - \Sum{p=1}{m-1}\ln_p\kappa_{-\beta}}}{
				\exp\pa{-\Sum{\footnotesize\begin{array}{c}
							p\in\Nbb^*\\ \zeta\tq \kappa_{-\zeta}\succeq^KP_0(k_{P_0})
					\end{array}}{} \ln_p\kappa_{-\zeta} + \Sum{i=0}{\pinf}\ln |P_0(i)|}}$}
		\lc{Therefor,}{$\omega^b\asymp\omega^a\exp\pa{\Sum{\footnotesize\begin{array}{c}
						(\zeta,p)\geq_{lex}(\beta,m)\\ \zeta\tq \kappa_{-\zeta}\succeq^KP_0(k_{P_0})
				\end{array}}{} \ln_p\kappa_{-\zeta} - \Sum{i=0}{\pinf}\ln |P_0(i)|}$}
		
		\lc{Notice that}{$\Sum{\footnotesize\begin{array}{c}
					(\zeta,p)\geq_{lex}(\beta,m)\\ \zeta\tq \kappa_{-\zeta}\succeq^KP_0(k_{P_0})
			\end{array}}{} \ln_p\kappa_{-\zeta} \sim \ln_m\kappa_{-\beta}\succ P_0(0)$}
		and then
		\begin{calculs}
			&\omega^a &\asymp& \partial(\ln_m\kappa_{-\beta})\exp\pa{
				\epsilon + \Sum{\footnotesize\begin{array}{c}
						(\zeta,p)>_{lex}(\beta,m)\\ \zeta\tq \kappa_{-\zeta}\succeq^KP_0(k_{P_0})
				\end{array}}{} \ln_p\kappa_{-\zeta}\right.\\&&&\left. \qquad\vphantom{\Sum{\footnotesize\begin{array}{c}
					(\zeta,p)>_{lex}(\beta,m)\\ \zeta\tq \kappa_{-\zeta}\succeq^KP_0(k_{P_0})
			\end{array}}{}}
				- \Sum{(\beta,m)<_{lex}(\zeta,p)<_{lex}(\alpha,n)}{} \ln_p\kappa_{-\zeta} - \Sum{i=0}{\pinf}\ln |P_0(i)|
			}
		\end{calculs}
		
		\lc{Since}{$\epsilon-\Sum{i=0}{\pinf}\ln|P(i)| \sim\epsilon\prec\ln_m\kappa_{-\beta}$}
		\lc{and}{$\Sum{\footnotesize\begin{array}{c}
					(\zeta,p)>_{lex}(\beta,m)\\ \zeta\tq \kappa_{-\zeta}\succeq^KP_0(k_{P_0})
			\end{array}}{} \ln_p\kappa_{-\zeta} - \Sum{(\beta,m)<_{lex}(\zeta,p)<_{lex}(\alpha,n)}{} \ln_p\kappa_{-\zeta}\prec\ln_m\kappa_{-\beta}$}
		Moreover,
		\centre{$\NR\pa{\ln\partial(\ln_m\kappa_{-\beta}) + \epsilon -\Sum{(\beta,m)<_{lex}(\zeta,p)<_{lex}(\alpha,n)}{} \ln_p\kappa_{-\zeta}}\leq \NR(x)$}
		and using Proposition \ref{prop:majorationNRPartial},
		\begin{calculs}
			&\NR\pa{\Sum{\footnotesize\begin{array}{c}
						(\zeta,p)>_{lex}(\beta,m)\\ \zeta\tq \kappa_{-\zeta}\succeq^KP_0(k_{P_0})
				\end{array}}{} \hspace{-2em}\ln_p\kappa_{-\zeta}- \Sum{i=0}{\pinf}\ln |P_0(i)|}&\leq& \NR(\partial P_0)\\
			&&\leq& k_{P_0}(\NR(x)+1)+\omega(\gamma+1)
		\end{calculs}
		We then conclude that there is some $\eta\prec\ln_m\kappa_{-\beta}$ such that
		\centre{$\omega^b = \partial(\ln_m\kappa_{-\beta})\exp\eta$}
		and by Corollary \ref{cor:NRprod},
		\centre{$\NR(\omega^b)\leq(k_{P_0}+1)(\NR(x)+1) + \omega(\gamma+1)$}
		
		Now assume $b\in\supp\pa{\f{\partial s}{s^2}\omega^a}$. 
		Notice that
		\begin{calculs}
			& \partial s &=& s\pa{	-\Sum{\footnotesize\begin{array}{c}
						m\in\Nbb^*\\ \beta\tq\kappa_{-\beta}\succeq^KP_0(k_{P_0})
				\end{array}}{} \partial\ln_m\kappa_{-\beta} + \Sum{i=0}{\pinf}\partial\ln |P_0(i)|}\\
			&&=& s\pa{-\Sum{\footnotesize\begin{array}{c}
						m\in\Nbb^*\\ \beta\tq\kappa_{-\beta}\succeq^KP_0(k_{P_0})
				\end{array}}{}\exp\pa{-\Sum{\zeta<\beta\ p\in\Nbb^*}{}\ln_p\kappa_{-\zeta} - \Sum{p=1}{m-1}\ln_p\kappa_{-\beta}} \right.\\ &&&\left. \vphantom{\Sum{\footnotesize\begin{array}{c}
					m\in\Nbb^*\\ \beta\tq\kappa_{-\beta}\succeq^KP_0(k_{P_0})
			\end{array}}{}}
				\qquad+\Sum{i=0}{\pinf}\partial\ln |P_0(i)|}
		\end{calculs}
		We then have the following sub-cases :
		\begin{itemize}
			\item There is some $m\in\Nbb^*$ and some ordinal $\beta$ such that $\kappa_{-\beta}\succeq^K P_0(k_{P_0})$ such that
			\begin{calculs}
				& \omega^b &\asymp& \f{\exp\pa{-\Sum{\zeta<\beta\ 				
							p\in\Nbb^*}{}\ln_p\kappa_{-\zeta} - \Sum{p=1}{m-1}\ln_p\kappa_{-\beta}}}{\exp\pa{-\Sum{\footnotesize\begin{array}{c}
								p\in\Nbb^*\\ \zeta\tq\kappa_{-\zeta}\succeq^KP_0(k_{P_0})
						\end{array}}{} \ln_p\kappa_{-\zeta}+\Sum{i=0}{\pinf}\ln|P_0(i)|}}\omega^a\\
				&&\asymp& \partial(\ln_m\kappa_{-\beta})\exp\pa{
					\epsilon +\Sum{\footnotesize\begin{array}{c}
							(\zeta,p)\geq_{lex}(\alpha,n)\\ \zeta\tq\kappa_{-\zeta}\succeq^KP_0(k_{P_0})
					\end{array}}{} \ln_p\kappa_{-\zeta}-\Sum{i=0}{\pinf}\ln|P_0(i)|}
			\end{calculs}
			with
			\centre{$\epsilon - \Sum{\footnotesize\begin{array}{c}
						\zeta\geq\alpha\ p\in\Nbb^*\\ \zeta\tq\kappa_{-\zeta}\succeq^KP_0(k_{P_0})
				\end{array}}{} \ln_p\kappa_{-\zeta}-\Sum{i=0}{\pinf}\ln|P_0(i)|\sim\epsilon\prec\ln_m\kappa_{-\beta}$}
			
			We then conclude that there is some $\eta\prec\ln_m\kappa_{-\beta}$ such that
			\centre{$\omega^b = \partial(\ln_m\kappa_{-\beta})\exp\eta$}
			and by Corollary \ref{cor:NRprod},
			\centre{$\NR(\omega^b)\leq(k_{P_0}+1)(\NR(x)+1) + \omega(\gamma+1)$}
			
			\item There is some path $P\in\Pcal_\Lbb(\epsilon)$ and some $i\geq 1$ such that for all $j<i$, $P(j)=P_0(j)$ and
			\centre{$\omega^b\asymp \f{\exp\pa{-\Sum{\footnotesize\begin{array}{c}
								p\in\Nbb^*\\ \zeta\tq\kappa_{-\zeta}\succeq^KP(k_P)
						\end{array}}{} \ln_p\kappa_{-\zeta}+\Sum{j=i}{\pinf}\ln|P(j)|}}{\exp\pa{-\Sum{\footnotesize\begin{array}{c}
								p\in\Nbb^*\\ \zeta\tq\kappa_{-\zeta}\succeq^KP_0(k_{P_0})
						\end{array}}{} \ln_p\kappa_{-\zeta}+\Sum{j=0}{\pinf}\ln|P_0(j)|}}\omega^a$}
			As in the first case, we get
			\centre{$\omega^b\asymp \omega^a\exp\pa{-\hspace{-1em}\Sum{\footnotesize\begin{array}{c}
							\beta\geq\alpha, m\in\Nbb^*\\ \beta\tq P_0(k_{P_0})\succ^K \kappa_{-\beta}\succeq^KP(k_P)
					\end{array}}{} \hspace{-3em}\ln_m\kappa_{-\beta} + \Sum{j=0}{\pinf}\ln \left|\f{P(i+j)}{P_0(j)}\right|}$}
		\end{itemize}
	\end{itemize}
\end{proof}

\begin{Prop} \label{prop:supportPhi2}
	Assume $x=\omega^a=\partial u\exp\epsilon$ with $\epsilon\succ\ln u$ and $u=\ln_n\kappa_{-\alpha}$. 
	Let $P_0$ be the dominant path of $\epsilon$.
	We denote for $P_1,\dots,P_{k+k'}\in\Pcal_\Lbb(\epsilon)$, $i_1,\dots,i_{k+k'}\in\Nbb^*$ and $(\beta,m)\leq_{lex}(\alpha,n)$,
	\begin{calculs}
		&e^{(\beta,m)}\left(\begin{matrix}
			P_1,\dots,P_k\\
			P_{k+1},\dots, P_{k+k'}\\
			i_1,\dots,i_{k+k'}
		\end{matrix}\right) &=&
		-k\Sum{i=0}{\infty}\ln|P_0(i)| + \Sum{j=1}{k}\Sum{i=i_j}{\infty}\ln|P_j(i)|\\ &&&
		-\Sum{j=1}{k}\Sum{\footnotesize\begin{array}{c}
				\gamma\geq\alpha, \ell\in\Nbb^*\\ \gamma\tq P_0(k_{P_0})\succ^K \kappa_{-\gamma}\succeq^KP_j(k_{P_j})
		\end{array}}{}\hspace{-2em} \ln_\ell\kappa_{-\gamma} \\ &&&
		-k'\Sum{\ell=m+2}{\pinf}\ln_\ell{\kappa_{-\beta}} + \Sum{j=k+1}{k+k'}\Sum{i=i_j}{\pinf}\ln |P_j(i)|\\ &&&
		-\Sum{j=k+1}{k+k'}\Sum{\footnotesize\begin{array}{c}
				\gamma>\beta, \ell\in\Nbb^*\\ \gamma\tq\kappa_{-\gamma}\succeq^KP_j(k_{P_j})
		\end{array}}{} \ln_\ell\kappa_{-\gamma}
		
	\end{calculs}
	We now define $E_{1,k,k'}^{(\beta,m)}$ as follows:
	\begin{calculs}
		&&& e^{(\beta,m)}\left(\begin{matrix}
			P_1,\dots,P_k\\
			P_{k+1},\dots, P_{k+k'}\\
			i_1,\dots,i_{k+k'}
		\end{matrix}\right)\in E_{1,k,k'}^{(\beta,m)}\\&&\Leftrightarrow&
			P_1,\dots,P_k\in\Pcal_\Lbb(\epsilon)\setminus\{P_0\}\\
		&&&\wedge\quad P_{k+1},\dots,P_{k+k'}\in\Pcal_\Lbb(\epsilon)\\
		&&&\wedge\quad i_1,\dots,i_k\in\Nbb\\
		&&&\wedge\quad i_{k+1},\dots,i_{k+k'}\in\Nbb^*\\
		&&&\wedge\quad 	\forall j\in\intn{1}{k+k'}\ \exists j'\in\intn 0{j-1}\\ &&&\qquad\qquad \forall i\in\intn0{i_j-1}\quad P_{j'}(i)=P_j(i)\\
		&&&\wedge\quad \forall j \in \intn{k+1}{k+k'}\\
			&&&\qquad\qquad\supp P_j(i_j)\subseteq\supp e^{(\beta,m)}\left(\begin{matrix}
			P_1,\dots,P_k\\ P_{k+1},\dots,P_j\\ i_1,\dots, i_j
		\end{matrix}\right)
	\end{calculs}
	\centre{$E_1^{(\beta,m)} = \begin{accolade}
			\Union{k\in\Nbb,\ k'\in\Nbb^*}{} E_{1,k,k'}^{(\beta,m)} & (\beta,m)\neq (\alpha,n)\\
			\Union{k\in\Nbb}{} E_{1,k,0}^{(\beta,m)} & (\beta,m)=(\alpha,n) 
		\end{accolade}$}
	Define sets $E_2^{(\beta,m)}$ as follows:
	\begin{itemize}
		\item If $(\beta,m)\neq (\alpha,n)$, then
		\centre{$
				-\Sum{\ell=m+2}{\pinf}\ln_\ell\kappa_{-\beta}
				-\hspace{-1em}\Sum{\gamma'<\gamma,\ \ell\in\Nbb^*}{}\hspace{-1em}\ln_\ell\kappa_{-\gamma'} - \Sum{\ell=1}{p}\ln_\ell\kappa_{-\gamma} \in  E_2^{(\beta,m)} $}
		\tiff $\gamma>\beta$, $p\in\Nbb$ and there is some $P\in\Pcal_\Lbb(\epsilon)$ such that $\kappa_{-\gamma}\succeq^K P(k_P)$
		\item If $(\beta,m)=(\alpha,n)$, then
		\centre{$-\Sum{j=0}{\infty}\ln|P_0(j)| - \Sum{\gamma>\zeta>\alpha,\ \ell\in\Nbb^*}{}\ln_\ell\kappa_{-\zeta} - \Sum{\ell=1}{p}\ln_\ell\kappa_{-\gamma} \in  E_2^{(\beta,m)}$}
		\tiff $\gamma>\alpha$, $p\in\Nbb$ and there is some $P\in\Pcal_\Lbb(\epsilon)$ such that $\kappa_{-\gamma}\succeq^K P(k_P)$.
	\end{itemize}
	Let also
	\begin{calculs}
		&E_3^{(\beta,m)} &=& \begin{accolade}
			\enstq{-\Sum{\ell=m+2}{p}\ln_\ell{\kappa_{-\beta}}}{p\geq m+2} & (\beta,m)\neq(\alpha,n)\\
			\emptyset & (\beta,m)=(\alpha,n)
		\end{accolade} \\ [.4cm]
		& E^{(\beta,m)} &=& E_1^{(\beta,m)}\cup E_2^{(\beta,m)}\cup E_3^{(\beta,m)}
	\end{calculs}
	and $\inner {E^{(\beta,m)}}$ be the monoid it generates. Finally, let $H^{(\beta,m)}$ defined by cases as follows:
	\centre{$\begin{accolade}
			\left\{\Sum{\footnotesize\begin{array}{c}
					(\zeta,p)>_{lex}(\beta,m)\\ \zeta\tq \kappa_{-\zeta}\succeq^KP_0(k_{P_0})
			\end{array}}{} \ln_p\kappa_{-\zeta}\hspace{8em}\right.\\ - \Sum{(\beta,m)<_{lex}(\zeta,p)<_{lex}(\alpha,n)}{} \hspace{-2em}\ln_p\kappa_{-\zeta} - \Sum{i=0}{\pinf}\ln |P_0(i)|,\\
			\left.\Sum{\footnotesize\begin{array}{c}
					(\zeta,p)\geq_{lex}(\alpha,n)\\ \zeta\tq\kappa_{-\zeta}\succeq^KP_0(k_{P_0})
			\end{array}}{} \hspace{-2em}\ln_p\kappa_{-\zeta}-\Sum{i=0}{\pinf}\ln|P_0(i)|\right\} & (\beta,m)\neq (\alpha,n) \\
			\{0\} & (\beta,m)=(\alpha,n)
		\end{accolade}$}
	Let $b\in\Union{q=0}{\pinf}\supp\Phi^q(\omega^a)$. Then, there are $\eta\in H^{(\beta,m)}$ and $y\in \inner {E^{(\beta,m)}}$ such that
	\centre{$\omega^b\asymp\partial (\ln_m\kappa_{-\beta})\exp(\epsilon + \eta + y)$}
\end{Prop}

\begin{proof}
	We prove it by induction on $q$.
	\begin{itemize}
		\item If $b\in\supp\Phi^0(\omega^a)$, then $b=a$ and $y=0$ with $(\beta,m)=(\alpha,n)$ and $\eta=0$ works.
		
		\item Assume the property for some $q\in\Nbb$. Let $b\in\supp\Phi^{q+1}(\omega^b)$. Then there is 
		$c\in\supp\Phi^q(\omega^a)$ such that $b\in\supp\Phi(\omega^c)$. Apply the induction hypothesis on $c$. Take $(\beta,m)$, $\eta\in H^{(\beta,m)}$ and $y\in\inner{E^{(\beta,m)}}$ such that
		\centre{$\omega^c\asymp\partial(\ln_m\beta)\exp(\epsilon+\eta+y)$}
		\begin{itemize}
			\item If $(\beta,m)<_{lex}(\alpha,n)$, then $y,\epsilon\prec\ln_{n+1}\kappa_{-\beta}$. Hence, using Lemma \ref{lem:formeEtaPhiOmegaA}, we get that there is $P\in\Pcal_\Lbb(\epsilon+\eta+y)$ such that
			\begin{calculs}
				&\omega^b &\asymp& \partial(\ln_m\kappa_{-\beta})\exp\pa{\epsilon+\eta+y
					-\Sum{\ell=m+2}{\pinf}\ln_\ell\kappa_{-\beta} \vphantom{\Sum{\footnotesize\begin{array}{c}
								\gamma>\beta, \ell\in\Nbb^*\\ \gamma\tq\kappa_{-\gamma}\succeq^KP(k_P)
						\end{array}}{}}\right. \\ &&& \left.\qquad
					-\hspace{-2em}\Sum{\footnotesize\begin{array}{c}
							\gamma>\beta, \ell\in\Nbb^*\\ \gamma\tq\kappa_{-\gamma}\succeq^KP(k_P)
					\end{array}}{}\hspace{-2em} \ln_\ell\kappa_{-\gamma} + \Sum{i=0}{\pinf}\ln |P(i)|}
			\end{calculs}
			If $P(0)$ a term of $\epsilon$, up to some real factor, then there is a real number $s$ and some $e\in E_{1,0,1}^{(\beta,m)}$ such that
			\begin{calculs}
				 & s\exp e &=& \exp\pa{-\Sum{\ell=m+2}{\pinf}\ln_\ell\kappa_{-\beta}
				 	-\hspace{-2em}\Sum{\footnotesize\begin{array}{c}
				 			\gamma>\beta, \ell\in\Nbb^*\\ \gamma\tq\kappa_{-\gamma}\succeq^KP(k_P)
				 	\end{array}}{}\hspace{-2em} \ln_\ell\kappa_{-\gamma} + \Sum{i=0}{\pinf}\ln |P(i)|}
			\end{calculs}
			Then $y+e\in \inner {E^{(\beta,m)}}$ and $\omega^b\asymp\partial(\ln_m\kappa_{-\beta})\exp(\epsilon+y+e)$. If not, then $P(0)$ is a term
			of $\eta+y$. Hence, we have the following cases:
			\begin{itemize}
				\item $P(0)=s\ln_p\kappa_{-\beta}$ for some $s\in\Rbb^*_-$ and $p\geq m+2$. Then,
				\begin{calculs}
					&&&-\Sum{\ell=m+2}{\pinf}\ln_\ell\kappa_{-\beta}
					-\Sum{\footnotesize\begin{array}{c}
							\gamma>\beta, \ell\in\Nbb^*\\ \gamma\tq\kappa_{-\gamma}\succeq^KP(k_P)
					\end{array}}{} \ln_\ell\kappa_{-\gamma} + \Sum{i=0}{\pinf}\ln |P(i)| \\
					&&=&  \ln|s| - \Sum{\ell=m+2}{p}\ln_\ell\kappa_{-\beta} \in \ln|s|+E_3^{(\beta,m)}
				\end{calculs}
				Then,
				\centre{$y-\Sum{\ell=m+2}{\pinf}\ln_\ell\kappa_{-\beta}
					-\hspace{-2em}\Sum{\footnotesize\begin{array}{c}
							\gamma>\beta, \ell\in\Nbb^*\\ \gamma\tq\kappa_{-\gamma}\succeq^KP(k_P)
					\end{array}}{}\hspace{-2em} \ln_\ell\kappa_{-\gamma} + \Sum{i=0}{\pinf}\ln |P(i)| \in\Rbb+\inner {E^{(\beta,m)}}$}
				
				\item $P(0)=s\ln_p\kappa_{-\gamma}$ with $\gamma>\beta$ and $p\in\Nbb^*$ such that there is some path $Q\in\Pcal_\Lbb(\epsilon)$ such that $\kappa_{-\beta}\succeq^K Q(k_Q)$. Then
				\centre{$-\Sum{\ell=m+2}{\pinf}\ln_\ell\kappa_{-\beta}
					-\hspace{-2em}\Sum{\footnotesize\begin{array}{c}
							\gamma>\beta, \ell\in\Nbb^*\\ \gamma\tq\kappa_{-\gamma}\succeq^KP(k_P)
					\end{array}}{} \hspace{-2em}\ln_\ell\kappa_{-\gamma} + \Sum{i=0}{\pinf}\ln |P(i)| \in \ln|s| + E_2^{(\beta,m)}$}
				Then,
				\centre{$y-\Sum{\ell=m+2}{\pinf}\ln_\ell\kappa_{-\beta}
					-\Sum{\footnotesize\begin{array}{c}
							\gamma>\beta, \ell\in\Nbb^*\\ \gamma\tq\kappa_{-\gamma}\succeq^KP(k_P)
					\end{array}}{} \ln_\ell\kappa_{-\gamma} + \Sum{i=0}{\pinf}\ln |P(i)| \in\Rbb+\inner {E^{(\beta,m)}}$}
				
				\item There are some paths $P_1,\dots,P_{k+k'}\in \Pcal_\Lbb(\epsilon)$ and some integers $i_1,\dots,i_{k+k'}$ such that 
				\centers{$e^{(\beta,m)}\left(\begin{matrix}
						P_1,\dots,P_k \\ P_{k+1},\dots,P_{k+k'}\\ i_1,\dots,i_{k+k'}
					\end{matrix}\right)\in E_{1,k,k'}^{(\beta,m)}$}
				\lc{and}{$\exists y'\in\inner E\qquad y=y'+e^{(\beta,m)}\left(\begin{matrix}
						P_1,\dots,P_k \\ P_{k+1},\dots,P_{k+k'}\\ i_1,\dots,i_{k+k'}
					\end{matrix}\right)$}
				and finally such that $P(0)\in\Rbb z$ for some $z$ a term of some $\ln|P_j({i_{k+k'+1}}')|$ with $j\in\intn0{k+k'}$ and ${i_{k+k'+1}}'\geq i_j$. 
				Let $P_{k+k'+1}$ be the following path :
				\centre{$P_{k+k'+1}(i) = \begin{accolade}
						P_j(i) & i\leq {i_{k+1}}' \\ z & i={i_{k+1}}'+1 \\ P(i-{i_{k+1}}'-1) & i>{i_{k+1}}'+1 
					\end{accolade}$}
				Then, $P_{k+k'+1}\in\Pcal(\epsilon)$. Moreover, 
				$$\partial P_{k+k'+1}=\underbrace{P_j(0)\cdots P_j({i_{k+1}}')}_{\neq 0}\underbrace{\partial P}_{\neq 0}$$
				Then $P_{k+k'+1}\in\Pcal_\Lbb(\epsilon)$. Note also that for all $\beta$, 
				
				\centre{$\kappa_{-\beta}\succeq^K P_{k+k'+1}(k_{P_{k+k'+1}})\Longleftrightarrow\kappa_{-\beta}\succeq^KP(k_{P})$}
				Finally,
				
				\begin{calculs}
					&&&-\Sum{\ell=m+2}{\pinf}\ln_\ell\kappa_{-\beta}
					-\Sum{\footnotesize\begin{array}{c}
							\gamma>\beta, \ell\in\Nbb^*\\ \gamma\tq\kappa_{-\gamma}\succeq^KP(k_P)
					\end{array}}{} \ln_\ell\kappa_{-\gamma} + \Sum{i=0}{\pinf}\ln |P(i)|\\
					&&=&-\Sum{\ell=m+2}{\pinf}\ln_\ell\kappa_{-\beta}	-\Sum{\footnotesize\begin{array}{c}
							\gamma>\beta, \ell\in\Nbb^*\\ \gamma\tq\kappa_{-\gamma}\succeq^KP_{k+k'+1}(k_{P_{k+k'+1}})
					\end{array}}{} \ln_\ell\kappa_{-\gamma}\\ &&&\qquad\qquad + \Sum{i={i_{k+1}}'+1}{\pinf}\ln |P_{k+k'+1}(i)| + \ln\underbrace{\left|\f{P(0)}{z}\right|}_{\in\Rbb^*_+}\\
				\end{calculs}
				From that we derive that
				\begin{calculs}
					&&& y-\Sum{\ell=m+2}{\pinf}\ln_\ell\kappa_{-\beta}
					-\Sum{\footnotesize\begin{array}{c}
							\gamma>\beta, \ell\in\Nbb^*\\ \gamma\tq\kappa_{-\gamma}\succeq^KP(k_P)
					\end{array}}{} \ln_\ell\kappa_{-\gamma} + \Sum{i=0}{\pinf}\ln |P(i)| \\
					&&=& y' +  e^{(\beta,m)}\left(\begin{matrix}
						P_1,\dots,P_k \\ P_{k+1},\dots,P_{k+k'+1}\\ i_1,\dots,i_{k+k'+1}
					\end{matrix}\right)
					+\ln\left|\f{P(0)}{z}\right| \in\Rbb + \inner {E^{(\beta,m)}}
				\end{calculs}
				where $i_{k+k'+1}={i_{k+k'+1}}'+1$ and $P_{k+k'+1}(i_{k+k'})=z$ has indeed its support (which is reduced to a singleton) included
				in the one of $e^{(\beta,m)}\left(\begin{matrix}
					P_1,\dots,P_k\\ P_{k+k'},\dots,P_{k+k'}\\ i_1,\dots, i_{k+k'}
				\end{matrix}\right)$.
			\end{itemize}
			Then there is a real number $s$, and $e\in\inner{E^{(\beta,m)}}$ such that 
			\centre{$\omega^b\asymp\partial (\ln_m\beta)\exp(\epsilon + \eta + e + s)\asymp \partial (\ln_m\beta)\exp(\epsilon + \eta + e)$}
			Then we get the property at rank $q+1$.

			\item If $(\beta,m)=(\alpha,n)$, we have $\eta=0$ and write
			\centre{$y=y'+e^{(\alpha,n)}\left(\begin{matrix}
					P_1,\dots, P_k\\ \emptyset\\ i_1,\dots,i_{k+k'}
				\end{matrix}\right)$}
			with, $y'\in\inner{E^{(\beta,m)}}$ and, $k,k'\in\Nbb$.
			Using Lemma \ref{lem:formeEpsilonPhiOmegaA}, we have
			
			\begin{itemize}
				\item \lc{\underline{First case} :}{$\omega^b\asymp\partial u\exp(\epsilon + y + e)$}
				where
				\centre{$e=-\Sum{\footnotesize\begin{array}{c}
							\gamma\geq\alpha, \ell\in\Nbb^*\\ \gamma\tq P_0(k_{P_0})\succ^K \kappa_{-\gamma}\succeq^KP(k_P)
					\end{array}}{} \ln_\ell\kappa_{-\gamma} + \Sum{j=0}{\pinf}\ln \left|\f{P(i+j)}{P_0(j)}\right|$}
				for some path $P\in\Pcal_\Lbb(\epsilon + y)$ and some $i\in\Nbb$ such that
				\centre{$\forall j\in\intn0{i-1}\qquad P(j)=P_0(j)$}
				Indeed, $y\in\inner{E^{(\alpha,n)}}$. In particular, $y\prec\epsilon$ and then $\epsilon+y\sim\epsilon$ so that $P_0$ is also the dominant path of $\epsilon+y$. 
				
				\begin{itemize}
					\item If $P(0)$ is, up to a real factor, a term of $\epsilon$, 
					then we get that there is some path $Q\in\Pcal_\Lbb(\epsilon)$ and a real number $s$ such that 
					\centre{$y+e=y'+e^{(\beta,m)}\left(\begin{matrix}
							P_1,\dots,P_k,Q\\ \emptyset\\ i_1,\dots, i_k,i
						\end{matrix}\right)+s$}
					Since $y\prec\epsilon$, and $P\neq P_0$, we also have $Q\neq P_0$. Then $y+e\in \inner{E^{(\beta,m)}}+E_{1,k+1,k'}^{(\beta,m)}+s$. Let  \centre{$y''=y+e-s\in\inner{E^{(\beta,m)}}$} 
					\lc{then,}{$\omega^b\asymp\partial u\exp(\epsilon + y'')$}
					In particular, $y''\prec\epsilon$.
					
					\item If $P(0)$ is a term of $y$, and more precisely if it can be written as $P(0)=s\ln_p\kappa_{-\gamma}$ for $s\in\Rbb$, $p\in\Nbb$ and $\gamma\geq\alpha$ such that 
					\centre{$P_0(k_{P_0})\succ^K \kappa_{-\gamma}\succeq^K Q(k_Q)$}
					for some path $Q\in\Pcal_\Lbb(\epsilon)\setminus\{P_0\}$. Then,
					\centre{$e=-\Sum{j=0}{\infty}\ln|P_0(j)| - \Sum{\gamma>\zeta>\alpha,\ \ell\in\Nbb^*}{}\ln_\ell\kappa_{-\zeta} - \Sum{\ell=1}{p+i}\ln_\ell\kappa_{-\gamma} + \ind_{i=0}\ln|s| \in E_2^{(\beta,m)}+\Rbb$}
					Then $y+e-\ln|s|\in\inner{E^{(\beta,m)}}$ and  since $e\prec\epsilon$, $y+e-s\prec\epsilon$ and 
					\centre{$\omega^b\asymp\partial u\exp(\epsilon+y+e-\ln|s|)$}
					
					\item If $P(0)$ is a term of $y$, and more precisely if it can be written as $P(0)=s\ln|P_\ell(j)|$ for some $s\in\Rbb$ and some $\ell\in\intn0{k+k'}$ (actually it is true if we have chosen well the $y'$ in the beginning, but up to a renaming, it is true). Consider the following path
					\centre{$ Q(p) = \begin{accolade}
							P_\ell(p) & p\leq j\\ P(p-j) & p>j
						\end{accolade}$}
					We have $Q\in\Pcal_\Lbb(\epsilon)$ and
					\centre{$y+e=y'+e^{(\beta,m)}\left(\begin{matrix}
							P_1,\dots,P_k,Q\\ \emptyset\\ i_1,\dots,i_{k},j
						\end{matrix}\right) + \ln|s|$}
					Then $y+e-\ln|s|\in\inner{E^{(\beta,m)}}$ and  since $e\prec\epsilon$, $y+e-s\prec\epsilon$ and 
					\centre{$\omega^b\asymp\partial u\exp(\epsilon+y+e-\ln|s|)$}
				\end{itemize}
				This concludes the first case.
				
				\item \underline{Second case} : There are $(\beta',m')<_{lex}(\alpha,n)$ and $\eta'\in H^{(\beta,m)}$ such that $\omega^b\asymp\partial(\ln_{m'}\kappa_{-\beta'})\exp(\epsilon+\eta'+y)$. This immediately conclude the second case.
			\end{itemize}
			We then have the property at rank $q+1$.
		\end{itemize}
	\end{itemize}
	Thanks to the induction principle, we conclude that the property holds for any $q\in\Nbb$.
\end{proof}

\begin{corollary} \label{cor:propsupportPhi2} Let $x$ be a surreal number such that
	\centre{$\exists u=\ln_n\kappa_{-\alpha}\quad \exists r_0\in\Rbb^* \quad \exists a_0\in\Nobf \quad \forall a\in\supp x\quad\exists\epsilon\sim r\omega^{a_0}\succ\ln u\qquad \omega^a\asymp\partial u\exp\epsilon$}
	\lc{Let}{$\Pcal_0(x)=\enstq{P\in\Pcal_\Lbb(x)}{\begin{array}{c}
				P(1)=r\omega^{a_0} \\ \forall i\geq 1 \quad P(i+1)\sim\ln|P(i)
		\end{array}}$}
	It is the set of all the possible dominant paths of the epsilon to which we add the corresponding term of $x$ at the beginning.
	We denote for $P_0\in\Pcal_0(x)$, $P_1,\dots,P_{k+k'}\in\Pcal_\Lbb(x)$, $i_1,\dots,i_{k+k'}\in\Nbb^*$ and $(\beta,m)\leq_{lex}(\alpha,n)$,
	\begin{calculs}
		&e^{(\beta,m)}\left(\begin{matrix}
			P_0; P_1,\dots,P_k\\
			P_{k+1},\dots, P_{k+k'}\\
			i_1,\dots,i_{k+k'}
		\end{matrix}\right) &=&
		-k\Sum{i=1}{\infty}\ln|P_0(i)| -k'\Sum{\ell=m+2}{\pinf}\ln_\ell{\kappa_{-\beta}}\\ &&& 
		-\Sum{j=1}{k}\hspace{-1em}\Sum{\scriptsize\begin{array}{c}
				\gamma\geq\alpha, \ell\in\Nbb^*\\ \gamma\tq P_0(k_{P_0})\succ^K \kappa_{-\gamma}\succeq^KP_j(k_{P_j})
		\end{array}}{}\hspace{-3em} \ln_\ell\kappa_{-\gamma}\\
		&&&
		+ \Sum{j=1}{k}\Sum{i=i_j}{\infty}\ln|P_j(i)|\\
		&&&	
		-\Sum{j=k+1}{k+k'}\Sum{\footnotesize\begin{array}{c}
				\gamma>\beta, \ell\in\Nbb^*\\ \gamma\tq\kappa_{-\gamma}\succeq^KP_j(k_{P_j})
		\end{array}}{}\hspace{-2em} \ln_\ell\kappa_{-\gamma}\\
		&&&
		+ \Sum{j=k+1}{k+k'}\Sum{i=i_j}{\pinf}\ln |P_j(i)|
	\end{calculs}

	We now define $E_{1,k,k'}^{(\beta,m)}$ as follows:
	\begin{calculs}
		&&& e^{(\beta,m)}\left(\begin{matrix}
			P_0;P_1,\dots,P_k\\
			P_{k+1},\dots, P_{k+k'}\\
			i_1,\dots,i_{k+k'}
		\end{matrix}\right)\in E_{1,k,k'}^{(\beta,m)} \\&&\Leftrightarrow&
		P_0\in\Pcal_0(x) \wedge P_1,\dots,P_k\in\Pcal_\Lbb(\epsilon)\setminus\{P_0\}\\
		&&&\wedge\quad P_{k+1},\dots,P_{k+k'}\in\Pcal_\Lbb(x)\\
		&&&\wedge\quad i_1,\dots,i_k\in\Nbb^*\\
		&&&\wedge\quad i_{k+1},\dots,i_{k+k'}\in\Nbb\setminus \{0,1\}\\
		&&&\wedge\quad 	\forall j\in\intn{1}{k+k'}\ \exists j'\in\intn 0{j-1}\\ &&&\qquad\qquad \forall i\in\intn0{i_j-1}\quad P_{j'}(i)=P_j(i)\\
		&&&\wedge\quad \forall j \in \intn{k+1}{k+k'}\\
		&&&\qquad\qquad\supp P_j(i_j)\subseteq\supp e^{(\beta,m)}\left(\begin{matrix}
			P_0; P_1,\dots,P_k\\ P_{k+1},\dots,P_j\\ i_1,\dots, i_j
		\end{matrix}\right)
	\end{calculs}
	\centre{$E_1^{(\beta,m)} = \begin{accolade}
			\Union{k\in\Nbb,\ k'\in\Nbb^*}{} E_{1,k,k'}^{(\beta,m)} & (\beta,m)\neq (\alpha,n)\\
			\Union{k\in\Nbb}{} E_{1,k,0}^{(\beta,m)} & (\beta,m)=(\alpha,n) 
		\end{accolade}$}
	Define sets $E_2^{(\beta,m)}$ as follows:
	\begin{itemize}
		\item If $(\beta,m)\neq (\alpha,n)$, then
		\centre{$
			-\Sum{\ell=m+2}{\pinf}\ln_\ell\kappa_{-\beta}
			-\hspace{-1em}\Sum{\gamma'<\gamma,\ \ell\in\Nbb^*}{}\hspace{-1em}\ln_\ell\kappa_{-\gamma'} - \Sum{\ell=1}{p}\ln_\ell\kappa_{-\gamma} \in  E_2^{(\beta,m)} $}
		\tiff $\gamma>\beta$, $p\in\Nbb$ and there is some $P\in\Pcal_\Lbb(\epsilon)$ such that $\kappa_{-\gamma}\succeq^K P(k_P)$
		\item If $(\beta,m)=(\alpha,n)$, then
		\centre{$-\Sum{j=0}{\infty}\ln|P_0(j)| - \Sum{\gamma>\zeta>\alpha,\ \ell\in\Nbb^*}{}\ln_\ell\kappa_{-\zeta} - \Sum{\ell=1}{p}\ln_\ell\kappa_{-\gamma} \in  E_2^{(\beta,m)}$}
		\tiff $\gamma>\alpha$, $p\in\Nbb$, $P_0\in\Pcal_0(x)$, $P_0(k_{P_0}) \succ^K\kappa_{-\gamma}$ and there is some $P\in\Pcal_\Lbb(\epsilon)$ such that $\kappa_{-\gamma}\succeq^K P(k_P)$.
	\end{itemize}

	Let also
	\begin{calculs}
		&E_3^{(\beta,m)} &=& \begin{accolade}
			\enstq{-\Sum{\ell=m+2}{p}\ln_\ell{\kappa_{-\beta}}}{p\geq m+2} & (\beta,m)\neq(\alpha,n)\\
			\emptyset & (\beta,m)=(\alpha,n)
		\end{accolade} \\ [.4cm]
		& E^{(\beta,m)} &=& E_1^{(\beta,m)}\cup E_2^{(\beta,m)}\cup E_3^{(\beta,m)}
	\end{calculs}
	and $\inner {E^{(\beta,m)}}$ be the monoid it generates. Finally, let
	$H^{(\beta,m)}$ defined by cases as follows:
	\begin{itemize}
		\item If $(\beta,m)\neq (\alpha,n)$, then
		\begin{calculs}
			& H^{(\beta,m)} &=& \enstq{\begin{array}{r}
					\Sum{\footnotesize\begin{array}{c}
						(\zeta,p)>_{lex}(\beta,m)\\ \zeta\tq \kappa_{-\zeta}\succeq^KP_0(k_{P_0})
				\end{array}}{}\hspace{-2em} \ln_p\kappa_{-\zeta} - \Sum{(\beta,m)<_{lex}(\zeta,p)<_{lex}(\alpha,n)}{} \hspace{-2em}\ln_p\kappa_{-\zeta}\\ - \Sum{i=0}{\pinf}\ln |P_0(i)|\end{array}}{P_0\in\Pcal_0(x)}\\
			&&&\qquad\bigcup\enstq{\Sum{\footnotesize\begin{array}{c}
						(\zeta,p)\geq_{lex}(\alpha,n)\\ \zeta\tq\kappa_{-\zeta}\succeq^KP_0(k_{P_0})
				\end{array}}{}\hspace{-2em} \ln_p\kappa_{-\zeta}-\Sum{i=0}{\pinf}\ln|P_0(i)|}{P_0\in\Pcal_0(x)}
		\end{calculs}
		\item If $(\beta,m) = (\alpha,n)$, then
		\centre{$H^{(\beta,m)} = \enstq{-\ln|P_0(x)|}{P_0\in\Pcal_0(x)}$}
	\end{itemize}
	Let $b\in\Union{q=0}{\pinf}\supp\Phi^q(x)$. Then, there are $\eta\in H^{(\beta,m)}$ and $y\in \inner {E^{(\beta,m)}}$ such that
	\centre{$\omega^b\asymp\f{\partial \ln_m\kappa_{-\beta}}{\partial u}\exp(\eta + y)$}
\end{corollary}

\begin{proof}
	Since $\Phi$ is strongly linear\index{Strongly linear function}, we just need to apply Proposition \ref{prop:supportPhi2} to each term of $x$. Each path of $\Pcal_0(x)$ involved is shifted one rank. In $H^{(\beta,m)}$ we the add $\ln |P_0(0)|$ compare to Proposition \ref{prop:supportPhi2}. Then $\exp(\eta)$ gives also $|\partial u\exp\epsilon|$. We just remove it so that it does not appear twice. 
\end{proof}

\begin{Prop}
	\label{prop:bonOrderPhi2}
	Let $x$ be a surreal number such that there $u$ of the form $u=\ln_n\kappa_{-\alpha}$, $ r\in\Rbb^*$ and $a_0\in\Nobf$ such that $r\omega^{a_0}\succ \ln u$ and
	\centre{$\forall a\in\supp x\quad \exists\epsilon\sim r\omega^{a_0}\qquad \omega^a\asymp\partial u\exp\epsilon$}
	\lc{Let}{$\Pcal_0(x)=\enstq{P\in\Pcal_\Lbb(x)}{\begin{array}{c}
				P(1)=r\omega^{a_0} \\ \forall i\geq 1 \quad P(i+1)\sim\ln|P(i)
		\end{array}}$}
	Consider $E_1^{(\beta,m)}$, $E_2^{(\beta,m)}$
	and $E_3^{(\beta,m)}$ as defined in Corollary \ref{cor:propsupportPhi2}.
	Let $\xi$ be the smallest ordinal such that $\kappa_{-\xi}\prec^K P(k_P)$ for all path $P\in\Pcal_\Lbb(x)$.
	Let $\lambda$ the least $\epsilon$-number greater than $\NR(x)$ and $\xi$. Then $E^{(\beta,m)}=E_1^{(\beta,m)}\cup E_2^{(\beta,m)}\cup E_3^{(\beta,m)}$ is reverse well-ordered with order type at most $2\lambda+\omega(\xi+1)$.
\end{Prop}

\begin{proof}
	First notice that $E_3^{(\beta,m)}$ is reverse well-ordered with order type at most $\omega$. $E_2^{(\beta,m)}$ is also reverse well-ordered with order at most $\omega\otimes\xi$. We then focus on $E_1^{(\beta,m)}$. For the moment, we will assume $(\beta,m)<_{lex}(\alpha,n)$.
	
	\begin{enumerate}[label=(\roman*)]
		\item We first claim that for all $i\geq 3$ and all path $P\in\Pcal(x)$, $P(i)\prec P(2)\preceq \ln_{m+2} \kappa_{-\beta}$. It is indeed the same proof as the point (i) of the proof of Proposition \ref{prop:bonOrderPhi1}. 
		
		\item We claim that for all path $P\in\Pcal(x)$, if $P(2)\asymp\ln_{m+2}\kappa_{-\beta}$, then, if $r$ is the real number such that 
		$P(2)\sim r\ln_{m+2}\kappa_{-\beta}$, we have $0<r\leq1$. It is indeed the same proof as the point (ii) of the proof of Proposition \ref{prop:bonOrderPhi1}.
		
		\item For all $j$ and $i\geq 2$, $\ln |P_j(i)|\preceq\ln_{m+3} \kappa_{-\beta}\prec \ln_{m+2}\kappa_{-\beta}$. Indeed, using (i), we know that $P_j(i)\preceq \ln_{m+2} \kappa_{-\beta}$. Then, there is a natural number $m\geq 1$ such that $|P_j(i)|\leq m\ln_{m+2}\kappa_{-\beta}$. Using the fact that $\ln$ is increasing,
		\centre{$\ln|P_j(i)| \leq \ln_{m+3}\kappa_{-\beta}+\ln m\preceq \ln_{m+3}\kappa_{-\beta}\prec \ln_{m+2}\kappa_{-\beta}$} 
		
		\item We now claim that $\Unionin k\Nbb E_{1,k,k'}^{(\beta,m)}>\Unionin k\Nbb E_{1,k,k'+2}^{(\beta,m)}$. Indeed, using (ii) and (iii) if $e_1\in \Unionin k\Nbb E_{1,k,k'}^{(\beta,m)}$, then there is $s\in\intff{-(k+1)}{-k}$ such that 
		$e_1\sim s\ln_{m+2} \kappa_{-\beta}$. Similarly, for $e_2\in\Unionin k\Nbb E_{1,k,k'+2}^{(\beta,m)}$, there is \linebreak $s'\in\intff{-(k+3)}{-(k+2)}$ such that $e_2\sim s'\ln_{m+2}\kappa_{-\beta}$. 
		
		\item We define the following sequence : 
		\begin{itemize}
			\item $a_0=1$
			\item $a_{k+1} = \omega^{\omega^{\omega(\omega(\NR(x)+\xi+1)a_k+1)}}$
		\end{itemize}
		We show that $E_{1,k,0}^{(\beta,m)}$ is reverse well-ordered with order type less than $a_k$. We also claim that the equivalence classes of $E_{1,k,0}^{(\beta,m)}/{\asymp}$ are finite and that 
		\centre{$\NR\pa{\Sum{t\in E_{1,k,0}^{(\beta,m)}}{}\exp t}\leq \omega (\NR(x)+\xi+1)a_k$}
		We show it by induction on $k\in\Nbb$.
		\begin{itemize}
			\item For $k=0$, $E_{1,0,0}^{(\beta,m)}=\{0\}$. Then it is reverse well-ordered with order type $1$. We also have
			\centre{$\NR\pa{\Sum{t\in E_{1,0,0}^{(\beta,m)}}{}\exp t}=\NR(1)=1\leq \omega (\NR(x)+\xi+1)$}
			
			\item Assume the property for some $k\in\Nbb$. Let $t\in E_{1,k+1,0}^{(\beta,m)}$. Let \linebreak $(P_0,0),(P_1,i_1),\dots,(P_{k+1},i_{k+1})$ minimal for the order $(<_{lex},<)_{lex}$ such that 
			\centre{$t=e^{(\beta,m)}\left(\begin{matrix}
					P_0;P_1,\dots,P_{k+1}\\
					\emptyset\\
					i_1,\dots,i_{k+1}
				\end{matrix}\right)$}
			Then,
			\begin{calculs}
				& t&=&e^{(\beta,m)}\left(\begin{matrix}
					P_0;P_1\dots,P_k\\ \emptyset\\i_1,\dots,i_k
				\end{matrix}\right)
				-\Sum{i=1}{\pinf}\ln|P_0(i)|\\
				&&&\qquad-\Sum{\footnotesize\begin{array}{c}
						\gamma\geq\alpha, m\in\Nbb^*\\ \gamma\tq P_0(k_{P_0})\succ^K\kappa_{-\gamma}\succeq^KP_{k+1}(k_{P_{k+1}})
				\end{array}}{}\hspace{-5em} \ln_m\kappa_{-\gamma}
				+ \Sum{i=i_{k+1}}{\pinf}\ln |P_{k+1}(i)|
			\end{calculs}
			Write $s=e^{(\beta,m)}\left(\begin{matrix}
				P_0;P_1,\dots,P_k\\ \emptyset\\i_1,\dots,i_k
			\end{matrix}\right)$ and consider the following path: 
			\centre{$\begin{accolade}
					R(0)=\exp s \\
					R(i)= P_{k+1}(i-1+i_{k+1}) & i>0
				\end{accolade}$}
			It is indeed a path since, by definition of $E_{1,k+1,0}^{(\beta,m)}$, $\supp P_{k+1}(i_{k+1})$ must be contained in $\supp s$. We then have,
			\centre{$\exp t = \f{\partial R}{\partial P_0[1:]}$}
			Moreover, $R\in\Pcal_\Lbb\pa{\Sum{s\in E_{1,k,0}^{(\beta,m)}}{}\exp s}$. By assumption on $x$, the set $\enstq{P_0[1:]}{P_0\in\Pcal_0(x)}$ is a singleton. Therefore, so is the set $\enstq{\partial P_0[1:]}{P_0\in\Pcal_0(x)}$. By induction hypothesis and Proposition \ref{prop:majorationNuPartial}, $E_{1,k+1,0}^{(\beta,m)}$ has order type less than 
			\centre{$\omega^{\omega^{\omega(\omega(\NR(x)+\xi+1)a_k+1)}}=a_{k+1}$}	
			Since the equivalences classes of $\Pcal_\Lbb\pa{\Sum{s\in E_{1,k,0}^{(\beta,m)}}{}\exp s}/{\asymp}$ are finite, the 
			ones of $E_{1,k+1,0}^{(\beta,m)}/{\asymp}$ are also finite. Finally, using Lemmas \ref{lem:NRsum} and \ref{lem:NRSommeLogAtomiques},
			\begin{calculs}
				& \NR(t) &\leq& (\omega\otimes\xi) + 
				\Sum{i=1}{k_{P_0}-1} \NR(\ln|P_0(i)|) + k_{P_0} \\ &&& \qquad +\Sum{j=1}{k+1}\Sum{i=i_j}{k_{P_j}-1}\NR\pa{\ln|P_j(i)|}+\Sum{j=0}{k+1}\max(0,k_{P_j}-i_j)+4\\
				&&\leq& \omega(\NR(x)+\xi+1)
			\end{calculs}
			\lc{Then,}{$\NR\pa{\Sum{t'\in E_{1,k+1}^{(\beta,m)}}{}\exp t'}\leq \omega(\NR(x)+\xi+1)a_{k+1}$}
		\end{itemize}
		We conclude thanks to the induction principle.
		
		\item We have $\Unionin k\Nbb E_{1,k,0}^{(\beta,m)}\subseteq\inner {E_{1,1,0}^{(\beta,m)}}$. Then, using (v) and applying Proposition \ref{prop:orderTypeMonoid}, it has order type at most $\omega^{\hat{a_1}}\leq\omega^{\omega a_1}$.
		
		\item We define the following sequence : 
		\begin{itemize}
			\item $b_0=\omega^{\hat{a_1}}$
			\item $b_{k'+1} = \omega^{\omega^{\omega(\omega(\NR(x)+\xi+4)b_{k'}+1)}}$
		\end{itemize}
		We show that $\Unionin k\Nbb E_{1,k,k'}^{(\beta,m)}$ is reverse well-ordered with order type less than $b_{k'}$. We also claim that the equivalence classes of $\Unionin k\Nbb E_{1,k,k'}^{(\beta,m)}/{\asymp}$ are finite and that 
		\centre{$\NR\pa{\Sum{t\in \Unionin k\Nbb E_{1,k,k'}^{(\beta,m)}}{}\exp t}\leq \omega (\NR(x)+\xi+4)b_{k'}$}
		We show it by induction on $k'\in\Nbb$.
		\begin{itemize}
			\item For $k'=0$, we just apply (vi).
			
			\item Assume the property for some $k'\in\Nbb$. Let $t\in \Unionin k\Nbb E_{1,k,k'+1}^{(\beta,m)}$. Let $(P_0,0)(P_1,i_1),\dots,(P_{k+k'+1},i_{k+k'+1})$ minimal for the order \linebreak $(<_{lex},<)_{lex}$ such that $t=e^{(\beta,m)}\left(\begin{matrix}
				P_0;P_1,\dots,P_k\\ P_{k+1},\dots, P_{k+k'+1}\\ i_1,\dots,i_{k+k'+1}
			\end{matrix}\right)$. Then,
			\begin{calculs}
				&t &=& e^{(\beta,m)}\left(\begin{matrix}
					P_0;P_1,\dots,P_k\\ P_{k+1},\dots, P_{k+k'}\\ i_1,\dots,i_{k+k'}
				\end{matrix}\right)
				-\Sum{\ell=m+2}{\pinf}\ln_\ell{\kappa_{-\beta}}\\
				&&&\qquad -\Sum{\footnotesize\begin{array}{c}
						\gamma>\beta, \ell\in\Nbb^*\\ \gamma\tq\kappa_{-\gamma}\succeq^KP_{k+k'+1}(k_{P_{k+k'+1}})
				\end{array}}{} \ln_m\kappa_{-\beta}
				+ \Sum{i=i_{k+1}}{\pinf}\ln |P_{k+1}(i)|
			\end{calculs}
			Write $s=e^{(\beta,m)}\left(\begin{matrix}
				P_0;P_1,\dots,P_k\\ P_{k+1},\dots, P_{k+k'}\\ i_1,\dots,i_{k+k'}
			\end{matrix}\right)$. We then have,
			\begin{calculs}
				&\partial(\ln_{m+1}\kappa_{-\beta})\exp t &=& 
				\exp(s)\exp\pa{
					\Sum{i=i_{k+1}}{\pinf}\ln |P_{k+1}(i)|\right.\\&&&\left.\qquad
					-\Sum{\footnotesize\begin{array}{c}
							\ell\in\Nbb^*\\ \gamma\tq\kappa_{-\gamma}\succeq^KP_{k+k'+1}(k_{P_{k+k'+1}})
					\end{array}}{}\hspace{-3em} \ln_m\kappa_{-\beta}}
			\end{calculs}
			Consider the following path:
			\centre{$\begin{accolade}
					R(0)=\exp s \\
					R(i)= P_{k+k'+1}(i-1+i_{k+1}) & i>0
				\end{accolade}$}
			It is indeed a path since, by definition of $E_{1,k,k'+1}^{(\beta,m)}$, $\supp P_{k+k'+1}(i_{k+k'+1})$ must be contained in $\supp s$. Then,
			\centre{$\partial(\ln_{m+1}\kappa_{-\beta})\exp t=\partial R$} 
			Moreover, $R\in\Pcal_\Lbb\pa{\Sum{s\in \Unionin k\Nbb E_{1,k,k'}^{(\beta,m)}}{}\exp s}$. By induction hypothesis and Proposition \ref{prop:majorationNuPartial}, $\Unionin k\Nbb E_{1,k,k'+1}^{(\beta,m)}$ has order type less than 
			\centre{$\omega^{\omega^{\omega(\omega(\NR(x)+\xi+4)b_{k'}+1)}}=b_{k'+1}$}	
			Since the equivalences classes of $\Pcal_\Lbb\pa{\Sum{s\in \Unionin k\Nbb E_{1,k,k'}^{(\beta,m)}}{}\exp s}/{\asymp}$ are finite, the 
			ones of $\Unionin k\Nbb E_{1,k,k'+1}^{(\beta,m)}/{\asymp}$ are also finite. Finally, using Lemmas \ref{lem:NRsum} and \ref{lem:NRSommeLogAtomiques},
			\begin{calculs}
				& \NR(t) &\leq& (\omega\oplus\omega\otimes\xi\oplus\omega) + 
				\Sum{j=0}{k+1}\Sum{i=i_j}{k_{P_j}-1}\NR\pa{\ln|P_j(i)|}+\Sum{j=0}{k+1}\max(0,k_{P_j}-i_j)\\
				&&\leq& \omega(\NR(x)+\xi+4)
			\end{calculs}
			\lc{Then,}{$\NR\pa{\Sum{t'\in \Unionin k\Nbb E_{1,k,k'+1}^{(\beta,m)}}{}\exp t'}\leq \omega(\NR(x)+\xi+4)b_{k'+1}$}
		\end{itemize}
		We conclude thanks to the induction principle.
		
		\item By easy induction, for all $k\in\Nbb$, $b_{k'}<\lambda$.
		
		\item Using (iv), we get that for all $N\in\Nbb$, $\Union{k'=0}{N}\Unionin k\Nbb E_{1,,k,2k'}^{(\beta,m)}$ is an initial segment of
		$\Unionin{k'}\Nbb\Unionin k\Nbb E_{1,k,2k'}^{(\beta,m)}$. We also have that $\Union{k'=0}{N}\Unionin k\Nbb E_{1,k,2k'+1}^{(\beta,m)}$ is an initial segment of
		$\Unionin {k'}\Nbb\Unionin k\Nbb E_{1,k,2k'+1}^{(\beta,m)}$. Using (vii), we get that $\Unionin{k'}\Nbb \Unionin k\Nbb E_{1,k,2k'}^{(\beta,m)}$ has order type at most
		\centre{$\sup\enstq{\Oplus{k=0}{N}b_{2k'}}{N\in\Nbb} = \sup\enstq{b_{2N}}{N\in\Nbb}\underset{\text{by (viii)}}{\leq}\lambda$}
		Similarly,  $\Unionin{k'}\Nbb\Union{k\in\Nbb}{}E_{1,k,2k'+1}^{(\beta,m)}$ has order type at most $\lambda$. Using Proposition \ref{prop:unionEnsBienOrd}, we conclude that $E_1^{(\beta,m)}$ has order type at most~$2\lambda$.
	\end{enumerate}
	Now we deal with the case $(\beta,m)=(\alpha,n)$. A close looking at point (v) above reveals that the property it shows does not depend on $(\beta,m)$. Then we have, using a similar argument as in points (viii) and (ix), that $\Unionin k\Nbb E_{1,k,0}^{(\alpha,n)}$ has order type at most~$2\lambda$. Then, for any $(\beta,m)\leq_{lex}(\alpha,n)$, $E_1^{(\beta,m)}$ is reverse well-ordered with order type at most $2\lambda$. Using again Proposition \ref{prop:unionEnsBienOrd} and the properties of $E_2^{(\beta,m)}$ and $E_3^{(\beta,m)}$ mentioned in the beginning of this proof, we get that $E^{(\beta,m)}$ is reverse well-ordered with order type at most $2\lambda+\omega(\xi+1)$.
\end{proof}

\subsection{Length of the series of the anti-derivative of an arbitrary surreal number}

\propsupportPhi*

\begin{proof}
	Let $\alpha<\gamma$ and $n\in\Nbb$. Write $x=\Sumin a{\supp x}r_a\omega^a$. For any ordinal $\alpha<\gamma$, $n\in\Nbb$, $r\in\Rbb\setminus\{-1\}$ and any term $s\omega^{a_0}$, define $S_{\alpha,n,1,s\omega^{a_0}}$ to be the set
	\centre{$\enstq{a\in\supp x}{\exists\epsilon\in\Nobf_\infty\ \forall (\beta,m)<_{lex}(\alpha,n)\quad
	\begin{accolade}
		\ln_n\kappa_{-\alpha} \prec\epsilon\\
		\epsilon\sim s\omega^{a_0}\\
		\epsilon\ln_m\kappa_{-\beta}\\
		\omega^a\asymp\partial (\ln_n\kappa_{-\alpha})\exp\epsilon
	\end{accolade}
	}$}
	and consider also
	\centre{$S_{\alpha,n,2,r}=\enstq{a\in\supp x}{\exists\epsilon\in\Nobf_\infty\quad \epsilon\sim r\ln_n\kappa_{-\alpha}\wedge \omega^a\asymp\partial (\ln_n\kappa_{-\alpha})\exp\epsilon}$}
	\centre{$x_{\alpha,n,1,s\omega^{a_0}}=\Sumin a{S_{\alpha,n,1,s\omega^{a_0}}}r_a\omega^a\qqandqq x_{\alpha,n,2,r}=\Sumin a{S_{\alpha,n,2,r}}r_a\omega^a$}
	All theses surreal numbers have disjoint supports and 
	$$x=\Sumin{s\omega^{a_0}}{\Rbb\omega^\Nobf}\Sumlt\alpha\gamma\Sumin n\Nbb x_{\alpha,n,1,s\omega^{a_0}} + \Sumin r{\Rbb\setminus\{-1\}} \Sumlt\alpha\gamma\Sumin n\Nbb x_{\alpha,n,2,r}$$
	We then study both sums of the above equality.
	
	\begin{itemize}
		\item The set $\enstq{r\in\Rbb\setminus\{-1\}}{S_{\alpha,n,2,r}\neq\emptyset}$ is reverse well-ordered with order type at most $\nu(x)$. Let
		\centre{$S_0=\Unionin i\Nbb \supp\Phi^i\pa{ \Sumin r{\Rbb\setminus\{-1\}} \Sumlt\alpha\gamma\Sumin n\Nbb x_{\alpha,n,2,r}}$}
		Since $\Phi$ is strongly linear, 
		\centre{$S_0\subseteq\Unionin r{\Rbb\setminus\{-1\}}\Unionlt\alpha\gamma\Unionin n\Nbb\Unionin i\Nbb \supp\Phi^i(x_{\alpha,n,2,r})$}
		Using Corollary \ref{cor:supportPhi1}, $\Unionin i\Nbb \supp\Phi^i(x_{\alpha,n,2,r})$ is reverse well-ordered with order type at most $\omega^{\omega(2\lambda+\omega(\gamma+1)+1)}$. Moreover, Lemma \ref{lem:formeEtaPhiOmegaA} ensure that is $(\alpha,n,r)>_{lex}(\alpha',n',r')$, then $\Unionin i\Nbb \supp\Phi^i(x_{\alpha,n,2,r})<\Unionin i\Nbb \supp\Phi^i(x_{\alpha',n',2,r'})$. We end up with the fact that $\Unionin i\Nbb \supp\Phi^i\pa{ \Sumin r{\Rbb\setminus\{-1\}} \Sumlt\alpha\gamma\Sumin n\Nbb x_{\alpha,n,2,r}}$ is reverse well-ordered with order type at most $\omega^{\omega(2\lambda+\omega(\gamma+1)+1)}\nu(x)\gamma$.
		
		\item \lc{Let}{$S_1=\Unionin i\Nbb \supp\Phi^i\pa{ \Sumin{s\omega^{a_0}}{\Rbb\omega^\Nobf}\Sumlt\alpha\gamma\Sumin n\Nbb x_{\alpha,n,1,s\omega^{a_0}} }$}
		Since $\Phi$ is strongly linear,
		\centre{$S_1\subseteq\Unionin {s\omega^{a_0}}{\Rbb\omega^\Nobf}\Unionlt\alpha\gamma\Unionin n\Nbb\Unionin i\Nbb \supp\Phi^i(x_{\alpha,n,1,s\omega^{a_0}})$}
		Denote $H^{(\beta,m)}(x_{\alpha,n,1,s\omega^{a_0}})$, $E^{(\beta,m)}(x_{\alpha,n,1,s\omega^{a_0}})$ the sets defined as in Corollary \ref{cor:propsupportPhi2} for $x_{\alpha,n,1,s\omega^{a_0}}$. Then, using this corollary, $S_1$ is contained in the set
		\centre{$
			\Union{\tiny \begin{array}{c}
					\beta<\gamma \\ m\in\Nbb \\ s\omega^{a_0}\in \Rbb\omega^\Nobf
			\end{array}}{}
			\Union{\tiny \begin{array}{c}
					\alpha,n\tq (\beta,m)\leq_{lex}(\alpha,n)\\\alpha<\gamma
			\end{array}}{}
			\Union{\tiny \begin{array}{c}
					\eta\in H^{(\beta,m)}(x_{\alpha,n,1,s\omega^{a_0}})\\
					y\in \inner{E^{(\beta,m)}(x_{\alpha,n,1,s\omega^{a_0}})}
			\end{array}}{}\hspace{-3em}\supp\pa{ \f{\partial\ln_m\kappa_{-\beta}}{\partial\ln_n\kappa_{-\alpha}}\exp(\eta+y)}
			$}
		We also know that $(\alpha,n,s\omega^{a_0})>_{lex}(\alpha',n',s'\omega^{a_0'})$, then the set
		\centre{$\Union{\tiny \begin{array}{c}
					\eta\in H^{(\beta,m)}(x_{\alpha,n,1,s\omega^{a_0}})\\
					y\in \inner{E^{(\beta,m)}(x_{\alpha,n,1,s\omega^{a_0}})}
			\end{array}}{}\hspace{-3em}\supp\pa{ \f{\partial\ln_m\kappa_{-\beta}}{\partial\ln_n\kappa_{-\alpha}}\exp(\eta+y)}$}
		is contained in the set
		\centre{$\Union{\tiny \begin{array}{c}
					\eta\in H^{(\beta,m)}\pa{x_{\alpha',n',1,s'\omega^{a_0'}}}\\
					y\in \inner{E^{(\beta,m)}\pa{x_{\alpha',n',1,s'\omega^{a_0'}}}}
			\end{array}}{}\hspace{-3em}\supp\pa{ \f{\partial\ln_m\kappa_{-\beta}}{\partial\ln_n\kappa_{-\alpha}}\exp(\eta+y)}$}
		Propositions \ref{prop:bonOrderPhi2} and \ref{prop:orderTypeMonoid} guarantee that all of theses sets are reverse well-ordered with order type less than~$\omega^{2\omega\lambda}$. Let
		\centre{$S_{\beta,m} = \Union{\tiny \begin{array}{c}
					\alpha,n\tq (\beta,m)\leq_{lex}(\alpha,n)\\\alpha<\gamma\\ s\omega^{a_0}\in \Rbb\omega^\Nobf
			\end{array}}{}\Union{\tiny \begin{array}{c}
					\eta\in H^{(\beta,m)}(x_{\alpha,n,1,s\omega^{a_0}})\\
					y\in \inner{E^{(\beta,m)}(x_{\alpha,n,1,s\omega^{a_0}})}
			\end{array}}{}\hspace{-3em}\supp\pa{ \f{\partial\ln_m\kappa_{-\beta}}{\partial\ln_n\kappa_{-\alpha}}\exp(\eta+y)}$}
		The set of possible $s\omega^{a_0}$ is reverse well-ordered with order type at most $\nu(x)$. Moreover, $\alpha$ and $n$ are determined from $s\omega^{a_0}$. Then $S_{\beta,m}$ is reverse well-ordered with order type at most $\omega^{2\omega^{\lambda+1}}\nu(x)$. Finally, if $(\beta,m)>_{lex}(\beta',m')$, then $S_{\beta,m}\subsetneq S_{\beta',m'}$ and there are at most $\omega\gamma$ such couples. Then, $S_1$ is reverse well-ordered with order type at most $\omega^{2\omega\lambda+1}\nu(x)\gamma$.
	\end{itemize}
	
	Both sets have order type less than $\omega^{\omega^{\lambda+2}}$, which is a multiplicative ordinal. Using Proposition \ref{prop:sommeEnsBienOrd}, $\Unionin i\Nbb\supp\Phi^i(x)$ has order type less than $\omega^{\omega^{\lambda+2}}$.

\end{proof}

\subsection{Stability of some surreal fields}

We are ready to exhibit a surreal field that is stable under $\exp, \ln, \partial$ and anti-derivation and that is not $\Nobf$ itself. We actually have a lot of such fields.

To get a field stable under derivation and anti-derivation, it sufficient that the field is stable under $\exp$ and $\ln$, and, if for all $\beta<\alpha$, $\kappa_{-\beta}$ is in the field, and if $\omega\otimes\alpha$ is less than the authorized length of a serie, then $\kappa_{-\alpha}$ must also be in the field. 

%
%
%

\thmcorpsStable*

\begin{proof}
	Let $\Kbb=\Unionlt\beta\alpha\SRF{\epsilon_\beta}{\Gamma_\beta^{\uparrow\epsilon_\beta}}$
	As an increasing union of fields, $\Kbb$ is indeed a field.
	
	\begin{enumerate}[label=(\roman*)]
		\item Using Theorem \ref{thm:SRFGammaUpStableExpLn}, each field $\SRF{\epsilon_\beta}{\Gamma_\beta^{\uparrow\epsilon_\beta}}$ is stable under $\exp$ and $\ln$, then so is $\Kbb$.
		
		\item Write $\Gamma_\beta^{\uparrow\epsilon_\beta}=\suitelt{\Gamma_{i,\beta}}i{\gamma_{\epsilon_{\beta}}}$. We use the notation introduce in the beginning Definition \ref{def:uparrow}. We prove by induction on $i<\gamma_{\epsilon_\beta}$ that for $x\in\Gamma_{i,\beta}$, $\NR(\omega^x)<\eta_\beta e_i$.
		\begin{itemize}
			\item For $i=0$ we have $e_0=1$ and $\Gamma_{0,\beta}=\Gamma_\beta$. By assumption on $\Gamma_\beta$, for all $x\in\Gamma_{0,\beta}$, $\NR(\omega^x)<\eta_\beta=\eta_\beta e_0$.
			
			\item Assume the property for some ordinal $i$. Then let $x\in\Gamma_{i+1,\beta}$. Write $x=u+v+\Sum{k=1}{p}h(w_k)$ with $u\in\Gamma_{i,\beta}$, $v\in\SRF{e_i}{g\pa{(\Gamma_{i,\beta})_+^*}}$ and $w_k$ such $r\omega^{w_k}$ is a term of some element $y_k\in\Gamma_{i,\beta}$, for some $r\in\Rbb$. Using Corollary \ref{cor:NRprod},
			\centre{$\NR(\omega^x)\leq\NR(\omega^u)+\NR(\omega^v)+\Sum{k=1}{p}\NR(\omega^{h(w_k)})+p+1$}
			From induction hypothesis,
			\centre{$\NR(\omega^u)<\eta_\beta e_i$}
			Write $v=\Sum{j<\nu}{}r_j\omega^{g(a_j)}$. Then $\omega^v=\exp\pa{\Sum{j<\nu}{}r_j\omega^{a_j}}$. From induction hypothesis, $\NR(\omega^{a_j})<\eta_\beta e_i$. Then $\NR(r_j\omega^{a_j})<\eta_\beta e_i+1$. Then
			\centre{$\NR(\omega^v)=\NR\pa{\Sum{j<\nu}{}r_j\omega^{a_j}}\leq(\eta_\beta e_i+ 1)\otimes\nu\leq(\eta_\beta e_i +1)\otimes e_i\leq\eta_\beta e_i^2$}
			We also have
			\centre{$\NR(\omega^{h(w_k)})=\NR(\omega^{\omega^{w_k}})\leq\NR(\omega^{y_k})<\eta_\beta e_i$}
			\lc{Finally,}{$\NR(\omega^x)\leq(p+1)(\eta_\beta e_i+1)+\eta_\beta e_i^2<\eta_\beta e_{i+1}$}
			
			\item Assume $i$ is a limit ordinal. Then by definition of $\Gamma_{i,\beta}$ for any $x\in\Gamma_{i,\beta}$ there is some $j<i$ such that $x\in\Gamma_{j,\beta}$. Then induction hypothesis concludes.
		\end{itemize}
		\item Let $x\in\Kbb$ and $\beta<\alpha$ such that $x\in\SRF{\epsilon_\beta}{\Gamma_\beta^{\uparrow\epsilon_\beta}}$. Using (ii), there is $i<\gamma_{\epsilon_\beta}$ such that \centre{$\NR(x)\leq(\eta_\beta e_i+1)\otimes\nu(x)<\eta_\beta\otimes\epsilon_\beta=\epsilon_\beta$}
		Since $\eta_\beta\otimes\epsilon_\beta$ is a limit ordinal, then we also have $\NR(x)+1<\eta_\beta\otimes \epsilon_\beta=\epsilon_\beta$.
		
		\item Let $x\in\Kbb$ and $\beta<\alpha$ such that $x\in\SRF{\epsilon_\beta}{\Gamma_\beta^{\uparrow\epsilon_\beta}}$. Using (iii) and Proposition \ref{prop:majorationNuPartial}, $\nu(\partial x)<\omega^{\omega^{\omega(\NR(x)+1)}}<\epsilon_\beta$. Using Corollary \ref{cor:LogAtomicDeGammaUp} and (i), we also have for all $P\in\Pcal(x)$, $\partial P\in\Rbb\omega^{\Gamma_\beta^{\uparrow\epsilon_\beta}}$. Then, 
		\centre{$\partial x\in\SRF{\epsilon_\beta}{\Gamma_\beta^{\uparrow\epsilon_\beta}}\subseteq\Kbb$}
		Then $\Kbb$ is stable under $\partial$.
		
		\item Let $x\in\Kbb$ and $\beta<\alpha$ such that $x\in\SRF{\epsilon_\beta}{\Gamma_\beta^{\uparrow\epsilon_\beta}}$. Using Proposition \ref{prop:supportPhi} and the definition of $\Acal$, 
		\centre{$\nu\pa{\Acal\circ\pa{\Sumin i\Nbb\Phi^i}(x)}<\omega^{\omega^{\lambda+2}}$}
		where $\lambda$ is least $\epsilon$-number greater $\NR(x)$ and such that
		\centre{$\forall P\in\Pcal_\Lbb(x)\qquad \kappa_{-\lambda}\prec^K P(k_P)$}
		Using (iii), $\NR(x)<\epsilon_\beta$. Let $P\in\Pcal_\Lbb(x)$. Using (i), $\SRF{\epsilon_\beta}{\Gamma_\beta^{\uparrow\epsilon_\beta}}$ is stable under $\exp$ and $ln$. Since $P(i+1)$ is a term of $\ln|P(i)|$, if $P(i)\in\omega^{\Gamma_\beta}$, then $P(i+1)\in\omega^{\Gamma_\beta}$. By induction, $P(k_P)\in\omega^{\Gamma_\beta}$. Since $P(k_P)$ is infinitely large, $P(k_P)\in\omega^{\pa{\Gamma_\beta}^*_+}$. By assumption on $\Gamma_\beta$, $P(k_P)\succ^K\kappa_{-\epsilon_\beta}$. Finally, $\lambda\leq\epsilon_\beta$ and
		\centre{$\nu\pa{\Acal\circ\pa{\Sumin i\Nbb\Phi^i}(x)}<\omega^{\omega^{\epsilon_\beta+2}}<\epsilon_{\beta+1}$}
		Propositions \ref{prop:supportPhi1} and \ref{prop:supportPhi2} and the third assumption about $\Gamma_\beta$ ensure that each term of $\Acal\circ\pa{\Sumin i\Nbb\Phi^i}(x)$ is in $\omega^{\Gamma_\beta}\subseteq\omega^{\Gamma_{\beta+1}}$. Then
		\centre{$\Acal\circ\pa{\Sumin i\Nbb\Phi^i}(x)\in\SRF{\epsilon_{\beta+1}}{\Gamma_{\beta+1}^{\uparrow\epsilon_{\beta+1}}}$}
		Application of Corollary \ref{cor:existencePrimitive} gives that $\Kbb$ is stable under anti-derivation.
	\end{enumerate}
\end{proof}

The previous theorem may seem have a lot of strong hypothesis but we can actually give a non-trivial application.

	Take $\alpha=\omega$ and for $n<\omega$, $\Gamma_n=\enstq{x\in\Nolt{\epsilon_n}}{\NR(\omega^x)<\epsilon_{n-1}}$, with $\epsilon_{-1}:=\omega$.
	We first recall that from Lemma \ref{lem:kappaMoinsAlpha}, for any ordinal $\alpha$, 
	\centre{$\kappa_{-\alpha}=\omega^{\omega^{-\omega\otimes\alpha}}$}
	\lc{in particular}{$\kappa_{-\epsilon_n} = \omega^{\omega^{-\omega\otimes\epsilon_n}} = \omega^{\omega^{-\epsilon_n}} = \omega^{\f1{\epsilon_n}}$}
	From Theorem \ref{thm:serieToSignExp}, we know that the sign sequence of $\omega^{-\omega\otimes\alpha}$ is $(+)(-)^{\omega\otimes\alpha}$, which has length $1\oplus\omega\otimes\alpha$.
	\begin{itemize}
		\item Since $\epsilon_n$ is an $\epsilon$-number, hence an additive ordinal, for any $n\in\Nbb$, $\Gamma_n$ is an abelian group.
		\item Of course for any $n\leq m$, $\Gamma_n\subseteq\Gamma_m$.	
		\item Since $\length{\omega^{-\epsilon_n}}=1\oplus\omega\otimes\epsilon_n=\epsilon_n$, we have $\kappa_{-\epsilon_n}\notin\omega^{\Gamma_\beta}$. However, for $\alpha<\epsilon_n$, $\length{\omega^{-\omega\otimes\alpha}}=1\oplus\omega\otimes\alpha<\epsilon_n$, and, since $\kappa_{-\alpha}\in\Lbb$, $\NR(\kappa_{-\alpha})=0<\epsilon_{n-1}$. Therefor,  we have $\kappa_{-\alpha}\in\omega^{\Gamma_\beta}$. Since, $\kappa_{-\epsilon_\beta}\in\omega^\Nobf$, is $x\in\omega^{\Gamma_\beta}$ is such that $x\preceq^K\kappa_{-\epsilon_\beta}$ then $\kappa_{-\epsilon_\beta}$ is a prefix of $x$ and $\length x\geq\epsilon_\beta$ what is impossible from Lemma \ref{lem:lengthExpLog}.
		\item We can take $\eta_\beta=\epsilon_{\beta-1}<\epsilon_\eta$.
	\end{itemize}
	Theorem \ref{thm:corpsStable} applies and $\Unionin n\Nbb \SRF{\epsilon_n}{\Gamma_n^{\uparrow\epsilon_n}}$ is stable under $\exp$, $\ln$, $\partial$ and anti-derivation. As a final note, we can notice that
	\centre{$\Unionin n\Nbb \SRF{\epsilon_n}{\Gamma_n^{\uparrow\epsilon_n}} = \Unionin n\Nbb \SRF{\epsilon_n}{\Nolt{\epsilon_n}^{\uparrow\epsilon_n}}$}

\bibliographystyle{alpha}
\bibliography{bournez,perso,biblio}

\end{document}